\documentclass[11pt,leqno]{amsart}
\usepackage{a4wide}
\usepackage{amssymb,upgreek}
\usepackage{amsmath}
\usepackage{amsthm}
\usepackage{mathrsfs}
\usepackage{bm,euscript,csquotes,comment}
\usepackage[curve]{xypic}
\usepackage[usenames,dvipsnames]{xcolor}
\usepackage{enotez}
\let\footnote\endnote

\usepackage{hyperref}

\hypersetup{colorlinks,  citecolor=ForestGreen,  linkcolor=Mahogany, urlcolor=black}

\theoremstyle{plain}

\newcommand{\id}{\operatorname{id}}
\newcommand{\im}{\operatorname{im}}

\newcommand{\pr}{\operatorname{pr}}
\newcommand{\ev}{\operatorname{ev}}
\newcommand{\Sh}{\operatorname{Sh}}
\newcommand{\res}{\operatorname{res}}
\newcommand{\Res}{\operatorname{Res}}
\newcommand{\Hom}{\operatorname{Hom}}
\newcommand{\End}{\operatorname{End}}
\newcommand{\Ext}{\operatorname{Ext}}

\newcommand{\Mod}{\operatorname{Mod}}

\newcommand{\ind}{\operatorname{ind}}

\newcommand{\Tor}{\operatorname{Tor}}
\newcommand{\Spec}{\operatorname{Spec}}

\newcommand{\chara}{\operatorname{char}}
\newcommand{\cores}{\operatorname{cores}}
\newcommand{\val}{\operatorname{val}}

\def\fhom{\mathfrak h}
\def\ften{\mathfrak t}
\def\X{{\mathbf X}}

\def\anti{\EuScript J}
\def\trace{\EuScript S}

\def\ima{\check\alpha([\mathbb F_q^\times])}
\def\kera{\mu_{\check\alpha}}

\newtheorem{theorem}{Theorem}[section]
\newtheorem{corollary}[theorem]{Corollary}
\newtheorem{lemma}[theorem]{Lemma}
\newtheorem{proposition}[theorem]{Proposition}

\theoremstyle{definition}
\newtheorem{remark}[theorem]{Remark}

\title{\textbf{The modular pro-$p$ Iwahori-Hecke $\Ext$-algebra}}
\author{Rachel Ollivier, Peter Schneider}
\date{July 23, 2018}

\address{University of British Columbia, Mathematics Department, 1984 Mathematics Road, Vancouver, BC V6T 1Z2, Canada}
\email{ollivier@math.ubc.ca}
\urladdr{http://www.math.ubc.ca/~ollivier}

\address{ Universit\"at M\"unster,  Mathematisches Institut,  Einsteinstr. 62, 48291 M\"unster, Germany}
\email{pschnei@uni-muenster.de}
\urladdr{http://www.uni-muenster.de/math/u/schneider/}

\begin{document}

\maketitle

\rightline{\it Dedicated to J.\ Bernstein on the occasion of his 72nd birthday}

\begin{abstract}
Let $\mathfrak F$ be a locally compact nonarchimedean field of positive residue
characteristic $p$ and $k$ a field of characteristic $p$. Let    $G$ be the group of $\mathfrak{F}$-rational points
of a connected reductive group  over $\mathfrak{F}$ which we suppose $\mathfrak F$-split. Given a pro-$p$ Iwahori subgroup $I$ of $G$, we consider the space $\X$ of $k$-valued functions with compact
support on $G/I$.
It is naturally an object   in the category ${\Mod(G)}$ of all smooth $k$-representations of $G$.

We study the graded Ext-algebra  $E^*=\Ext_{\Mod(G)}^*(\mathbf X, \mathbf X)$.
Its degree zero piece $E^0$ is the usual pro-$p$ Iwahori Hecke algebra $H$.
We describe  the product  in $E^*$ and provide an involutive anti-automorphism of $E^*$. When $I$ is a Poincar\'e group of dimension $d$, the $\Ext$-algebra $E^*$  is
supported in degrees $i\in\{0\dots d\}$ and we establish  a duality theorem between
$E^i$ and $E^{d-i}$. Under the same hypothesis (and assuming that $\mathbf G$ is
almost simple and simply connected), we compute
$E^d$ as an $H$-module on the left and on the
right. We prove that it is a direct sum of the
trivial character, and of supersingular modules.

\end{abstract}

\tableofcontents

\section{Introduction}\label{sec:Intro}

Let $\mathfrak F$ be a locally compact nonarchimedean field with residue
characteristic $p$,   and  let    $G$ be the group of $\mathfrak{F}$-rational points
of a connected reductive group  $\mathbf G$ over $\mathfrak{F}$. We
suppose that $\mathbf G$ is $\mathfrak F$-split in this article.

Let $k$ be a field and let  $\Mod(G)$ denote the category of all smooth
representations of $G$ in $k$-vector spaces. When $k=\mathbb C$,
by a theorem of Bernstein  \cite{Ber} Cor.\ 3.9.ii, one of the blocks of $\Mod(G)$ is equivalent to the category of modules over the Iwahori-Hecke algebra of $G$. This block is the subcategory of all representations which are generated by their Iwahori-fixed vectors. It does not contain any supercuspidal representation of $G$.

When $k$ has characteristic $p$, it is natural to consider the Hecke algebra $H$ of the
pro-$p$ Iwahori subgroup $I \subset G$. In this case the natural left exact functor
\begin{align*}
 \fhom: \Mod(G) & \longrightarrow \Mod(H) \\
              V & \longmapsto V^I = \Hom_{k[G]}(\mathbf{X},V) \
\end{align*}
sends a nonzero representation onto a nonzero module. Its left adjoint is
\begin{align*}
\ften:\Mod(H) & \longrightarrow  \Mod^I(G)\subseteq \Mod(G) \\
                M & \longmapsto \mathbf{X} \otimes_H M \ .
\end{align*}
Here $\mathbf X$ denotes the space of $k$-valued functions with compact
support on $G/I$ with the natural left action of $G$.
The functor $\ften$ has values in the category  $\Mod^I(G)$ of all smooth
$k$-representations of $G$ generated by their $I$-fixed vectors.  This category,
which a priori has no reason to be an abelian subcategory of $\Mod(G)$, contains all
irreducible representations including the supercuspidal ones. But in general $\fhom$ and $\ften$ are not quasi-inverse equivalences of categories
and little is known about $\Mod^I(G)$ and $\Mod(G)$ unless $G={\rm GL}_2(\mathbb Q_p)$ or $G={\rm SL}_2(\mathbb Q_p)$ (\cite{Koz1}, \cite{Ollequiv}, \cite{OS2}, \cite{Pas}).

From now on we assume $k$ has characteristic $p$. The functor $\fhom$, although left exact, is not right exact since $p$ divides the pro-order of $I$. It is therefore natural to consider the derived functor. In \cite{SDGA} the following result is shown: When
$\mathfrak F$ is a finite extension of $\mathbb Q_p$ and $I$ is a torsionfree
pro-$p$ group, there exists a derived version of the functor $\fhom$ and $\ften $ providing an equivalence between the derived category of smooth representations of $G$ in $k$-vector spaces and the derived category of differential graded modules over a certain differential graded pro-$p$ Iwahori-Hecke algebra $H^\bullet$.

The current article is largely motivated by this theorem. The derived categories  involved are not understood,
and in fact the  Hecke differential graded algebra $H^\bullet$
itself has no concrete description yet.
We provide here our first results on the
structure of its cohomology algebra $\Ext_{\Mod(G)}^*(\mathbf X, \mathbf X)$:

\begin{itemize}

\item[-] We describe explicitly the product in  $\Ext_{\Mod(G)}^*(\mathbf X, \mathbf
X)$ (Proposition \ref{prop:techn-formula}).

\item[-] We deduce the existence of an
involutive anti-automorphism of $\Ext_{\Mod(G)}^*(\mathbf X, \mathbf X)$ as a  graded
$\Ext$-algebra (Proposition \ref{prop:anti+product}).
\item[-] When $I$ is a Poincar\'e group of dimension $d$, the $\Ext$ algebra is
supported in degrees $0$ to  $d$ and we establish  a duality theorem between its
$i^{\rm th}$ and $d-i^{\rm th}$ pieces (Proposition \ref{prop:bimodule}).

\item[-] Under the same hypothesis (and assuming that $\mathbf G$ is
almost simple and simply connected), we compute
$\Ext_{\Mod(G)}^d(\mathbf X, \mathbf X)$ as an $H$-module on the left and on the
right (Corollary  \ref{coro:supersingEd}). We prove that it is a direct sum of the
trivial character, and of supersingular modules.
\end{itemize}

We hope  that these results illustrate that the $\Ext$-algebra $\Ext_{\Mod(G)}^*(\mathbf X, \mathbf X)$
is a natural object whose structure is  rich and interesting in itself, even beyond its link to the representation theory of $p$-adic reductive groups. As a derived version of the Hecke algebra of the $I$-equivariant functions on an (almost) affine flag variety, we suspect that it will contribute to relating the mod $p$ Langlands program to methods that appear in the study of geometric Langlands.

Both authors thank the PIMS at UBC Vancouver for support and for providing a very stimulating atmosphere during the \emph{Focus Period on Representations in Arithmetic}. The first author is partially funded by NSERC Discovery Grant.

\section{Notations and preliminaries}\label{sec:Not}

Throughout  the paper we fix a locally compact nonarchimedean field $\mathfrak{F}$ (for now of any characteristic) with ring of integers $\mathfrak{O}$, its maximal ideal $\mathfrak{M}$,  and a prime element $\pi$. The residue field $\mathfrak{O}/\pi \mathfrak{O}$ of $\mathfrak{F}$ is $\mathbb{F}_q$ for some power $q = p^f$ of the residue characteristic $p$. We choose the valuation $\val_{\mathfrak{F}}$ on   $\mathfrak{F}$  normalized by $\val_{\mathfrak{F}}(\pi)=1$ We let $G := \mathbf{G}(\mathfrak{F})$ be the group of $\mathfrak{F}$-rational points of a connected reductive group $\mathbf{G}$  over $\mathfrak{F}$ which we always assume to be $\mathfrak{F}$-split.

We fix an $\mathfrak{F}$-split maximal torus $\mathbf{T}$ in $\mathbf{G}$, put $T := \mathbf{T}(\mathfrak{F})$, and let $T^0$ denote the maximal compact subgroup of $T$ and $T^1$ the pro-$p$ Sylow subgroup of $T^0$. We also fix a chamber $C$ in the apartment of the semisimple Bruhat-Tits building  $\mathscr X$ of $G$ which corresponds to $\mathbf{T}$. The stabilizer $\EuScript P_C^\dagger$ of $C$ contains an Iwahori subgroup $J$. Its pro-$p$ Sylow subgroup $I$ is called the pro-$p$ Iwahori subgroup and is the main player in this paper. We have $T\cap J= T^0$ and $T\cap I= T^1$. If $N(T)$ is the normalizer of $T$ in $G$, then we define the group $\widetilde{W} := N(T)/T^1$. In particular, it contains $T^0/T^1$. The quotient $W:=N(T)/T^0\cong\widetilde{W}/(T^0/T^1)$ is the extended affine Weyl group. The finite Weyl group is $W_0:= N(T)/T$. For any compact open subset $A \subseteq G$ we let $\mathrm{char}_A$ denote the characteristic function of $A$.

The coefficient field for all representations in this paper is an arbitrary field $k$ of characteristic $p > 0$. For any open subgroup $U \subseteq G$ we let $\Mod(U)$ denote the abelian category of smooth representations of $U$ in $k$-vector spaces. As usual, $K(U)$ denotes the homotopy category of unbounded (cohomological) complexes in $\Mod(U)$ and $D(U)$ the corresponding derived category.

\subsection{Elements of Bruhat-Tits theory}

We consider  the root data associated to the choice of the maximal $\mathfrak F$-split torus $\mathbf T$ and record in this section the notations and properties we will need.
We follow the exposition of \cite{OS1} \S4.1--\S4.4 which refers mainly to \cite{SchSt} I.1. Further references are given in \cite{OS1}.

\subsubsection{\label{subsubsec:apartment}}

The root datum $(\Phi , X^*({T }), {\check\Phi}, X_*({T }))$ is reduced because the group $\mathbf{G}$ is $\mathfrak F$-split.  Recall that $X^*({T })$ and $X_*({T })$ denote respectively the group of algebraic characters and cocharacters of $T $. Similarly, let $X^*({ Z})$ and $X_*({ Z})$ denote respectively the group of algebraic characters and cocharacters
of the connected center ${Z}$ of $G$. The standard apartment $\mathscr A$ attached to $T$ in the semisimple building $\mathscr X$ of $G$ is denoted by  $\mathscr A$. It can be seen as the vector space
\begin{equation*}
    \mathbb R\otimes _{\mathbb Z}(X_*({T})/X_*({\rm Z}))
\end{equation*}
considered as an affine space on itself. We fix a hyperspecial vertex of the chamber $C$ and, for simplicity, choose it to be the zero point in $\mathscr A$. 
Denote by
$
    \langle  \, .\,,.\, \rangle :X_*({T })\times X^*({T })\rightarrow \mathbb Z
$
the natural perfect pairing, as well as its $\mathbb R$-linear extension. Each root $\alpha \in \Phi$ defines a function $x\mapsto  \alpha(x) $ on  $\mathscr A$. For any subset $Y$ of $\mathscr{A}$, we write $\alpha(Y)\geq 0$ if $\alpha$ takes nonnegative values on $Y$. To $\alpha$ is also associated a coroot $\check\alpha\in{\check\Phi}$ such that $\langle \check\alpha, \alpha   \rangle =2$ and a reflection on  $\mathscr{A}$ defined by
\begin{equation*}
    s_\alpha: x\mapsto x- \alpha( x)\check\alpha \mod X_*({\rm Z})\otimes_{\mathbb Z}\mathbb R\ .
\end{equation*}
The subgroup of the transformations of $\mathscr{A}$  generated by these reflections identifies with the finite Weyl group $W_0$. The finite Weyl group $W_0$ acts by conjugation on $T$ and this induces a faithful linear action on $\mathscr A$.
Thus $W_0$ identifes with a subgroup of the transformations of $\mathscr{A}$ and this subgroup is the one  generated by the reflections     $s_\alpha$ for all $\alpha\in\Phi$.
To an element $g\in T $ corresponds a vector $\nu(g)\in \mathbb R\otimes _{\mathbb Z}X_*({T })$ defined by
\begin{equation*}
    \langle \nu(g),\, \chi\rangle  =-\val_{\mathfrak F}(\chi(g))  \qquad \textrm{for any } \chi\in X^*(T ).
\end{equation*}
 The quotient of $T $ by  $ \ker(\nu)= T^0$ is a free abelian group $\Lambda$ with rank equal to ${\rm dim}(T )$, and $\nu$ induces an isomorphism $\Lambda \cong X_*(T )$. The group $\Lambda$ acts by translation on $\mathscr{A}$ via $\nu$.
The actions of $W_0$ and $\Lambda$ combine into an action of ${W}$ on $\mathscr{A}$.
The extended affine Weyl group $W $  is  the semi-direct product ${W }_0\ltimes \Lambda$ if we identify $W _0$ with the subgroup of $W$ that fixes any lift of $x_0$ in the extended building of $G$.

\subsubsection{Affine roots and root subgroups}

We now recall the definition of the affine roots and the properties of the  associated root subgroups.
To a root $\alpha$ is attached a unipotent subgroup ${\EuScript U}_\alpha$ of $G$ such that for any $u\in {\EuScript U}_\alpha\setminus\{1\}$, the intersection ${\EuScript U}_{-\alpha}u {\EuScript U}_{-\alpha}\cap N(T )$ consists in only one element called $m_\alpha(u)$.  The image in $W $ of this element $m_\alpha(u)$ is the reflection at the affine hyperplane $\{x\in \mathscr{A}, \: \alpha(x)=-\mathfrak h_\alpha(u)\}$  for a certain $\mathfrak h_\alpha(u) \in \mathbb{R}$. Denote by $\Gamma_\alpha$ the discrete unbounded subset  of $\mathbb R$ given by $\{\mathfrak h_\alpha(u), \: u\in {\EuScript U}_\alpha\setminus\{1\}\}$. Since our group $\mathbf{G}$ is $\mathfrak F$-split we have
$\{\mathfrak h_\alpha(u), \: u\in {\EuScript U}_\alpha\setminus\{1\}\}=\mathbb Z.$
The affine functions
\begin{equation*}
    (\alpha, \mathfrak h):=\alpha(\,.\,)+\mathfrak h  \qquad\text{for $\alpha\in\Phi$ and $\mathfrak h\in \mathbb Z$}
\end{equation*}
are called the affine roots.

We identify an element $\alpha$ of $\Phi $ with the affine root $(\alpha,0)$ so that the set of affine roots $\Phi _{aff}$ contains $\Phi $. The action of $W _0$  on $\Phi $ extends  to an action of $W $ on $\Phi _{aff}$. Explicitly, if $w=w_0  t_\lambda \in W $ is the composition of the translation by $\lambda\in  \Lambda$ with $w_0\in W _0$, then the action of $w$ on the affine root $(\alpha,\mathfrak h)$
\begin{equation*}
    (w_0(\alpha),  \mathfrak h + (\val_{\mathfrak F} \circ \alpha)(\lambda) ) = (w_0(\alpha),  \mathfrak h - \langle \nu(\lambda), \alpha\rangle)\ .
\end{equation*}


Define a filtration of ${\EuScript U}_\alpha$, $\alpha\in \Phi $ by
\begin{equation*}
    {\EuScript U}_{\alpha , r}:=\{u\in {\EuScript U}_\alpha\setminus\{1\}, \: \mathfrak h_\alpha(u)\geq r\}\cup\{1\}\textrm{  for } r\in \mathbb R \ .
\end{equation*} For $(\alpha, \mathfrak h)\in \Phi_{aff}$, we put ${\EuScript U}_{(\alpha, \mathfrak h)} := {\EuScript U}_{\alpha, \mathfrak h}$.  Obviously, for $r\in\mathbb R$ a real number,  $\mathfrak h\geq r$ is equivalent to ${\EuScript U}_{(\alpha, \mathfrak h)} \subseteq {\EuScript U}_{\alpha, r}$.


By abuse of notation we write throughout the paper $w Mw^{-1}$, for some $w \in W $ and some subgroup $M \subseteq G$, whenever the result of this conjugation is independent of the choice of a representative of $w$ in $N(T )$. For example, for $(\alpha, \mathfrak h)\in \Phi_{aff}$ and $w\in W $, we have \begin{equation}\label{f:conjU}w {\EuScript U}_{(\alpha, \mathfrak h)}w^{-1}= {\EuScript U}_{w(\alpha, \mathfrak h)}\ .\end{equation}

For any non empty subset $\mathbf Y\subset \mathscr{A}$, define
\begin{align*}
    f_{\mathbf Y} : \Phi  & \longrightarrow \mathbb R\cup \{\infty\} \\ \alpha & \longmapsto  -\inf_{x\in \mathbf Y} \alpha(x) \ .
\end{align*}
and the subgroup of $G$
\begin{equation}
\label{defiU}
{\EuScript U}_{\mathbf Y}= \; < {\EuScript U}_{\alpha, f_{\mathbf Y}(\alpha)}, \:\: \alpha\in \Phi >
\end{equation}
generated by  all ${\EuScript U}_{\alpha, f_{\mathbf Y}(\alpha)}$ for $ \alpha\in \Phi $.
We have (\cite{SchSt} Page 103, point 3.)
\begin{equation}
  {\EuScript U}_{\mathbf Y} \cap {\EuScript U}_\alpha= \;  {\EuScript U}_{\alpha, f_{\mathbf Y}(\alpha)} \: \textrm{ for any }\alpha\in\Phi \ .
\end{equation}

\subsubsection{Positive roots and length \label{sec:posiroots}}

The choice of the chamber $C$  determines the   subset $\Phi ^+$ of  the positive  roots, namely the set of $\alpha\in\Phi$ taking nonnegative values on $C$.  Denote by $\Pi $ a basis for $\Phi ^+$.
Likewise, the set of positive affine roots ${\Phi _{aff}^+}$ is defined to be the set of  affine roots taking nonnegative values on  $C$. The set of negative affine roots is $\Phi^-_{aff}:= -\Phi^+_{aff}$.

\begin{lemma}\phantomsection\label{lemma:UC}
\begin{itemize}
\item[i.] We have
$\EuScript{U}_{\alpha, f_C(\alpha)} =
    \EuScript{U}_{\alpha,0}$  for $\ \alpha \in \Phi^+$ and     $\EuScript{U}_{\alpha, f_C(\alpha)} =\EuScript{U}_{\alpha,1}$ for $\ \alpha \in \Phi^-$.

\item[ii.] The pro-$p$ Iwahori subgroup $I$ is generated by $T^1$ and $\EuScript U_C$, namely by $T^1$ and
 all root subgroups $\EuScript U_A$ for $A\in \Phi_{aff}^+$.
 \item[iii.] For $\alpha\in \Phi$, we have $I\cap \EuScript{U}_\alpha= {\EuScript U}_{\alpha, f_{\mathbf C}(\alpha)}$.
\end{itemize}
\end{lemma}
\begin{proof}
i., ii. This is given by \cite{SchSt} Prop. I.2.2 (recalled in \cite{OS1} Proof of Lemma 4.8) and
\cite{OS1} Proof of Lemma 4.2. iii. This also follows from \cite{SchSt} Prop.\ I.2.2.
\end{proof}

The finite Weyl group $W _0$ is a Coxeter system generated by  the set $S := \{s_\alpha : \alpha \in \Pi\}$ of reflections associated to the simple roots $\Pi$. It is endowed with a length function denoted by $\ell$.
This length extends to $W $ in such a way that the length of an element $w\in W $ is the cardinality of $\{ A\in\Phi ^+_{aff},\: w(A)\in {\Phi _{aff}^-}\}$. For any affine root $(\alpha, \mathfrak h)$, we have in $W $ the reflection $    s_{(\alpha, \mathfrak h)} $ at the affine hyperplane $\alpha(\, .\,) = - \mathfrak{h}$. The affine Weyl group is defined as the subgroup $W_{aff}$ of $W$ generated by all
$ s_A$ for all  $ A \in \Phi_{aff}$.

There is  a partial order on $\Phi $ given by $\alpha\leq \beta$ if and only if $\beta -\alpha$ is a linear combination with (integral) nonnegative coefficients of elements in $\Pi$. Let
$\Phi^{min} := \{\alpha \in \Phi : \alpha\ \textrm{is minimal for $\leq$}\}$
and
$\Pi_{aff} := \Pi \cup \{(\alpha,1) : \alpha \in \Phi^{min}\} \subseteq \Phi^+_{aff}$. Let  $  S_{aff} := \{s_A : A \in \Pi_{aff}\}$, then
the pair $(W_{aff}, S_{aff})$ is a Coxeter system and the length function $\ell$ restricted to $W _{aff}$ coincides with the length function of this Coxeter system. For any $s\in S_{aff}$ there is a unique positive affine root $A_s\in \Phi^+_{aff}$ such that $sA_s\in \Phi^-_{aff}$. In fact $A_s$ lies in $\Pi_{aff}$.

We have the following formula, for every $A \in \Pi_{aff}$ and $w\in W$ (\cite{Lu} \S 1):
\begin{equation}\label{add}
   \ell(w s_A)=
   \begin{cases}
       \ell(w)+1 & \textrm{ if }w (A)\in {\Phi_{aff}^+},\\  \ell(w)-1 & \textrm{ if }w (A)\in {\Phi_{aff}^-}.
    \end{cases}
\end{equation}

The Bruhat-Tits decomposition of $G$ says that $G$ is the disjoint union of the double cosets $JwJ$ for $w \in {W}$.
As in \cite{OS1} \S 4.3 (see the references there) we will denote by $\Omega$ the abelian subgroup of $ W$ of all elements with length zero and recall that $\Omega$ normalizes $S_{aff}$. Furthermore, $W$ is the semi-direct product $W=\Omega\ltimes W_{aff}$. The length function is constant on the double cosets $\Omega w \Omega$ for $w\in W$.
The stabilizer $\EuScript P_C^\dagger$ of $C$ in $G$
is the disjoint union of the double cosets $J\omega J$ for $\omega \in {\Omega}$ (\cite{OS1} Lemma 4.9). We have
\begin{equation*}
  I\subseteq J \subseteq  \EuScript P^\dagger_C
\end{equation*}
and $I$, the pro-unipotent radical of the parahoric subgroup $J$, is normal in $\EuScript P_C^\dagger$ (see \cite{OS1} \S 4.5). In fact, $J$ is also normal in $\EuScript P_C^\dagger$ since $J$ is generated by $T^0$ and $I$, and since the action by conjugation of $\omega\in \Omega$ on $T$ normalizes $T^0$.

\subsubsection{\label{pro-p Weyl}}

Recall that we denote by $\widetilde W$ the quotient of $N(T)$   by $T^1$ and obtain the exact sequence
\begin{equation*}
  0 \rightarrow T^0/T^1 \rightarrow \widetilde W \rightarrow W \rightarrow 0 \ .
\end{equation*}
The length function $\ell$ on $W$ pulls back to a length function $\ell$ on $\widetilde W$ (\cite{Vigprop} Prop.\ 1). The Bruhat-Tits decomposition of $G$ says that $G$ is the disjoint union of the double cosets $IwI$ for $w \in \widetilde{W}$. We will denote by $\widetilde \Omega$ the preimage of $\Omega$ in $\widetilde W$. It contains $T^0/T^1$.

Obviously $\widetilde{W}$ acts by conjugation on its normal subgroup $T^0/T^1$, and we denote this action simply by $(w,t)\mapsto w(t)$. But, since $T/T^1$ is abelian the action in fact factorizes through $W_0 = \widetilde{W}/T$. On the other hand $W_0$, by definition, acts on $T$ and this action stabilizes the maximal compact subgroup $T^0$ and its pro-$p$ Sylow $T^1$. Therefore we again have an action of $W_0$ on $T^0/T^1$. These two actions, of course, coincide.

\subsubsection{On certain open compact subgroups of the pro-$p$ Iwahori subgroup}

Let $g\in G$. We let \begin{equation} I_g:= I\cap g I g^{-1}\ .\end{equation} Since this definition depends only on $gJ$, we may consider $I_w:= I\cap w I w^{-1}$ for any $w\in W$ or $w\in \widetilde W$. Since $I$ is normal in $\EuScript P_C^\dagger$, we have $I_{w\omega}= I_w$ for any $\omega \in \Omega$ and any $w\in W$ (but in general not $I_{\omega w}= I_w$).

Note that if $\tilde w\in\widetilde W$ lifts $w\in W$, then $I_{\tilde w}= I_w$.

\begin{lemma} \label{lemma:vw}  Let $v,w\in  W$  such that $\ell(vw)=\ell(v)+\ell(w)$. We have
\begin{equation}\label{f:length2}
 I_{vw} \subseteq I_v
\end{equation} and
\begin{equation}\label{f:length1}
 I\subseteq  I_{v^{-1}} w I w^{-1}\ .
\end{equation}
\end{lemma}

\begin{proof} \eqref{f:length2}: The claim is clear when $w$ has length $0$ since we then have $I_{vw}= I_v$. By induction, it suffices to treat the case when $w=s\in S_{aff}$.
Note first that adjoining $ vsC$ to a given minimal gallery of $\mathscr{X}$  between $C$ and $vC$ gives a minimal gallery between $C$ and $vsC$.

Let $y\in I_{vs}$. The apartment $y\mathscr{A}$ contains $ vsC$ and $C$ and therefore it contains any minimal gallery from $C$ to $vsC$ by \cite{BT1} 2.3.6. In particular, it contains $ vC$. Let $F$ be the facet with codimension 1 which is contained in both the closures of  the chambers $ vC$ and $ vsC$. It is fixed by $y$ and therefore the image $C'$ of the chamber $  vC$  under the action of $y$  also contains $F$ in its closure. By \cite{BT1} 1.3.6, only two chambers of $y\mathscr{A}$  contain $F$ in their closure, so $C'= vC$.  Therefore, $y \in    v \EuScript P_{C}^\dagger v^{-1}  \cap I =\EuScript P_{vC}^\dagger  \cap I $ where $ \EuScript P_{vC}^\dagger$ denotes the stabilizer in $G$ of the chamber $v C$. By \cite{OS1} Lemma 4.10, the intersection $\EuScript P_{vC}^\dagger  \cap I $ is contained in the parahoric subgroup $ v J v^{-1}$ of $\EuScript P_{vC}^\dagger $, and since it is a pro-$p$ group, it is contained in its  pro-$p$ Sylow subgroup $v I v^{-1}$. We have proved that $y$ lies in $I_v$.

\eqref{f:length1}: Let $w\in W$. We prove the following statement by induction on $\ell(v)$: let $v\in W$ such that $\ell(vw)=\ell(v)+\ell(w)$; then any $A\in \Phi_{aff}^+$ such that $v A\in \Phi^-_{aff}$ satisfies $w^{-1} A\in \Phi_{aff}^+$. Using \eqref{f:conjU} and Lemma \ref{lemma:UC}.ii, this means that  for $v\in W$ such that $\ell(vw)=\ell(v)+\ell(w)$  and $A\in \Phi_{aff}^+$  we have $v A\in \Phi^+_{aff}$ and
${\EuScript U}_A\subseteq v^{-1} {\EuScript U}_C v^{-1}\cap I \subseteq I_{v^{-1}}$, or $v A\in \Phi^-_{aff}$  and $w^{-1} A\in \Phi_{aff}^+$ so ${\EuScript U}_A\subseteq w {\EuScript U}_C w^{-1} \subseteq  w I w^{-1}$.  This implies that for $v\in W$ such that $\ell(vw)=\ell(v)+\ell(w)$ we have ${\EuScript U}_C\subseteq I_{v^{-1}} w Iw^{-1}$ and again using Lemma \ref{lemma:UC}.ii, that  $I\subseteq I_{v^{-1}} w Iw^{-1}$. We now proceed to the proof of the claim by induction.

When $v$ has length zero the claim is clear. Now  let $v$ such that $\ell(vw)=\ell(v)+\ell(w)$ and $s\in S_{aff}$ such that $\ell(sv)=\ell(v)-1$.  This implies that $\ell(svw)\leq \ell(sv)+\ell(w)=\ell(vw)-1$ and therefore $\ell(svw)=\ell(vw)-1=\ell(sv)+\ell(w)$. In particular we have $(vw)^{-1} A_s\in \Phi^-_{aff}$ with the notation introduced in \S\ref{sec:posiroots}
 (see \cite{Lu} Section 1, recalled in \cite{OS1} (4.2)). By induction hypothesis,   any $A\in \Phi_{aff}^+$ such that $ sv A\in \Phi^-_{aff}$ satisfies $w^{-1} A\in \Phi_{aff}^+$.
Now let $A\in \Phi_{aff}^+$ such that $v A\in \Phi^{-}_{aff}$. We need to show that $w^{-1} A\in \Phi^{+}_{aff}$. If  $sv A\in \Phi^-_{aff}$ then it follows from  the induction hypothesis. Otherwise, it means that
$ vA= -A_s$   and therefore $w^{-1} A= -(vw)^{-1} A_s\in \Phi^+_{aff}$.
\end{proof}

\begin{lemma} \label{lemma:CwC}Let $w\in W$.
The product map induces a   bijection
\begin{equation}\label{f:prodIw}
     \prod_{\alpha\in \Phi^-} {\EuScript U}_{\alpha, f_{C\cup wC}(\alpha)} \times  T^1 \times \prod_{\alpha\in \Phi^+} {\EuScript U}_{\alpha, f_{C\cup wC}(\alpha)}\overset{\sim}\longrightarrow I_w
\end{equation}
where the products on the left hand side are ordered in some arbitrarily chosen way.
\end{lemma}

\begin{proof}
The multiplication in $G$ induces an injective map
\begin{equation*}
   \prod_{\alpha\in \Phi^-} {\EuScript U}_{\alpha} \times  T \times \prod_{\alpha\in \Phi^+} {\EuScript U}_{\alpha} \hookrightarrow G \ .
\end{equation*}
In the notation of \cite{SchSt} \S I.2 we have $I = R_C$ and $wIw^{-1} = R_{wC}$. Therefore \cite{SchSt} Prop.\ I.2.2 says that the above map restricts to bijections
\begin{equation}\label{f:prodI}
     \prod_{\alpha\in \Phi^-} {\EuScript U}_{\alpha, f_{C}(\alpha)}\times  T^1 \times \prod_{\alpha\in \Phi^+} {\EuScript U}_{\alpha, f_{C}(\alpha)}\overset{\sim}\longrightarrow I
\end{equation}
and
\begin{equation*}
     \prod_{\alpha\in \Phi^-} {\EuScript U}_{\alpha, f_{wC}(\alpha)} \times  T^1 \times \prod_{\alpha\in \Phi^+} {\EuScript U}_{\alpha, f_{wC}(\alpha)}\overset{\sim}\longrightarrow wIw^{-1} \ ,
\end{equation*}
and hence to the bijection
\begin{equation*}
     \prod_{\alpha\in \Phi^-} {\EuScript U}_{\alpha, f_{C}(\alpha)} \cap {\EuScript U}_{\alpha, f_{wC}(\alpha)} \times  T^1 \times \prod_{\alpha\in \Phi^+} {\EuScript U}_{\alpha, f_{C}(\alpha)} \cap {\EuScript U}_{\alpha, f_{wC}(\alpha)}\overset{\sim}\longrightarrow I_w \ .
\end{equation*}
Since, obviously, $f_{C\cup wC}(\alpha) = \max (f_{C}(\alpha), f_{wC}(\alpha))$ we have $ {\EuScript U}_{\alpha, f_{C}(\alpha)} \cap {\EuScript U}_{\alpha, f_{wC}(\alpha)} = {\EuScript U}_{\alpha, f_{C\cup wC}(\alpha)}$.

\end{proof}

\begin{remark}\label{rema:fg}
Let $\alpha\in \Phi$  and $w\in W$. Define  $g_w(\alpha):={\rm min}\lbrace  m \in\mathbb Z, \: (\alpha, m) \in \Phi_{aff}^+\cap w \Phi_{aff}^+\rbrace$. We have $\EuScript U_{\alpha, g_{ w}(\alpha)} =
\EuScript U_{\alpha, f_{C\cup wC}(\alpha)}$. This is Lemma \ref{lemma:UC}.i when $w=1$.
\end{remark}

\begin{proof}
First note that $ \Phi_{aff}^+\cap w \Phi_{aff}^+$ is the set of affine roots
which are positive on $C\cup wC$.
Let $\alpha\in \Phi$. Since  $\alpha+g_w(\alpha)\geq 0$ on $C\cup wC$ we have  $f_{C\cup wC}(\alpha)\leq g_w(\alpha)$ and
$\EuScript U_{\alpha,  g_w(\alpha)}\subseteq
\EuScript U_{\alpha, f_{C\cup wC}(\alpha)  }$. Now let $u\in  \EuScript U_{\alpha,  f_{C\cup wC}(\alpha) }\setminus\{1\}$. It implies $\mathfrak h_\alpha(u)\geq f_{C\cup wC}(\alpha) $ so $\alpha+ \mathfrak h_\alpha(u)\geq 0$ on $C\cup wC$, therefore  $\mathfrak h_\alpha(u)\geq g_w(\alpha)$ so $u\in
\EuScript U_{\alpha,  g_w(\alpha)}$.
\end{proof}

\begin{corollary}\label{coro:known}
Let $v,w\in W$ and $s\in S_{aff}$ with respective lifts $\tilde v, \tilde w$ and $\tilde s$ in $\widetilde W$. We have:
\begin{itemize}
\item[i.] $\vert I/I_w\vert= q^{\ell(w)}$;
\item[ii.]  if $\ell(vw)=\ell(v)+\ell(w)$ then $I \tilde v I \cdot I \tilde w I= I \tilde v \tilde w I$;
\item[iii.] if $\ell(ws)=\ell(w)+1$ then $I_{ws}$ is a normal subgroup of $I_w$ of index $q$.
\end{itemize}
\end{corollary}
\begin{proof}
Points i. and ii. are well known. Compare i.\ with  \cite{IM} Prop.\ 3.2 and \S I.5
and ii.\ with \cite{IM} Prop.\ 2.8(i). For the convenience of the reader we add the arguments.

i. We obtain the result by induction on $\ell(w)$.  Suppose that $\ell(ws)=\ell(w)+1$. Again by \cite{Lu} Section 1 (recalled in \cite{OS1} (4.2))
we have $wA_s\in \Phi_{aff}^+$ and $\Phi_{aff}^+\cap ws \Phi_{aff}^+ = (\Phi_{aff}^+\cap w \Phi_{aff}^+) \setminus \{wA_s\}$. So if we let $ (\beta, m):= w A_s$, then using Remark \ref{rema:fg} we have
\begin{equation}\label{f:notbeta}
   \EuScript{U}_{\alpha, f_{C\cup wC}(\alpha)}=  \EuScript{U}_{\alpha, f_{C\cup wsC}(\alpha)} \quad\text{ for any $\alpha\in \Phi$, $\alpha\neq\beta$}
\end{equation}
and
\begin{equation}\label{f:beta}
   \EuScript{U}_{\beta, f_{C\cup wC}(\beta)}=  \EuScript{U}_{wA_s}= \EuScript{U}_{(\beta,m)} \quad\text{ and }\quad  \EuScript{U}_{\beta, f_{C\cup wsC}(\beta)}= \EuScript{U}_{(\beta, 1+m)} .
\end{equation}
Hence using Lemma \ref{lemma:CwC} we deduce that $I_{ws} \subseteq I_w$ and
\begin{equation}\label{f:IandU}
  I_w/I_{ws} \simeq \EuScript{U}_{\beta, f_{C\cup wC}(\beta)}/ \EuScript{U}_{\beta, f_{C\cup wsC}(\beta)}=  \EuScript{U}_{\beta, m}/ \EuScript{U}_{\beta,m+1} \ ,
\end{equation}
which has cardinality $q$ by \cite{Tit} 1.1.

ii.  It suffices to treat the case  $v=s\in S_{aff}$. The claim then follows by  induction on $\ell(v)$.
Using Lemma \ref{lemma:UC} we have $I \tilde s I=I \tilde s \EuScript U_{A_s}$ since $s A\in \Phi_{aff}^+$ for any $A\in\Phi_{aff}^+\setminus\{A_s\}$. Now $\ell(sw)=\ell(w)+1$ means that $w^{-1}A_s\in \Phi^+_{aff}$ therefore (again by Lemma \ref{lemma:UC}.ii) $I \tilde s I \tilde w I= I \tilde s \EuScript U_{A_s} \tilde w I=I \tilde s  \tilde w \EuScript U_{w^{-1}A_s} I = I \tilde s  \tilde w  I$.

iii. We first treat the case that $w=1$. By i. we only need to show that $I_s$ is normal in $I$. Let $F$ be the $1$-codimensional facet common to $C$ and $sC$. The pro-unipotent radical $I_F$ of the parahoric subgroup $\mathbf G_F^\circ(\mathfrak O)$ attached to $F$ is a normal subgroup of $\mathbf G_F^\circ(\mathfrak O)$ (\cite{SchSt} I.2).  It follows from \cite{SchSt} Prop.\ I.2.11 and its proof that $I_F$ is a normal subgroup of $I = I_C$. Obviously $I_s \subseteq I_F$. Hence it suffices to show equality. By \cite{SchSt} Prop. I.2.2 (see also \cite{OS1} Proof of Lemma 4.8), the product map induces a  bijection
\begin{gather*}
     \prod_{\alpha\in \Phi^-} \EuScript{U}_{\alpha, f^*_F(\alpha)} \times
     T^1 \times \prod_{\alpha\in \Phi^+} \EuScript{U}_{\alpha, f^*_F(\alpha)} \overset{\sim}\longrightarrow I_F
\end{gather*}
where $f_F^*(\alpha)= f_F(\alpha)$ if $\alpha\neq \alpha_0$ and
$f_F^*(\alpha_0)= f_F(\alpha_0)+1 = \epsilon +1$ with $(\alpha_0, \epsilon) = A_s$.  In view of Lemma \ref{lemma:CwC} it remains to show that $\EuScript{U}_{\alpha,f_{F}^*(\alpha)} = \EuScript{U}_{\alpha,f_{C\cup sC}(\alpha)}$ for any $\alpha \in \Phi$. But this is immediate from \eqref{f:notbeta} and \eqref{f:beta} applied with $w=1$.

Coming back to the general case we see that $wIw^{-1} \cap wsI(ws)^{-1}$ is normal in $wIw^{-1}$. Hence $I_w \cap I_{ws} = I \cap wIw^{-1} \cap wsI(ws)^{-1}$ is normal in $I_w$. But in the proof of i. we have seen that $I_{ws} \subseteq I_w$. The assertion about the index also follows from i.
\end{proof}

\subsubsection{\label{subsubsec:Chevalley}Chevalley basis and double cosets decompositions}

Let $\mathbf{G}_{x_0}$ denote the Bruhat-Tits group scheme over $\mathfrak{O}$ corresponding to the hyperspecial vertex $x_0$ (cf. \cite{Tit}). As part of a Chevalley basis we have (cf.\ \cite{BT2} 3.2), for any root $\alpha \in \Phi$, a homomorphism $\varphi_\alpha : {\rm SL}_2 \longrightarrow \mathbf{G}_{x_0}$ of $\mathfrak{O}$-group schemes which restricts to isomorphisms
\begin{equation*}
    \{ \begin{pmatrix}
    1 & \ast \\ 0 & 1
    \end{pmatrix}
    \}
    \xrightarrow{\;\cong\;} {\EuScript U}_\alpha
    \qquad\text{and}\qquad
    \{ \begin{pmatrix}
    1 & 0 \\ \ast & 1
    \end{pmatrix}
    \}
    \xrightarrow{\;\cong\;} {\EuScript U}_{-\alpha} \ .
\end{equation*}
Moreover, one has $\check\alpha (x) = \varphi_\alpha( \bigl( \begin{smallmatrix}
x & 0 \\ 0 & x^{-1}
\end{smallmatrix}
\bigr) )$. We let the subtorus $\mathbf{T}_{s_\alpha} \subseteq \mathbf{T}$ denote the image (in the sense of algebraic groups) of the cocharacter $\check\alpha$. We always view these as being defined over $\mathfrak{O}$ as subtori of $\mathbf{G}_{x_0}$. The group of $\mathbb{F}_q$-rational points $\mathbf{T}_{s_\alpha} (\mathbb{F}_q)$ can be viewed as a subgroup of $T ^0 / T ^1 \xrightarrow{\cong} \mathbf{T}(\mathbb{F}_q)$ (and is abstractly isomorphic to $\mathbb{F}_q^\times$). Given $z\in \mathbb F^\times_q$, we consider $[z]\in\mathfrak O^\times$ the Teichm\"uller representative (\cite{Se2} II.4 Prop. 8) and denote by $\check\alpha([_-])$ the composite morphism of groups
\begin{equation}\label{f:alphafq}
   \check\alpha([_-])\::\:\mathbb F_q^\times\xrightarrow{[_-]} \mathfrak O^\times \overset{\check\alpha} \longrightarrow \mathbf{T}(\mathfrak{O}) = T^0 \ .
\end{equation}
We will denote its kernel by $\kera$. Looking at the commutative diagram
\begin{equation*}
  \xymatrix{
    \mathfrak{O}^\times \ar[d]_{red} \ar[r]^-{\check{\alpha}} & \mathbf{T}_{s_\alpha}(\mathfrak{O}) \ar[d]_{red} \ar[r]^{\subseteq} & \mathbf{T}(\mathfrak{O}) \ar[d]_{red} \ar[r]^{\pr} & T^0/T^1 \ar[dl]^{red}_{\cong} \\
    \mathbb{F}_q^\times \ar[r]^-{\check{\alpha}} & \mathbf{T}_{s_\alpha}(\mathbb{F}_q) \ar[r]^{\subseteq} & \mathbf{T}(\mathbb{F}_q) &    }
\end{equation*}
of reduction maps and using that the Teichm\"uller map is a section of the reduction map $\mathfrak{O}^\times \twoheadrightarrow \mathbb{F}_q^\times$ we deduce that the reduction map induces isomorphisms
\begin{equation*}
  \ima \xrightarrow{\cong} \check{\alpha}(\mathbb{F}_q^\times)  \qquad\text{and}\qquad \kera \xrightarrow{\cong} \ker(\check{\alpha}_{|\mathbb{F}_q^\times}) \ .
\end{equation*}

\begin{remark}\label{not:bart}
In Propositions \ref{prop:explicitleftaction} and \ref{prop:Hdformulas} we will need to differentiate between an element $t\in\ima\subset T^0$ and its image in $T^0/T^1\subset \widetilde W$ which we will denote by $\bar t$.
\end{remark}

\begin{remark}\label{rema:kernelalpha}
  If $\varphi_\alpha$ is not injective then its kernel is $\{\left(\begin{smallmatrix}
a & 0  \\
0 & a
\end{smallmatrix}\right) : a = \pm 1\}$ (cf.\ \cite{Jan} II.1.3(7)). It follows that $\kera$ has cardinality $1$ or $2$.
\end{remark}

For the following two lemmas below, recall that the  notation for the action of $W$  on $T^0/T^1$ was introduced in \S\ref{pro-p Weyl}.

\begin{lemma}\label{lemma:Omega}
   Suppose that  $\mathbf{G}$ is semisimple simply connected, then:
\begin{itemize}
\item[i.]  $\check{\alpha}(\mathbb{F}_q^\times) = \mathbf T_{s_\alpha} (\mathbb F_q)$ for any $\alpha\in\Phi$;
\item[ii.] $T^0/T^1=\widetilde \Omega$ is generated by the union of its subgroups $\mathbf{T}_{s_\alpha} (\mathbb{F}_q)$ for $\alpha \in \Pi$;
\end{itemize}
\end{lemma}
\begin{proof} By our assumption that $\mathbf{G}$ is semisimple simply connected the set $\{\check{\alpha} : \alpha \in \Pi\}$ is a basis of the cocharacter group $X_*(T)$. This means that
\begin{equation*}
  \prod_{\alpha \in \Pi} \mathbb{G}_m \xrightarrow{\prod_\alpha \check{\alpha}} \mathbf{T}
\end{equation*}
is an isomorphism of algebraic tori. It follows that multiplication induces an isomorphism
\begin{equation*}
  \prod_{\alpha \in \Pi} \mathbf{T}_{s_\alpha}(\mathbb{F}_q) \xrightarrow{\; \cong \;} \mathbf{T}(\mathbb{F}_q) \ .
\end{equation*}
This implies ii. But, since any root is part of some basis of the root system, we also obtain that $\Phi \cap 2X_*(T) = \emptyset$. Hence $\varphi_\alpha$ and $\check{\alpha}$ are injective for any $\alpha \in \Phi$ (cf. \cite{Jan} II.1.3(7)). For ii. it therefore remains to notice that $\check{\alpha} : \mathbb{F}_q^\times \rightarrow \mathbf{T}_{s_\alpha}(\mathbb{F}_q)$ is a map between finite sets of the same cardinality.
\end{proof}
\begin{lemma}\label{lemma:normal}
   For any $t\in T^0/T^1$ and any $A = (\alpha,\mathfrak{h}) \in \Phi_{aff}$, we have
    $s_A(t) t^{-1} \in \ima$.
\end{lemma}
\begin{proof}
Recall from \S\ref{pro-p Weyl} that the action of $s_A$ on $t$ is inflated from the action of its image $s_\alpha \in W_0$. The action of $W_0$ on $X_*(T)$ is induced by its action on $T$, i.e., we have
\begin{equation*}
  (w(\xi))(x) = w(\xi(x))  \qquad\text{for any $w \in W_0$, $\xi \in X_*(T)$, and $x \in T$}.
\end{equation*}
On the other hand the action of $s_\alpha$ on $\xi\in X_*(T)$  is given by
\begin{equation*}
  s_\alpha(\xi)= \xi-\langle \xi, \alpha \rangle \check \alpha \ .
\end{equation*}
So for any $\xi \in X_*(T)$ and any $y \in \mathbb{F}_q^\times$ we have
\begin{equation*}
  s_\alpha(\xi([y]))= (s_\alpha(\xi))([y]) = \xi([y]) \check\alpha([y])^{-\langle \xi, \alpha \rangle} \in \xi([y]) \check\alpha([\mathbb{F}_q^\times]) \ .
\end{equation*}
It remains to notice that, if $\xi_1, \ldots, \xi_m$ is a basis of $X_*(T)$, then $\prod_i \xi([\mathbb{F}_q^\times]) = T^0/T^1$.
\end{proof}

For any $\alpha \in \Phi$  we have the additive isomorphism $x_\alpha : \mathfrak{F} \xrightarrow{\cong} \EuScript{U}_\alpha$ defined by
\begin{equation}
  x_\alpha (u) := \varphi_\alpha( \left(
  \begin{smallmatrix}
    1 & u \\ 0 & 1
  \end{smallmatrix}
  \right) ) \ .\label{f:xalpha}
\end{equation}
Let  $(\alpha, \mathfrak h)\in \Pi_{aff}$ and $s= s_{(\alpha,\mathfrak h)}$. We put
\begin{equation*}
    n_s := \varphi_\alpha (
    \begin{pmatrix}
    0 & \pi^\mathfrak{h} \\
    - \pi^{-\mathfrak{h}} & 0
    \end{pmatrix}
    ) \in N(T ) \ .
\end{equation*}
We have $n_s^2 = \check\alpha(-1) \in T ^0$ and $n_s T ^0 = s \in W $.
We set:
\begin{equation}\label{f:ns}
   \tilde s= n_sT^1 \:\:\in \widetilde W\ .
\end{equation}

From \eqref{f:IandU}, Lemma \ref{lemma:CwC} and \cite{Tit} 1.1 we deduce that $\{x_{\alpha}(\pi^{\mathfrak h} u)\}$ is a system of representatives of $I/I_s$ when $u$ ranges over a system of representatives of $\mathfrak O/\pi \mathfrak O$. We have the decomposition
\begin{equation}
    I n_s I= n_s I\:\dot \cup\: \dot\bigcup_{z\in\mathbb F_q^\times} x_{\alpha}(\pi^{\mathfrak h}[z])n_s I \ .
\end{equation}
Note that since $[-1]= -1\in \mathfrak O$, we have ${x_{\alpha}(\pi^{\mathfrak h} [z])}^{-1}= {x_{\alpha}(-\pi^{\mathfrak h} [z])}={x_{\alpha}(\pi^{\mathfrak h} [-z])}$ and for $z\in\mathbb F_q^\times$, we compute, using $\varphi_\alpha$, that
\begin{equation}\label{f:usefulequality}
   x_{\alpha}(\pi^{\mathfrak h}[z]) \check\alpha([z]) n_s= n_s  x_{\alpha}(\pi^{\mathfrak h}[-z^{-1}])n_s x_{\alpha}(\pi^{\mathfrak h}[-z])\in   n_s I n_s I
\end{equation}
because the Teichm\"uller map $[_-]: \mathbb F_q^\times \rightarrow \mathfrak O^\times$ is a morphism of groups. Since $I n_s I= n_s I\:\dot \cup\: \dot\bigcup_{z\in\mathbb F_q^\times} x_{\alpha}(\pi^{\mathfrak h}[-z^{-1}]) n_s I$, it follows that
\begin{equation}\label{f:sIsI}
   n_s I n_s I=I {n_s^2} \;\dot\cup \; \dot\bigcup_{z\in\mathbb F_q^\times} x_{\alpha}(\pi^{\mathfrak h}[z]) \check\alpha([z]) n_s I \subset  I {n_s^2} \;\dot\cup \;\bigcup_{z\in\mathbb F_q^\times}I  \check\alpha([z]) n_s I
\end{equation}
and hence
\begin{equation}\label{f:quadratic}
  I n_s I\cdot In_s I=  I n_s^2 \; \dot\cup \; \dot\bigcup_{t\in \ima}  I  t n_ s I \ .
\end{equation}

\begin{remark}
Let $s\in S_{aff}$ with lift $\tilde s\in\widetilde W$. Let
 $w\in \widetilde W$ such that $\ell(\tilde sw)=\ell(w)-1$.
From \eqref{f:quadratic} and Cor.\ \ref{coro:known}.ii, we  deduce that
\begin{equation}\label{f:quadratic2}
  I\tilde s I \cdot IwI =  I \tilde sI \cdot I \tilde s I\cdot I(\tilde s^{-1} w) I = I \tilde s w I \;\dot\cup \; \dot\bigcup_{t\in\ima}  I  t w I \ .
\end{equation}
\end{remark}

\begin{lemma}\label{lemma:IuI}
For $u, v,w\in\widetilde W$ satisfying $IuI \subseteq IvI \cdot IwI$, we have:
\begin{equation*}
  |\ell(w) - \ell(v)| \leq \ell(u) \leq \ell(w) + \ell(v) \ .
\end{equation*}

\end{lemma}
\begin{proof}
It is enough the prove that for  $u, v,w\in W$ satisfying $JuJ \subseteq JvJ \cdot JwJ$, we have:
\begin{equation*}
  |\ell(w) - \ell(v)| \leq \ell(u) \leq \ell(w) + \ell(v) \ .
\end{equation*}
 Let $w\in  W$.
We prove  by induction  with respect to $\ell(v)$ that for $u\in W$ such that $JuJ\subseteq J v J\cdot J w J$ we have
\begin{equation}\label{f:ineqs}
  \ell(w) - \ell(v) \leq \ell(u) \leq \ell(w) + \ell(v) \ .
\end{equation}
This will prove the lemma because  $Ju^{-1}J\subseteq J w^{-1} J\cdot J v^{-1} J$ so we have $\ell(u)=\ell(u^{-1})\geq \ell(v^{-1})-\ell(w^{-1})=\ell(v)-\ell(w)$. If $v$ has length $0$ then
$J  v J w J= J v w J$ since $v$ normalizes $J$.  Therefore $u= vw$  and $\ell(u)= \ell(w)$ so the claim is true. Now suppose $v$ has length $1$ meaning $v=\omega s$ for some $s\in S_{aff}$ and $\omega\in \Omega$.
Recall that  $s^2 = 1$. From \eqref{f:quadratic} we deduce
\begin{equation}\label{f:quadraticJ}
  J s J \cdot J s J=  J \; \dot\cup \;  J  s J \ ,
\end{equation}
since $J= T^0 I$ contains $n_s^2$ and $\ima$, where  $(\alpha,\mathfrak h)\in \Pi_{aff}$ is such that $s=s_{(\alpha,\mathfrak h)}$.

If $\ell( s w)=\ell(w)+1$, then $\ell( vw)=\ell(v)+\ell(w)$ and  $J  v J \cdot J w J =  J  v w J$  using  Cor. \ref{coro:known}.ii. Therefore
 \eqref{f:ineqs} is obviously satisfied when  $u=v w$.  Otherwise,  $\ell( s w)=\ell(w)-1$ and $\ell( vw)=\ell(w)-1$. By  \eqref{f:quadraticJ} we get
\begin{equation}\label{f:quadraticJ2}
  J s J \cdot JwJ=  J  sJ \cdot J s J\cdot Js w J = J s w J \;\dot\cup \; J   w J.
\end{equation} Therefore $  J v J \cdot JwJ=  J  \omega s w J \;\dot\cup \; J   \omega w J$
and we see that \eqref{f:ineqs} is satisfied for all $u\in\{vw,\, \omega w\}$.

Now let  $v\in W$ with length $>1$ and  $s\in S_{aff}$ such that $\ell( s v)=\ell(v)-1$.  By induction hypothesis, $J s vJ\cdot J w J$ is the disjoint union of double cosets of the form $J u' J$ with
\begin{equation} \label{f:inequ'}
   \ell(w) - \ell(v)+1 \leq \ell(u') \leq \ell(w) + \ell(v) -1
\end{equation}
and, using the previous case,   $J v J \cdot J w J= J  s J\cdot J  s v J \cdot J w J$ is a union of
double cosets of the form $J u J$ with
\begin{equation}\label{f:inequu'}
   \ell(u') - 1 \leq \ell(u) \leq \ell(u') + 1 \ .
\end{equation}
Combining \eqref{f:inequ'} and  \eqref{f:inequu'}, we see that $u$ satisfies $\ell(w) - \ell(v) \leq \ell(u) \leq \ell(w) + \ell(v)$.
\end{proof}

\subsection{\label{pro-p}The pro-$p$ Iwahori-Hecke algebra}

We start from the compact induction $\mathbf{X} := \ind_I^G(1)$ of the trivial $I$-representation.  It can be seen as the space of  compactly supported functions $G\rightarrow k$ which are constant on the left cosets mod $I$. It lies in $\Mod(G)$.
For $Y$ a compact subset of $G$ which is right invariant under $I$, we denote by $\chara_Y$ the characteristic function of $Y$. It is an element of $\X$.

The pro-$p$ Iwahori-Hecke algebra is defined to be the $k$-algebra $$H := \End_{k[G ]}(\mathbf{X})^{\mathrm{op}}\ .$$ We often will identify $H$, as a right $H$-module, via the map
\begin{align*}
    H & \xrightarrow{\; \cong \;}  \mathbf{X} ^I \\
    h & \longmapsto (\mathrm{char}_I ) h
\end{align*}
 with the submodule $\X^I $ of $I $-fixed vectors in $\X$.
The Bruhat-Tits decomposition of $G$ says that $G$ is the disjoint union of the double cosets $IwI$ for $w \in \widetilde{W}$. Hence we have the $I$-equivariant decomposition
\begin{equation*}
  \mathbf{X} = \oplus_{w \in \widetilde{W}} \mathbf{X}(w)  \quad\text{with}\quad  \mathbf{X}(w) := \ind_I^{IwI}(1) \ ,
\end{equation*}
where the latter denotes the subspace of those functions in $\mathbf{X}$ which are supported on the double coset $IwI$. In particular, we have $\mathbf{X}(w)^I = k \tau_w$ where $\tau_w := \mathrm{char}_{IwI}$ and hence
\begin{equation*}
  H = \oplus_{w \in \widetilde{W}} k \tau_w \
\end{equation*} as a  $k$-vector space. If $g \in IwI$ we sometimes also write $\tau_g := \tau_w$.
The defining relations of $H$ are the
braid relations (see \cite{Vigprop} Thm. 1)
\begin{equation}\label{braid}
  \tau_{w} \tau_{w'} = \tau_{ww'} \qquad\textrm{ for $w,w'\in \tilde W $ such that } \ \ell(ww')=\ell(w)+\ell(w')
\end{equation}
together with  the
quadratic relations which we describe now (compare with \cite{OS1} \S 4.8, and see  more references therein). We refer to the notation introduced in \S\ref{subsubsec:Chevalley}, see \eqref{f:alphafq} in particular. To  any $s = s_{(\alpha,\mathfrak{h})} \in S_{aff}$,  we attach the following idempotent element:
\begin{equation}\label{f:thetas}
    \theta_s := -\vert\kera\vert \sum_{t \in \ima} \tau_t\ \in H\ .
\end{equation}
The quadratic relations in $H$ are:
\begin{equation}\label{quadratic}
    \tau_{n_s}^2 =  -\tau_{n_s}\, \theta_s =-\theta_s\,\tau_{n_s}  \qquad\text{for any $s \in S_{aff}$},
\end{equation}
Since in the existing literature the definition of $\theta_s$ is not correct we will give a proof further below.
Recall that we defined $\tilde s= n_sT^1\in\widetilde W$ in \eqref{f:ns}.
The quadratic relation says that $\tau_{\tilde s}^2= -\theta_s \tau_{\tilde s}=-\tau_{\tilde s } \theta_s$.
A general element $w \in \widetilde W$ can be decomposed into $w = \omega \tilde{s_1} \ldots \tilde{s_\ell}$ with $\omega \in \widetilde \Omega$, $s_i \in S_{aff}$, and $\ell = \ell(w)$.  The braid relations imply $\tau_w = \tau_\omega \tau_{\tilde{s_1}} \ldots \tau_{\tilde{s_\ell}}$.

The subgroup of $G$ generated by all parahoric subgroups is denoted by $G_{aff}$ as in \cite{OS1} \S 4.5. It is a normal subgroup of $G$ and we have $G/G_{aff}= \Omega$. By Bruhat decomposition, $G_{aff}$ is the disjoint union of all $I w I$ for $w$ ranging over  the preimage  $\widetilde W_{aff}$ of $W_{aff}$ in $\widetilde W$. The subalgebra of $H$ of the functions with support in $G_{aff}$ is denoted by $H_{aff}$ and has basis the set of all $\tau_w$, $w\in\widetilde W_{aff}$. When $\mathbf G$ is simply connected semisimple, we have $G=G_{aff}$ and $H= H_{aff}$.

Recall that there is a unique involutive automorphism of $H$ satisfying
\begin{equation}\label{f:upiota}
   \upiota(\tau_{n_s})=-\tau_{n_s}- \theta_s \quad \textrm{and } \upiota(\tau_\omega)=\tau_{\omega} \textrm{ for all $s\in S_{aff}$ and $\omega\in \widetilde \Omega$}
\end{equation}
(see for example \cite{OS1} \S 4.8). It restricts to an
involutive automorphism of $H_{aff}$.  

\begin{proof}[Proof of  \eqref{quadratic}]   First we notice that $\theta_s$ is indeed an idempotent because
\begin{align*}
   \theta_s^2 & = \vert\kera\vert^2\sum_{u,t\in \ima}\tau_{ut} =
   \vert \kera\vert^2\vert\alpha([\mathbb F_q^\times])\vert   \sum_{t\in \ima}\tau_{t}  \\
   & = (q-1)\vert \kera\vert  \sum_{t\in \ima}\tau_{t} = -\vert \kera\vert  \sum_{t\in \ima}\tau_{t} = \theta_s \ .
\end{align*}
The support of $\tau_{n_s}^2$ is contained in $I n_s In_ sI$. Its value at $h \in G$ is equal to $|I n_s I \cap hIn_s^{-1}I/I| \cdot 1_k$, and by \eqref{f:quadratic} we need to consider the cases of $h= n_s^2$ and of $h= t n_s$ for $t\in\ima$. For $h = n_s^2$ this value is equal to $|In_s I/I| \cdot 1_k = \vert I/I_{s}\vert \cdot 1_k = q \cdot 1_k = 0$.  From \eqref{f:sIsI}, we deduce that   for
$t=\check\alpha([\zeta])$ where $\zeta\in\mathbb F_q^\times$, we have
\begin{equation*}
  tn_ s In_ s ^{-1}I=I t \;\dot\cup \;\dot \bigcup_{z\in \mathbb F_q^\times} t x_{\alpha}(\pi^{\mathfrak h}[z]) \check\alpha([z]) n_s^{-1}I=I \check\alpha(\zeta) \;\dot\cup \;\dot\bigcup_{z\in \mathbb F_q^\times} \check\alpha(\zeta) x_{\alpha}(\pi^{\mathfrak h}[z]) \check\alpha([z]) n_s^{-1}I \ .
\end{equation*}
For $z\in\mathbb F_q^\times$ we compute
\begin{equation*}
  \check\alpha([\zeta]) x_{\alpha}(\pi^{\mathfrak h}[z]) \check\alpha([z]) n_s^{-1}= x_{\alpha}(\pi^{\mathfrak h}[\zeta^2z])  \check\alpha([-\zeta z])n_s\in I  \check\alpha([-\zeta z])n_s I \ .
\end{equation*}
So
\begin{equation*}
  tn_ s In_ s ^{-1}I\cap I n_s I= \dot\bigcup_{z\in\mathbb F_q^\times, \: \check\alpha([-\zeta z])=1 } \check\alpha([\zeta]) x_{\alpha}(\pi^{\mathfrak h}[z]) \check\alpha([z]) n_s^{-1}I
\end{equation*}
and $\vert tn_ s In_ s ^{-1}I\cap I n_s I/I\vert=\vert \kera\vert.$ We have proved that  $$\tau_{n_s}^2=\vert \kera\vert\sum_{t\in \ima}  \tau_{tn_s}=-\theta_s n_s\ .$$
Noticing that $\ima n_s=n_s\ima$, we then  obtain $\tau_{n_s}^2=- \tau_{n_s}\theta_s$.
\end{proof}

\subsubsection{\label{sec:idempotents}Idempotents in $H$.}

We now introduce idempotents in $H$ (compare with \cite{OS2} \S 3.2.4). To any $k$-character $\lambda : T^0/T^1 \rightarrow k^\times$ of $T^0/T^1$, we  associate the following idempotent in $H$:
\begin{equation}\label{defel}
   e_\lambda := - \sum_{t \in {T^0/T^1}} \lambda(t^{-1}) \tau_t \ .
\end{equation}
It satisfies
\begin{equation}\label{f:elt}
   e_\lambda\,\tau_t = \tau_t \,e_\lambda=\lambda(t)\, e_\lambda
\end{equation}
for any $t\in {T^0/T^1}$.  In particular, when $\lambda$ is the trivial character we obtain the idempotent element denoted by $e_1$.   Let $s= s_{(\alpha,\mathfrak h)}\in S_{aff}$ with the corresponding idempotent $\theta_s$ as in \eqref{f:thetas}. Using \eqref{f:elt}, we easily see that
$e_\lambda \theta_s=-\sum_{z\in\mathbb F_q^\times} \lambda(\check\alpha([z])) e_\lambda$ therefore
\begin{equation}\label{f:elthetas}
   e_\lambda \theta_s =
   \begin{cases}
     e_\lambda& \textrm{if the restriction of $\lambda$ to $\ima$ is trivial,}\cr
     0&\textrm{otherwise.}
   \end{cases}
\end{equation}

The action of the finite Weyl group $W_0$  on $T^0/T^1$ gives an action on  the characters $T^0/T^1\rightarrow k$. This actions inflates to an action of $\widetilde W$ denoted by $(w,\lambda)\mapsto {}^{w\!}\lambda$. From the braid relations \eqref{braid}, one sees  that  for $w\in\widetilde W$, we have \begin{equation}\label{f:conjel}
  \tau_w\, e_\lambda= e_{{^{w}\lambda}} \, \tau_w \ .
\end{equation}

Suppose  for a moment that  $\mathbb{F}_q \subseteq k$. We then denote by $\widehat{T^0/T^1}$ the set of $k$-characters of $T^0/T^1$. The family  $\{e_\lambda\}_\lambda\in \widehat{T^0/T^1}$ is a family of orthogonal idempotents with sum equal to $1$. It  gives the following ring decomposition
\begin{equation}\label{f:idempotentsOmega}
   k[{T^0/T^1}] = \prod_{\lambda \in \widehat{{T^0/T^1}}} k e_\lambda \ .
\end{equation}
Let $\Gamma$ denote the set of $W_0$-orbits in $\widehat{{T^0/T^1}}$. To $\gamma \in \Gamma$ we attach the element $e_\gamma:=\sum_{\lambda\in \gamma} e_\lambda$. It is a central idempotent in $H$  (see \eqref{f:conjel}), and we have the obvious ring decomposition
\begin{equation}\label{f:idempotents}
   H = \prod_{\gamma \in \Gamma} H e_\gamma \ .
\end{equation}

\subsubsection{\label{sec:charH}Characters of $H$ and $H_{aff}$}

Since for $s\in S_{aff}$ we have $\tau_{n_s}(\tau_{n_s}+\theta_s)=0$ where $\theta_s$ is an idempotent element, we see that a character $H\rightarrow k$ (resp. $H_{aff}\rightarrow k$) takes value $0$ or $-1$ at $\tau_{n_s}$. In fact, the following morphisms of $k$-algebras $H\rightarrow k$ are well defined (compare with \cite{OS1} Definition after Remark 6.13):
\begin{equation}\label{defitriv}
    \chi_{triv}: \tau_{\tilde s}\longmapsto 0, \: \tau_{\omega} \longmapsto 1, \textrm{ for any $s\in S_{aff}$ and $\omega \in \widetilde \Omega$} \ ,
\end{equation}
\begin{equation}\label{defisign}
   \chi_{sign}: \tau_{\tilde s}\longmapsto -1, \: \tau_{\omega} \longmapsto 1, \textrm{ for any $s\in S_{aff}$ and $\omega \in  \widetilde \Omega$} \ .
\end{equation}
They satisfy in particular $\chi_{triv}(e_1)=\chi_{sign}(e_1)=1$. They are called the trivial and the sign character of $H$, respectively. Notice that $\chi_{sign}=\chi_{triv}\circ \upiota$ (see \eqref{f:upiota}). The restriction to $H_{aff}$ of $\chi_{triv}$ (resp. $\chi_{sign}$) is called the  trivial (resp. sign) character of $H_{aff}$.

We call a twisted sign character of $H_{aff}$ a character $\chi: H_{aff}\rightarrow k$ such that $\chi(\tau_{n_s})=-1$ for all $s\in S_{aff}$. The precomposition by $\upiota$ of a
twisted sign character of $H_{aff}$ is called a twisted trivial character. This definition given in \cite{VigIII} coincides with the definition given in \cite{Oll} \S 5.4.2 but it  is simpler and more concise.

\begin{remark}\phantomsection\label{rema:twistedonT'}
\begin{itemize}
\item[i.] A twisted sign character $\chi$  of $H_{aff}$ satisfies  $\chi(\theta_s)=1$ for all $s=s_{(\alpha, \mathfrak h)}\in S_{aff}$ which is equivalent to  $\chi(\tau_t)=1$ for all $t\in \ima$. This is also true for a twisted trivial character since the involution $\upiota$ fixes $\tau_t$ for $t\in T^0/T^1$. Therefore, the twisted sign (resp. trivial) characters of $H_{aff}$ are in bijection with the characters $\lambda: T^0/T^1\rightarrow k$ which are equal to $1$ on
    the subgroup $(T^0/T^1)'$ of $T^0/T^1$ generated by all $\ima$ for $\alpha\in\Phi$, equivalently by all  $\ima$ for $\alpha\in\Pi$.

\item[ii.] The twisted trivial characters of $H_{aff}$ are characterized by the fact that they send  $\tau_w$ to $0$ for all $w\in \widetilde W_{aff}$ with length $>0$ and $\tau_t$ to $1$ for all $t\in (T^0/T^1)'$.

 \item[iii.]    By  Lemma \ref{lemma:normal},
 the action of $\widetilde W$ on $T^0/T^1$, which is inflated from the action of $W_0$,    induces the trivial action on the quotient $(T^0/T^1)/(T^0/T^1)'$. This has the following consequences:
 \begin{itemize}
 \item   $(T^0/T^1)'$ is a normal subgroup of $\widetilde W$.
 \item  given $\lambda: {T^0/T^1}\rightarrow k^\times$ a character which is equal to $1$ on $(T^0/T^1)'$, the corresponding idempotent $e_\lambda$ is central in $H$ (use  \eqref{f:conjel})
 \item   the natural action $(\omega, \chi)\mapsto \chi({\tau_\omega^{-1}}_-\tau_\omega)$ of $\widetilde\Omega$ on the characters of $H_{aff}$ fixes the twisted trivial characters.  Since $\upiota$ fixes the elements $\tau_\omega$ for $\omega\in\widetilde\Omega$, this action also fixes the twisted sign characters.\end{itemize}
   \item[iv.]
 When $\mathbf G$ is  semisimple simply connected, then
$H=H_{aff}$ and $(T^0/T^1)'= T^0/T^1$ (Lemma \ref{lemma:Omega}.i) so the trivial (resp. sign) character of $H= H_{aff}$ is the only twisted   trivial (resp. sign)  character.
\end{itemize}
\end{remark}

As in \cite{VigIII} \S1.4, we notice that the  Coxeter system $(W_{aff}, S_{aff})$ is the direct product of the irreducible affine Coxeter systems $(W_{aff}^i , S_{aff}^i )_{1\leq i \leq r}$ corresponding  to the irreducible components $(\Phi^i, \Pi^i)_{1\leq i \leq r}$ of the based root system $(\Phi, \Pi)$. For $i\in\{1,\ldots, r\}$, the $k$-module of basis $(\tau_w)_{w\in\widetilde W_{aff}^i}$ is a subalgebra of $H_{aff}$.  Remark that it contains $\{\tau_t, t\in T^0/T^1\}$. Remark also that the anti-involution $\upiota$  restricts to an anti-involution of  $H_{aff}^i$. We call a twisted sign character of $H^i_{aff}$ a character $\chi: H^i_{aff}\rightarrow k$ such that $\chi(\tau_{n_s})=-1$ for all $s\in S^i_{aff}$. The precomposition by $\upiota$ of a twisted sign character of $H^i_{aff}$ is called a twisted trivial character. The algebras $(H_{aff}^i)_{1\leq i \leq r}$ will be called the irreducible components of $H_{aff}$.

A character of $H_{aff}\rightarrow k$ will be called \textbf{supersingular} if, for every $i\in\{1,\dots, r\}$, it does not restrict to a twisted trivial or sign character of $H_{aff}^i$. This terminology is justified in the following  subsection.

\subsection{\label{subsec:supersing}Supersingularity}

We refer here to  definitions and results of \cite{Oll} \S2 and \S5. Note that there the field  $k$ was algebraically closed,  but is easy to see that  the claims   that we are going to use are valid when $k$ is not necessarily algebraically closed. In fact, in \cite{VigIII}, these  definitions and results are generalized to the case where $k$ is an arbitrary  field of characteristic $p$ and  $\mathbf G$ is a general  connected reductive $\mathfrak F$-group.

In \cite{Oll} \S2.3.1. a central subalgebra  $\mathcal Z^0(H)$ of $H$  is defined (it is denoted by $\mathcal Z_T$ in \cite{VigIII}). This algebra is isomorphic to the affine semigroup algebra $k[X_*^{dom}(T)]$, where
$X_*^{dom}(T)$ denotes the semigroup of all dominant cocharacters of $T$ (\cite{Oll} Prop. 2.10).
The cocharacters $\lambda\in X_*^{dom}(T)\setminus (-X_*^{dom}(T))$ generate a proper ideal of $k[X_*^{dom}(T)]$, the image of which in $\mathcal Z^0(H)$  is denoted by $\mathfrak J$ as in \cite{Oll} \S5.2 (it coincides with the ideal $\mathcal Z_{T, \ell>0}$ of \cite{VigIII}).


Generalizing \cite{Oll} Prop.-Def.\ 5.10 and \cite{VigIII} Def. 6.10 we call an $H$-module $M$ supersingular  if any element in $M$ is annihilated by a power of $\mathfrak J$. Recall that a finite length $H$-module is always finite dimensional (see for example \cite{OS1} Lemma 6.9). Hence, if $M$ has finite length, then it is supersingular if and only if it is annihilated by a power of $\mathfrak{J}$. Also note that supersingularity can be tested after an arbitrary extension of the coefficient field $k$.

The supersingular characters of $H_{aff}$ were defined at the end of \S\ref{sec:charH}. The two notions of supersingularity are related by the following fact.

\begin{lemma}
\begin{itemize}
\item[-] Let $\chi: H_{aff}\rightarrow k$ be a supersingular  character of $H_{aff}$. The left (resp. right) $H$-module  $H\otimes _{H_{aff}} \chi$ (resp.  $\chi\otimes  _{H_{aff}} H $) is  annihilated by $\mathfrak J$. In particular, it is a supersingular $H$-module.
\item[-] Let $\chi: H_{aff}\rightarrow k$ be a  twisted trivial or sign  character of $H_{aff}$. The left (resp. right) $H$-module  $H\otimes _{H_{aff}} \chi$ (resp.  $\chi\otimes  _{H_{aff}} H $) does not have any  nonzero supersingular subquotient.
\end{itemize}
\phantomsection
\label{lemma:critsupersing}
\end{lemma}
\begin{proof}  Since supersingularity can be tested after an arbitrary extension of the coefficient field $k$, we may assume  in this proof that $k$ is algebraically closed. We only need to show that the generator $1 \otimes 1$ of $H \otimes_{H_{aff}}\chi$ is annihilated by $\mathfrak J$. In that case,
Theorem 5.14 in \cite{Oll} (when the root system is irreducible) and  Corollary 6.13 in \cite{VigIII} state that, if a simple (left) $H$-module $M$ contains a  supersingular character  $\chi$ of  $H_{aff}$, then it is a  supersingular $H$-module (in fact they state that this is an equivalence). The proof of the statement    consists in picking an element $m$ in $M$ supporting the character $\chi$ and proving that $\mathfrak J$ acts by zero on it, the simplicity of $M$ being used only to ensure that $\mathfrak J$ acts by zero on the whole $M$. Therefore the arguments apply to the left $H$-module    $H \otimes_{H_{aff}}\chi$ when  $\chi$ is a supersingular  character of $H_{aff}$ (although this module may not even be of finite length).

Now let $\chi: H_{aff}\rightarrow k$ be a  twisted trivial (resp. sign)  character of $H_{aff}$.
As an $H_{aff}$-module, $H\otimes_{H_{aff}}\chi$ is isomorphic to a direct sum
of  copies of $\chi$ (Remark \ref{rema:twistedonT'}.iii).
A nonzero $H$-module which is a subquotient of
 $H\otimes_{H_{aff}}\chi$
 is therefore also  a direct sum
of  copies of $\chi$ as an $H_{aff}$-module.
Suppose that  $H\otimes_{H_{aff}}\chi$ has a nonzero supersingular subquotient. Then it has a
nonzero supersingular finitely generated subquotient. Since the latter admits a nonzero simple  quotient,   the $H$-module $H\otimes_{H_{aff}}\chi$ has a nonzero simple supersingular subquotient  $M$. This is not compatible with $M$ being a direct sum of  copies of $\chi$ as an $H_{aff}$-module, as proved
in   \cite{Oll} Lemma 5.12 when the root system is irreducible or in \cite{VigIII}  Corollary  6.13.

The proof is the same for right $H$-modules.
\end{proof}

Define the decreasing filtration
\begin{equation}\label{f:defifil}
   F^n H:=\oplus _{\ell(w)\geq n}k\tau_w \qquad \textrm{  for $n\geq 0$}
\end{equation}
of $H$ as a bimodule over itself.

\begin{lemma}\label{lemma:F1/Fm}
Under the hypothesis that $\mathbf G$ is semisimple simply connected with irreducible root system, we have:
\begin{itemize}
\item[i.] As an $H$-module on the left and on the right, $F^m H/  F^{m+1} H$, for any $m\geq 1$, is annihilated by $\mathfrak{J}$; in particular:
\begin{equation*}
  \mathfrak J^{m-1}\cdot F^1 H= F^1 H\cdot \mathfrak J^{m-1}\subset  F^{m} H \ .
\end{equation*}
\item[ii.] As an $H$-module on the left and on the right,   $(1-e_1)\cdot (F^0 H/  F^{1} H)$ is annihilated by $\mathfrak{J}$; in particular:
\begin{equation*}
  \mathfrak J^{m}\cdot [(1- e_1) \cdot F^0 H+ F^1H]\subset  F^{m} H \ .
\end{equation*}
\end{itemize}
\end{lemma}
\begin{proof}
In this proof we consider left $H$-modules. The arguments are the same for the structures of right modules. Recall that under the hypothesis of the lemma, we have $W = W_{aff}$ and $H= H_{aff}$. Furthermore, we may assume that $\mathbb{F}_q \subseteq k$ and that $\mathbf{G} \neq 1$.

i. We will show that $F^m H/  F^{m+1} H$, in fact, is a direct sum of supersingular characters. Since $\mathbb F_q\subseteq k$, a basis for $F^m H/  F^{m+1} H$  is given  by all $e_\lambda\tau_{\tilde w}$ for $w\in W$  with length $m$ and lift $\tilde w\in \widetilde W$ and all $\lambda\in\widehat {T^0/T^1}$ (notation in \S\ref{sec:idempotents}). For $s= s_{(\alpha, \mathfrak h)}\in S_{aff}$ we pick the  lift $\tilde s$ as in \eqref{f:ns}.
Using \eqref{f:conjel}, \eqref{braid} and  \eqref{quadratic}:
\begin{equation*}
  \tau_{\tilde s}\cdot e_\lambda  \tau_{\tilde w} =
  \begin{cases}
  e_{{}^{s\!}\lambda}  \, \tau_{\tilde s}^2\, \tau_{ \tilde s ^{-1}\tilde w}= - \theta_s \,e_{{}^{s\!}\lambda} \,\tau_{ \tilde w} & \textrm{ if $\ell(sw)=\ell(w)-1$,}\cr
  e_{{}^{s\!}\lambda} \tau_{\tilde s \tilde w}& \textrm{ if $\ell(sw)=\ell(w)+1$.}
  \end{cases}
\end{equation*}
So in  $  F^m H/ F^{m+1} H$,  we have, using  \eqref{f:elthetas}:
\begin{equation*}
  \tau_{\tilde s}\cdot  e_\lambda\tau_{\tilde w}=\left\lbrace\begin{array}{ll} 0 & \textrm{ if $\ell(sw)=\ell(w)-1$ and $\lambda\vert_{\ima}\neq 1$,}\cr
  -e_{{}^{s\!}\lambda}\tau_{\tilde w} & \textrm{ if $\ell(sw)=\ell(w)-1$ and  $\lambda\vert_{\ima}= 1$,}\\
  0& \textrm{ if $\ell(sw)=\ell(w)+1$.}
 \end{array}\right.
\end{equation*}
Using Lemma \ref{lemma:normal}, notice that if $\lambda$ is trivial on $\ima$, then ${}^s\lambda=\lambda$. This proves that $e_\lambda\tau_{\tilde w}$ supports the character $\chi: H_{aff}\rightarrow k$ defined by $\chi(\tau_t)= \lambda(t)$ and
\begin{equation*}
  \chi(\tau_{\tilde s})=\left\lbrace\begin{array}{ll} 0 & \textrm{ if $\ell(sw)=\ell(w)-1$ and $\lambda\vert_{\ima}\neq 1$ or if $\ell(sw)=\ell(w)+1$,}\cr
  -1 & \textrm{ if $\ell(sw)=\ell(w)-1$ and  $\lambda\vert_{\ima}= 1$} \end{array}\right.
\end{equation*}
for $s\in S_{aff}$. If there is $s_{(\alpha, \mathfrak h)}\in S_{aff}$ such that $\lambda\vert_{\ima}\neq 1$ then $\chi$ is supersingular (Remark \ref{rema:twistedonT'}). Otherwise $\lambda$ is trivial on $T^0/T^1$ and we want to check that $\chi$ is not a twisted sign or a twisted trivial character. This is because $m\geq 1$ and the root system of $\mathbf G$ is irreducible of rank $>0$. Therefore there are $s,s'\in S_{aff}$ such that $\ell(s w)=\ell(w)-1$ and $\ell(s'w)=\ell(w)+1$. Since $H_{aff}$ has only one irreducible component we see that $\chi$ is a supersingular character. By Lemma \ref{lemma:critsupersing}, the left $H$-module $F^m H/ F^{m+1 }H$ then is annihilated by $\mathfrak J$, which concludes the proof of i.

ii. Again we will show that $(1-e_1)\cdot (F^0 H/  F^{1} H)$, in fact, is  a direct sum of supersingular characters. A basis for $(1-e_1)\cdot (F^0 H/  F^{1} H)$  is given  by all $e_\lambda$ for all $\lambda\in\widehat {T^0/T^1}\setminus\{1\}$. But $\tau_{\tilde s} e_\lambda \in F^1 H$ by \eqref{braid}. This proves that $e_\lambda$ supports the character $\chi: H_{aff}\rightarrow k$ defined by $\chi(\tau_t)= \lambda(t)$ (see \eqref{f:elt}) and $ \chi(\tau_{\tilde s})=0$  for $s\in S_{aff}$. It is not a twisted sign or a twisted trivial character since $\lambda$ is nontrivial on $(T^0/T^1)' = T^0/T^1$. As the root system is irreducible, it is a supersingular character. Conclude using point i. and Lemma \ref{lemma:critsupersing}.
\end{proof}

\section{The $\Ext$-algebra}\label{sec:Ext}

\subsection{The definition}\label{subsec:def}

In order to introduce the algebra in the title we  again start from the compact induction $\mathbf{X} = \ind_I^G(1)$ of the trivial $I$-representation,  which lies in $\Mod(G)$. We form the graded $\Ext$-algebra
\begin{equation*}
  E^* := \Ext_{\Mod(G)}^*(\mathbf{X},\mathbf{X})^{\mathrm{op}}
\end{equation*}
over $k$ with the multiplication being the (opposite of the) Yoneda product. Obviously
\begin{equation*}
  H := E^0 = \End_{\Mod(G)}^*(\mathbf{X},\mathbf{X})^{\mathrm{op}}
\end{equation*}
is the usual pro-$p$ Iwahori-Hecke algebra over $k$. By using Frobenius reciprocity for compact induction and the fact that the restriction functor from $\Mod(G)$ to $\Mod(I)$ preserves injective objects we obtain the identification
\begin{equation}\label{f:frob}
  E^* = \Ext_{\Mod(G)}^*(\mathbf{X},\mathbf{X})^{\mathrm{op}} = H^*(I,\mathbf{X}) \ .
\end{equation}
The only part of the multiplicative structure on $E^*$ which is still directly visible on the cohomology $H^*(I,\mathbf{X})$ is the right multiplication by elements in $E^0 = H$, which is functorially induced by the right action of $H$ on $\mathbf{X}$. It is one of the main technical issues of this paper to make the full multiplicative structure visible on $H^*(I,\mathbf{X})$. We recall that for $* = 0$ the above identification is given by
\begin{align*}
  H & \xrightarrow{\; \cong \;} \mathbf{X}^I  \\
  \tau & \longmapsto (\mathrm{char}_I) \tau \ .
\end{align*}

\subsection{The technique}\label{subsec:technique}

The technical tool for studying the algebra $E^*$ is the $I$-equivariant decomposition
\begin{equation*}
  \mathbf{X} = \oplus_{w \in \widetilde{W}} \mathbf{X}(w)
\end{equation*}
introduced in section \ref{pro-p}. Noting that the cohomology of profinite groups commutes with arbitrary sums, we obtain
\begin{equation*}
  H^*(I,\mathbf{X}) = \oplus_{w \in \widetilde{W}} H^*(I,\mathbf{X}(w)) \ .
\end{equation*}
Similarly as we write $IwI$, since this double coset only depends on the coset $w \in N(T)/T^1$, we will silently abuse notation in the following whenever something only depends on the coset $w$. We have the isomorphism of $I$-representations
\begin{align*}
  \ind_I^{IwI}(1) & \xrightarrow{\;\cong\;} \ind_{I_w}^I(1) \\
  f & \longmapsto \phi_f(a) := f(aw) \ .
\end{align*}
This gives rise to the left hand cohomological isomorphism
\begin{equation*}
  H^*(I,\mathbf{X}(w)) \xrightarrow{\;\cong\;} H^*(I,\ind_{I_w}^I(1)) \xrightarrow{\;\cong\;} H^*(I_w,k)
\end{equation*}
which we may combine with the right hand Shapiro isomorphism. For simplicity we will call in the following the above composite isomorphism the Shapiro isomorphism and denote it by $\Sh_w$. Equivalently, it can be described as the composite map
\begin{equation}\label{f:Shapiro1}
  \Sh_w : H^*(I,\mathbf{X}(w)) \xrightarrow{\; \res \;} H^*(I_w,\mathbf{X}(w)) \xrightarrow{\; H^*(I_w,\ev_w) \;} H^*(I_w,k)
\end{equation}
where
\begin{align*}
   \ev_w : \mathbf{X}(w) & \longrightarrow k \\
      f & \longrightarrow f(w) \ .
\end{align*}
 We leave it as an exercise to the reader to check that the map
\begin{equation}\label{f:Shapiro-inverse}
  \Sh_w^{-1} : H^*(I_w,k) \xrightarrow{\; \mathrm{i}_w \;} H^*(I_w,\mathbf{X}(w)) \xrightarrow{\; \cores \;} H^*(I,\mathbf{X}(w)) \ ,
\end{equation}
where
\begin{align*}
   \mathrm{i}_w : k & \longrightarrow \mathbf{X}(w) \\
      a & \longrightarrow a \chara_{wI} \ ,
\end{align*}
is the inverse of the Shapiro isomorphism $\Sh_w$.

\subsection{The cup product}\label{subsec:cup-prod}

There is a naive product structure on the cohomology $H^*(I,\mathbf{X})$. By multiplying maps we obtain the $G$-equivariant map
\begin{align*}
  \mathbf{X} \otimes_k \mathbf{X} & \longrightarrow \mathbf{X} \\
  f \otimes f' & \longmapsto f f' \ .
\end{align*}
It gives rise to the cup product
\begin{equation}\label{f:cup}
  H^i(I, \mathbf{X}) \otimes_k H^j(I,\mathbf{X}) \xrightarrow{\; \cup \;} H^{i+j}(I,\mathbf{X}) \
\end{equation} which,
quite obviously, has the property that
\begin{equation}\label{f:orth}
  H^i(I, \mathbf{X}(v)) \cup H^j(I,\mathbf{X}(w)) = 0 \quad\text{whenever $v \neq w$}.
\end{equation}
On the other hand, since $\ev_{{w}}(f f') = \ev_{{w}}(f) \ev_{{w}}(f')$ and since the cup product is functorial and commutes with cohomological restriction maps, we have the commutative diagrams
\begin{equation}\label{f:cup+Sh}
  \xymatrix{
     H^i(I, \mathbf{X}(w)) \otimes_k H^j(I,\mathbf{X}(w)) \ar[d]_{\Sh_w \otimes \Sh_w} \ar[r]^-{\cup} & H^{i+j}(I,\mathbf{X}(w)) \ar[d]^{\Sh_w} \\
     H^i(I_w,k) \otimes_k H^j(I_w,k) \ar[r]^-{\cup} & H^{i+j}(I_w,k)   }
\end{equation}
for any $w \in \widetilde{W}$, where the bottom row is the usual cup product on the cohomology algebra $H^*(I_w,k)$.  In particular, we see that the cup product \eqref{f:cup} is anticommutative.

\section{Representing cohomological operations on resolutions}\label{sec:resolutions}

\subsection{The Shapiro isomorphism}\label{subsec:Shapiro}

The Shapiro isomorphism \eqref{f:Shapiro1} also holds for nontrivial coefficients provided we choose once and for all, as we will do in the following, a representative $\dot{w} \in N(T)$ for each $w \in \widetilde{W}$. Compact induction is an exact functor
\begin{align*}
  \ind_I^G : \Mod(I) & \longrightarrow \Mod(G) \\
   Y & \longmapsto \ind_I^G(Y) \ .
\end{align*}
Moreover, as before we have the decomposition $\ind_I^G(Y) = \oplus_{w \in \widetilde{W}} \ind_I^{IwI}(Y)$ and the isomorphism
\begin{align*}
  \ind_I^{IwI}(Y) & \xrightarrow{\;\cong\;} \ind_{I_{\dot{w}}}^I(\dot{w}_* \res_{I_{\dot{w}^{-1}}}^I (Y)) \\
  f & \longmapsto \phi_f(a) := f(a\dot{w})
\end{align*}
as $I$-representations. On cohomology we obtain the commutative diagram
\begin{equation}\label{f:Shapiro2}
  \xymatrix{
  H^*(I,\ind_I^{IwI}(Y)) \ar[r]^-{\cong} \ar[dr]_{\res} & H^*(I,\ind_{I_w}^I(\dot{w}_* \res_{I_{w^{-1}}}^I(Y)) \ar[r]^-{\cong} & H^*(I_w, \dot{w}_* Y)      \\
                &  H^*(I_w,\ind_I^{IwI}(Y))   \ar[ur]_{H^*(I_w,\ev_{\dot{w}})}             }
\end{equation}
in which $\ev_{\dot{w}}$ now denotes the evaluation map in $\dot{w}$ and in which the composite map in the top row is an isomorphism denoted by $\Sh_{\dot{w}}$.

To lift this to the level of complexes we first make the following observation.

\begin{lemma}\label{Sh-injective}
If $\mathcal{J}$ is an injective object in $\Mod(I)$ then $\ind_I^{IwI}(\mathcal{J})$, for any $w \in \widetilde{W}$, is an injective object in $\Mod(I)$ as well.
\end{lemma}
\begin{proof}
We use the isomorphism $\ind_I^{IwI}(\mathcal{J}) \cong \ind_{I_w}^I(\dot{w}_* \res_{I_{w^{-1}}}^I (\mathcal{J}))$. As recalled before, the restriction functor to open subgroups preserves injective objects. The functor $\dot{w}_* : \Mod(I_{w^{-1}}) \xrightarrow{\simeq} \Mod(I_w)$ is an equivalence of categories and hence preserves injective objects. Finally, the induction functor from open subgroups of finite index also preserves injective objects (cf.\ \cite{Vig} I.5.9.b)).
\end{proof}

Let now $k \xrightarrow{\sim} \mathcal{I}^\bullet$ and $k \xrightarrow{\sim} \mathcal{J}^\bullet$ be any two injective resolutions in $\Mod(I)$ of the trivial representation. By Lemma \ref{Sh-injective} then $\mathbf{X}(w) \xrightarrow{\sim} \ind_I^{IwI}(\mathcal{J}^\bullet)$ is an injective resolution in $\Mod(I)$ as well. Hence
\begin{equation*}
  H^*(I,\mathbf{X}(w)) = \Hom_{K(I)}(\mathcal{I}^\bullet,\ind_I^{IwI}(\mathcal{J}^\bullet)[*]) \ ,
\end{equation*}
i.e., any cohomology class in $H^*(I,\mathbf{X}(w))$ is of the form $[\alpha^\bullet]$ for some homomorphism of complexes $\alpha^\bullet : \mathcal{I}^\bullet \longrightarrow \ind_I^{IwI}(\mathcal{J}^\bullet)[*]$ in $\Mod(I)$ (which is unique up to homotopy). Composition with the evaluation map gives rise to the homomorphism of injective complexes
\begin{equation*}
  \Sh_{\dot{w}}(\alpha^\bullet) := \ev_{\dot{w}} \circ \alpha^\bullet : \mathcal{I}^\bullet \longrightarrow \dot{w}_* \mathcal{J}^\bullet [*]
\end{equation*}
in $\Mod(I_w)$ whose cohomology class is $\Sh_w([\alpha^\bullet]) \in H^*(I_w,k) = \Hom_{K(I_w)}(\mathcal{I}^\bullet,\dot{w}_*\mathcal{J}^\bullet[*])$.

In fact it will be more convenient  later on to  use the following modified version of the Shapiro isomorphism. For this we assume that $k \xrightarrow{\sim} \mathcal{J}^\bullet$ is actually  an injective resolution in $\Mod(G)$ (and hence in $\Mod(I)$). Then the composite map
\begin{equation}\label{f:modified-Sh}
  \Sh'_{\dot{w}}(\alpha^\bullet) : \mathcal{I}^\bullet \xrightarrow{\Sh_{\dot{w}}(\alpha^\bullet)} \dot{w}_* \mathcal{J}^\bullet [*] \xrightarrow{y \mapsto \dot{w}y} \mathcal{J}^\bullet [*]
\end{equation}
is defined and is also a homomorphism of injective complexes in $\Mod(I_w)$ representing the same cohomology class as $\Sh_{\dot{w}}(\alpha^\bullet)$ but viewed in $H^*(I_w,k) = \Hom_{K(I_w)}(\mathcal{I}^\bullet,\mathcal{J}^\bullet[*])$, i.e., we have
\begin{equation}\label{f:Shapiro-compl}
  [\Sh'_{\dot{w}}(\alpha^\bullet)] = [\Sh_{\dot{w}}(\alpha^\bullet)] = \Sh_w([\alpha^\bullet]) \ .
\end{equation}
 The homomorphism $\alpha^\bullet$ can be reconstructed from $\Sh'_{\dot{w}}(\alpha^\bullet)$ by the formula
\begin{equation}\label{f:reconstruct}
  \alpha^\bullet(x)(a\dot{w} b) = b^{-1}(({^{a^{-1}}(}\alpha^\bullet(x)))(\dot{w})) = b^{-1}(\alpha^\bullet (a^{-1}x)(\dot{w})) = (a\dot{w}b)^{-1} a (\Sh'_{\dot{w}}(\alpha^\bullet)(a^{-1}x))
\end{equation}
for any $x \in \mathcal{I}^\bullet$ and any $a,b \in I$.

\subsection{The Yoneda product}\label{subsec:Yoneda}

Here we consider an injective resolution $\mathbf{X} \xrightarrow{\sim} \mathcal{I}^\bullet$ of our $G$-representation $\mathbf{X}$ in $\Mod(G)$. Then
\begin{equation*}
  E^* = \Ext_{\Mod(G)}^*(\mathbf{X},\mathbf{X}) = \Hom_{D(G)}(\mathbf{X},\mathbf{X}[*]) = \Hom_{K(I)}(\mathcal{I}^\bullet,\mathcal{I}^\bullet [*]) \ ,
\end{equation*}
and the Yoneda product is the obvious composition of homomorphisms of complexes (cf.\ \cite{Har} Cor.\ I.6.5). We recall, though, that our convention is to consider the opposite of this composition. For our purposes it is crucial to replace $\mathcal{I}^\bullet$ by a quasi-isomorphic complex constructed as follows.

We begin with an injective resolution $k \xrightarrow{\sim} \mathcal{J}^\bullet$ in $\Mod(I)$ of the trivial representation. Then $\mathbf{X} \xrightarrow{\sim} \ind_I^G(\mathcal{J}^\bullet) = \oplus_{w \in \widetilde{W}} \ind_I^{IwI}(\mathcal{J}^\bullet)$ is a resolution in $\Mod(G)$. By Lemma \ref{Sh-injective} each term $\ind_I^{IwI}(\mathcal{J}^\bullet)$ is injective in $\Mod(I)$. Since the cohomology functor $H^*(I,-)$ commutes with arbitrary sums in $\Mod(I)$ it follows that each term $\ind_I^G(\mathcal{J}^\bullet)$ is an $H^0(I,-)$-acyclic object in $\Mod(I)$. But by Frobenius reciprocity we have the isomorphism $\Hom_{\Mod(G)}(\ind_I^G(1),-) \cong H^0(I,-)$ of left exact functors on $\Mod(G)$. We conclude that $\mathbf{X} \xrightarrow{\sim} \ind_I^G(\mathcal{J}^\bullet)$ is a resolution of $\mathbf{X}$ in $\Mod(G)$ by $\Hom_{\Mod(G)}(\mathbf{X},-)$-acyclic objects. It follows that
\begin{align}\label{f:Ext}
  \Ext^*_{\Mod(G)} (\mathbf{X}, \mathbf{X}) & = h^*(\Hom_{\Mod(G)}(\mathbf{X}, \ind_I^G(\mathcal{J}^\bullet))) \\
   & = h^*(\ind_I^G(\mathcal{J}^\bullet)^I)  \nonumber \\
   & = \oplus_{w \in \widetilde{W}} \; h^*(\ind_I^{IwI}(\mathcal{J}^\bullet)^I) \cong \oplus_{w \in \widetilde{W}} \; h^*(\ind_{I_w}^I(\dot{w}_* \res_{I_{w^{-1}}}^I (\mathcal{J}^\bullet))^I) \nonumber \\
   & \cong \oplus_{w \in \widetilde{W}} \; h^*((\mathcal{J}^\bullet)^{I_{w^{-1}}}) \ .   \nonumber
\end{align}
In order to lift  these equalities to the level of complexes we consider the commutative diagram
\begin{equation*}
  \xymatrix{
    \Ext^*_{\Mod(G)} (\mathbf{X}, \mathbf{X})                       &   \\
    \Hom_{D(G)}(\ind_I^G(\mathcal{J}^\bullet), \ind_I^G(\mathcal{J}^\bullet)[*]) \ar@{=}[u]_{} \ar@{=}[r]^{} & h^*(\Hom_{\Mod(G)}(\mathbf{X}, \ind_I^G(\mathcal{J}^\bullet))) \\
    \Hom_{K(G)}(\ind_I^G(\mathcal{J}^\bullet), \ind_I^G(\mathcal{J}^\bullet)[*]) \ar[u]_{} &  \\
    \Hom_{K(I)}(\mathcal{J}^\bullet, \ind_I^G(\mathcal{J}^\bullet)[*]) \ar[u]_{\cong}^{\text{Frobenius reciprocity}} \ar[r]^{} & H^*(I,\mathbf{X})) \ar[uu]_{\text{Frobenius reciprocity}}^{\cong} \\
    \oplus_{w \in \widetilde{W}} \Hom_{K(I)}(\mathcal{J}^\bullet, \ind_I^{IwI}(\mathcal{J}^\bullet)[*]) \ar[u]^{} \ar[r]^-{\cong} & \oplus_{w \in \widetilde{W}} H^*(I,\mathbf{X}(w)). \ar@{=}[u]^{}  }
\end{equation*}
The isomorphism in the bottom row is a consequence of Lemma \ref{Sh-injective}. The computation \eqref{f:Ext} shows that the composite map in the first column is an isomorphism.

We point out that these two ways of representing $E^*$ by homomorphisms of complexes, through $\mathcal{I}^\bullet$ and through $\ind_I^G(\mathcal{J}^\bullet)$, are related by the unique (up to homotopy) homomorphism of complexes in $\Mod(G)$ which makes the diagram
\begin{equation*}
  \xymatrix@R=0.5cm{
                &         \ind_I^G(\mathcal{J}^\bullet) \ar@{-->}[dd]     \\
  \mathbf{X} \ar[ur]^-{\sim} \ar[dr]_-{\sim}                 \\
                &         \mathcal{I}^\bullet                 }
\end{equation*}
commutative.

Consider any classes $[\alpha^\bullet] \in H^i(I,\mathbf{X}(v))$ and $[\beta^\bullet] \in H^j(I,\mathbf{X}(w))$ represented by homomorphisms of complexes
\begin{equation*}
  \alpha^\bullet : \mathcal{J}^\bullet \longrightarrow \ind_I^{IvI}(\mathcal{J}^\bullet)[i] \subseteq \ind_I^G(\mathcal{J}^\bullet)[i]  \quad\text{and}\quad  \beta^\bullet : \mathcal{J}^\bullet \longrightarrow \ind_I^{IwI}(\mathcal{J}^\bullet)[j] \subseteq \ind_I^G(\mathcal{J}^\bullet)[j] \ ,
\end{equation*}
respectively. According to the above diagram these induce, by Frobenius reciprocity, homomorphisms of complexes $\tilde{\alpha}^\bullet : \ind_I^G(\mathcal{J}^\bullet) \longrightarrow \ind_I^G(\mathcal{J}^\bullet)[i]$ and $\tilde{\beta}^\bullet : \ind_I^G(\mathcal{J}^\bullet) \longrightarrow \ind_I^G(\mathcal{J}^\bullet)[j]$ which represent our original classes when viewed in $\Ext^i_{\Mod(G)} (\mathbf{X}, \mathbf{X})$ and $\Ext^j_{\Mod(G)} (\mathbf{X}, \mathbf{X})$, respectively. The Yoneda product of the latter is represented by the composite $\tilde{\beta}^\bullet[i] \circ \tilde{\alpha}^\bullet$, which we may write as $\tilde{\beta}^\bullet[i] \circ \tilde{\alpha}^\bullet = \tilde{\gamma}^\bullet$ for a homomorphism of complexes $\gamma^\bullet : \mathcal{J}^\bullet \longrightarrow \ind_I^G(\mathcal{J}^\bullet)[i+j]$. By writing out the Frobenius reciprocity isomorphism we see that
\begin{align}\label{f:tilde}
  \tilde{\beta}^\bullet : \ind_I^G(\mathcal{J}^\bullet) & \longrightarrow \ind_I^G(\mathcal{J}^\bullet)[j]  \\
                                                      f & \longmapsto \sum_{g \in G/I} g \beta^\bullet (f(g)) \ .   \nonumber
\end{align}
We deduce that, if $f$ has support in the subset $S \subseteq G/I$, then $\tilde{\beta}^\bullet (f) = \sum_{g \in S} g \beta^\bullet (f(g))$ has support in $S \cdot IwI$. Applying this to functions in the image of $\alpha^\bullet$, which are supported in $IvI$, we obtain that $\gamma^\bullet$, in fact, is a homomorphism of complexes
\begin{equation*}
  \gamma^\bullet : \mathcal{J}^\bullet \longrightarrow \ind_I^{IvI \cdot IwI}(\mathcal{J}^\bullet)[i+j] \ .
\end{equation*}
This shows that, if $\cdot$ denotes the multiplication on $H^*(I,\mathbf{X})$ induced by the opposite of the Yoneda product, then we have $[\alpha^\bullet] \cdot [\beta^\bullet] =  (-1)^{ij}  [\gamma^\bullet]$ and hence
\begin{equation}\label{f:support}
  H^i(I,\mathbf{X}(v)) \cdot H^j(I,\mathbf{X}(w)) \subseteq H^{i+j}(I, \ind_I^{IvI \cdot IwI}(1)) \ .
\end{equation}

\subsection{The cup product}\label{subsec:cup-pr2}

Let $U$ be any profinite group.  It is well known that, under the identification $H^*(U,k) = \Ext^*_{\Mod(U)}(k,k)$, the cup product pairing
\begin{equation*}
  H^i(U,k) \times H^j(U,k) \xrightarrow{\; \cup \;} H^{i+j}(U,k)
\end{equation*}
coincides with the Yoneda composition product
\begin{equation*}
  \Ext^i_{\Mod(U)} (k,k) \times \Ext^j_{\Mod(U)}(k,k) \xrightarrow{\; \circ \;} \Ext^{i+j}_{\Mod(U)}(k,k) \ .
\end{equation*}
For discrete groups this is, for example, explained in \cite{Bro} V\S4. The argument there uses projective resolutions and therefore cannot be generalized directly to profinite groups. Instead one may use the axiomatic approach in \cite{Lan} Chap.\ IV (see also p.\ 136).

We will use this in the following way. Let $k \xrightarrow{\sim} \mathcal{I}^\bullet$, $k \xrightarrow{\sim} \mathcal{J}^\bullet$, and $k \xrightarrow{\sim} \mathcal{K}^\bullet$ be three injective resolutions in $\Mod(U)$. Any two cohomology classes $\alpha \in H^i(U,k)$ and $\beta \in H^j(U,k)$ can be represented by homomorphisms of complexes $\alpha^\bullet : \mathcal{J}^\bullet \rightarrow \mathcal{K}^\bullet[i]$ and $\beta : \mathcal{I}^\bullet \rightarrow \mathcal{J}^\bullet[j]$, respectively. Then $\alpha \cup \beta \in H^{i+j}(U,k)$ is represented by the composite $\alpha^\bullet[j] \circ \beta^\bullet : \mathcal{I}^\bullet \rightarrow \mathcal{K}^\bullet[i+j]$.

\subsection{Conjugation}\label{subsec:conj}

The cohomology of profinite groups is functorial in pairs $(\xi,f)$ where $\xi : V' \rightarrow V$ is a continuous homomorphism of profinite groups and $f : M \rightarrow M'$ is a $k$-linear map between an $M$ in $\Mod(V)$ and an $M'$ in $\Mod(V')$ such that
\begin{equation*}
  f(\xi(g) m) = g f(m)  \qquad\text{for any $g \in V$ and $m \in M$}.
\end{equation*}
One method to construct the corresponding map on cohomology $(\xi,f)^* : H^i(V,M) \rightarrow H^i(V',M')$ proceeds as follows. We pick injective resolutions $M \xrightarrow{\sim} \mathcal{I}^\bullet_M$ in $\Mod(V)$ and $M' \xrightarrow{\sim} \mathcal{I}^\bullet_{M'}$ in $\Mod(V')$. By viewing, via $\xi$, $M \xrightarrow{\sim} \mathcal{I}^\bullet_M$ as a resolution in $\Mod(V')$ we see that $f$ extends to a homomorphism of complexes $f^\bullet : \mathcal{I}^\bullet_M \rightarrow \mathcal{I}^\bullet_{M'}$ such that
\begin{equation}\label{f:compatible}
  f^i(\xi(g) x) = g f^i(x)  \qquad\text{for any $i \geq 0$, $g \in V$, and $x \in \mathcal{I}^i_{M}$}.
\end{equation}
Then
\begin{align*}
  (\xi,f)^* : H^i(V,M) = \Hom_{K(V)}(k,\mathcal{I}^\bullet_M[i]) & \longrightarrow \Hom_{K(V')}(k,\mathcal{I}^\bullet_{M'}[i]) = H^i(V',M') \\
  \alpha^\bullet & \longmapsto f^\bullet[i] \circ \alpha^\bullet \ .
\end{align*}

We are primarily interested in the case where $M = k$ and $M' = k$ are the trivial representations, $f = \id_k$, and $\xi$ is an isomorphism. We simply write $\xi^* := (\xi,\id_k)^*$ in this case. We may then  take $\mathcal{I}^\bullet_{M'} := \xi^* \mathcal{I}^\bullet_k$ to be $\mathcal{I}^\bullet_k$ but with $V'$ acting through $\xi$ and $f^\bullet := \id_{\mathcal{I}^\bullet_k}$. Let $k \xrightarrow{\sim} \mathcal{I}^\bullet$ be another injective resolution in $\Mod(V)$. We have the commutative diagram
\begin{equation*}
  \xymatrix{
    \Hom_{K(V)}(k,\mathcal{I}^\bullet_k[i]) \ar[rr]^-{\alpha^\bullet  \mapsto \alpha^\bullet} && \Hom_{K(V')}(k,\xi^* \mathcal{I}^\bullet_k[i])  \\
    \Hom_{K(V)}(\mathcal{I}^\bullet,\mathcal{I}^\bullet_k[i]) \ar[u]^{\cong} \ar[rr]^-{\alpha^\bullet  \mapsto \alpha^\bullet} && \Hom_{K(V')}(\xi^*\mathcal{I}^\bullet,\xi^*\mathcal{I}^\bullet_k[i]). \ar[u]_\cong  }
\end{equation*}
In other words, if the cohomology class $\alpha \in H^i(V,k)$ is represented by the homomorphism of complexes $\alpha^\bullet : \mathcal{I}^\bullet \rightarrow \mathcal{I}^\bullet_k[i]$, then its image $\xi^* \alpha \in H^i(V',k)$ is represented by
\begin{equation}\label{f:conj0}
  \xi^* \alpha^\bullet := \alpha^\bullet : \xi^* \mathcal{I}^\bullet \rightarrow \xi^* \mathcal{I}^\bullet_k[i] \quad\text{(viewed as a $V'$-equivariant homomorphism via $\xi$)}.
\end{equation}

A specific instance of this situation is the following. Assume that the profinite group $V$ is a subgroup of some topological group $H$, let $h \in H$ be a fixed element, $V' := hVh^{-1}$, and $\xi :V' \rightarrow V$ be the isomorphism given by conjugation by $h^{-1}$. We then write
\begin{equation*}
  h_* = (h^{-1})^* := \xi^*: H^i(V,k) \longrightarrow H^i(hVh^{-1},k)
\end{equation*}
for the map on cohomology and $h_* \alpha^\bullet$ for the representing homomorphisms. We suppose now that $V$ is open in $H$, in which case there is the following alternative description. We choose injective resolutions $k \xrightarrow{\sim} \mathcal{I}^\bullet$ and $k \xrightarrow{\sim} \mathcal{J}^\bullet$ in $\Mod(H)$. Then they are also  injective resolutions in $\Mod(V)$ and $\Mod(V')$, so that we may take $\mathcal{I}^\bullet_k := \mathcal{J}^\bullet$. The map $f^\bullet : \mathcal{J}^\bullet \xrightarrow{\; h \cdot \;} \mathcal{J}^\bullet$ satisfies the condition \eqref{f:compatible}, and we obtain
\begin{align*}
   h_*: H^i(V,k) = \Hom_{K(V)}(k,\mathcal{J}^\bullet[i]) & \longrightarrow \Hom_{K(hVh^{-1})}(k,\mathcal{J}^\bullet[i]) = H^i(hVh^{-1},k) \\
  \alpha^\bullet & \longmapsto h \alpha^\bullet \ .
\end{align*}
This time one checks that the diagram
\begin{equation*}
  \xymatrix{
    \Hom_{K(V)}(k,\mathcal{J}^\bullet[i]) \ar[rr]^-{\alpha^\bullet  \mapsto h \alpha^\bullet} && \Hom_{K(hVh^{-1})}(k,\mathcal{J}^\bullet[i])  \\
    \Hom_{K(V)}(\mathcal{I}^\bullet,\mathcal{J}^\bullet[i]) \ar[u]^{\cong} \ar[rr]^-{\alpha^\bullet  \mapsto h \alpha^\bullet h^{-1}} && \Hom_{K(hVh^{-1})}(\mathcal{I}^\bullet,\mathcal{J}^\bullet[i]) \ar[u]_\cong  }
\end{equation*}
is commutative. We conclude that, if the cohomology class $\alpha \in H^i(V,k)$ is represented by the homomorphism of complexes $\alpha^\bullet : \mathcal{I}^\bullet \rightarrow \mathcal{J}^\bullet[i]$, then its image $h_* \alpha \in H^i(hVh^{-1},k)$ is represented by
\begin{equation}\label{f:conj1}
  {^{h_*} \alpha^\bullet} := h\alpha^\bullet h^{-1} : \mathcal{I}^\bullet \rightarrow \mathcal{J}^\bullet[i] \ .
\end{equation}

\subsection{The corestriction}\label{subsec:cores}

Let $U$ be a profinite group with open subgroup $V \subseteq U$ and let $M$ be in $\Mod(U)$. In this situation we have the corestriction map $\cores^V_U : H^*(V,M) \rightarrow H^*(U,M)$. It can be constructed as follows (cf.\ \cite{NSW} I\S5.4). Let $M \xrightarrow{\sim} \mathcal{I}^\bullet_M$ be an injective resolution in $\Mod(U)$. Then
\begin{align*}
   \cores^V_U : H^i(V,M) = \Hom_{K(V)}(k,\mathcal{I}^\bullet_M[i]) & \longrightarrow \Hom_{K(U)}(k,\mathcal{I}^\bullet_M[i]) = H(U,M) \\
  \alpha^\bullet & \longmapsto \sum_{g \in U/V} g \alpha^\bullet \ .
\end{align*}
For a variant of this, which we will need, let $k \xrightarrow{\sim} \mathcal{I}^\bullet$ be an injective resolution in $\Mod(U)$. One easily checks that the diagram
\begin{equation*}
  \xymatrix{
    \Hom_{K(V)}(k,\mathcal{I}^\bullet_M[i]) \ar[rrr]^-{\alpha^\bullet  \mapsto \sum_{g \in U/V} g\alpha^\bullet} &&& \Hom_{K(U)}(k,\mathcal{I}^\bullet_M[i])  \\
    \Hom_{K(V)}(\mathcal{I}^\bullet,\mathcal{I}^\bullet_M[i]) \ar[u]^{\cong} \ar[rrr]^-{\alpha^\bullet  \mapsto \sum_{g \in U/V} g\alpha^\bullet g^{-1}} &&& \Hom_{K(U)}(\mathcal{I}^\bullet,\mathcal{I}^\bullet_M[i]). \ar[u]_\cong  }
\end{equation*}
is commutative. This means that, if the cohomology class $\alpha \in H^i(V,M)$ is represented by the homomorphism of complexes $\alpha^\bullet : \mathcal{I}^\bullet \rightarrow \mathcal{I}^\bullet_M[i]$, then its image $\cores^V_U(\alpha) \in H^i(U,M)$ is represented by
\begin{equation}\label{f:cores}
  \sum_{g \in U/V} g\alpha^\bullet g^{-1} : \mathcal{I}^\bullet \rightarrow \mathcal{I}^\bullet_M[i] \ .
\end{equation}

\subsection{Basic properties\label{sec:basicprop}}

 For later reference we record from \cite{NSW} Prop.\ 1.5.4 that on cohomology restriction as well as corestriction commute with conjugation and from \cite{NSW} Prop.\ 1.5.3(iv) that the projection formulas
\begin{equation*}
  \cores^V_U(\alpha \cup \res^U_V(\beta)) = \cores^V_U(\alpha) \cup \beta \quad\text{and}\quad \cores^V_U(\res^U_V(\beta) \cup \alpha) = \beta \cup \cores^V_U(\alpha)
\end{equation*}
hold when $V$ is an open subgroup of the profinite group $U$ and $\alpha \in H^*(V,M)$, $\beta \in H^*(U,M)$.

\section{The product  in $E^*$}\label{sec:Yoneda-cup}

\subsection{A technical formula relating the Yoneda and cup products}\label{subsec:tech}

We fix classes $[\alpha^\bullet] \in H^i(I,\mathbf{X}(v))$ and $[\beta^\bullet] \in H^j(I,\mathbf{X}(w))$ represented by homomorphisms of complexes
\begin{equation*}
  \alpha^\bullet : \mathcal{J}^\bullet \longrightarrow \ind_I^{IvI}(\mathcal{J}^\bullet)[i]   \quad\text{and}\quad  \beta^\bullet : \mathcal{J}^\bullet \longrightarrow \ind_I^{IwI}(\mathcal{J}^\bullet)[j] \ ,
\end{equation*}
respectively. Here we always take $k \xrightarrow{\sim} \mathcal{I}^\bullet$ and $k \xrightarrow{\sim} \mathcal{J}^\bullet$ to be injective resolutions in $\Mod(G)$ (and hence in $\Mod(I)$). By \eqref{f:tilde} and \eqref{f:support} their Yoneda product $[\gamma^\bullet] := (-1)^{ij} [\alpha^\bullet] \cdot [\beta^\bullet]$ is represented by the homomorphism
\begin{align*}
  \gamma^\bullet : \mathcal{J}^\bullet & \longrightarrow \ind_I^{IvI \cdot IwI}(\mathcal{J}^\bullet)[i+j]   \\
  x & \longmapsto \sum_{g \in IvI/I} g \beta^\bullet[i] (\alpha^\bullet (x)(g)) \ .
\end{align*}
In fact, we introduce, for any $u \in \widetilde{W}$ such that $IuI \subseteq IvI \cdot IwI$, the homomorphism
\begin{equation*}
  \gamma^\bullet_u(-) := \gamma^\bullet(-)_{| IuI} : \mathcal{J}^\bullet \longrightarrow \ind_I^{IuI}(\mathcal{J}^\bullet)[i+j] \ .
\end{equation*}
Then
\begin{equation}\label{f:careful-ij}
  (-1)^{ij} [\alpha^\bullet] \cdot [\beta^\bullet] = \sum_{IuI \subseteq IvI \cdot IwI} [\gamma^\bullet_u] \ .
\end{equation}
Our goal here is to give a formula for the class $[\Sh'_{\dot{u}}(\gamma^\bullet_u)] = \Sh_u([\gamma^\bullet_u]) \in H^{* + i + j}(I_u,k)$ (cf.\ \eqref{f:Shapiro-compl}) in terms of group cohomological operations. We fix throughout a $u \in \widetilde{W}$ such that $IuI \subseteq IvI \cdot IwI$.

\begin{remark}\label{rem:bij}
The map
\begin{align*}
  \{a \in I/I_v : v^{-1} a u \in IwI\} & \xrightarrow{\; \simeq \;} I_{v^{-1}} \backslash (v^{-1} Iu \cap IwI) \\
  a & \longmapsto v^{-1} a^{-1} \dot{u}
\end{align*}
is a well defined bijection.
\end{remark}
\begin{proof}
The map is well defined since $v^{-1} I_v = I_{v^{-1}} v^{-1}$. It  is obviously  surjective. For injectivity suppose that $I_{v^{-1}} v^{-1} a \dot{u} = I_{v^{-1}} v^{-1} b \dot{u}$ for some $a,b \in I$. Then $v^{-1} I_v a^{-1} = v^{-1} I_v b^{-1}$ and hence $a I_v = b I_v$.
\end{proof}

Using the above formula for $\gamma^\bullet$ and Remark \ref{rem:bij} we compute
\begin{align}\label{f:tech1}
  \Sh'_{\dot{u}}(\gamma^\bullet_u)(x) & = \dot{u} (\gamma^\bullet_u(x)(\dot{u})) = \dot{u} (\gamma^\bullet(x)(\dot{u})) = \dot{u} \big( \big(\sum_{a \in I/I_v} av \beta^\bullet[i] (\alpha^\bullet(x)(av)) \big)(\dot{u}) \big)   \\  \nonumber
   & = \sum_{a \in I/I_v} \dot{u} ( \beta^\bullet[i] (\alpha^\bullet(x)(av)) (v^{-1} a^{-1} \dot{u}) )   \\  \nonumber
   & = \sum_{a \in I/I_v, v^{-1}a^{-1}u \in IwI} \dot{u} (\beta^\bullet[i] (\alpha^\bullet(x)(av)) (v^{-1} a^{-1} \dot{u}) )   \\   \nonumber
   & = \sum_{h \in I_{v^{-1}} \backslash (v^{-1} Iu \cap IwI)} \dot{u} (\beta^\bullet[i] (\alpha^\bullet(x)(\dot{u} h^{-1}))(h)) \ .
\end{align}
We fix, for the moment, an element $h \in v^{-1} Iu \cap IwI$ written as $h = c \dot{w} d = \dot{v}^{-1} a^{-1} \dot{u}$ with $a,c,d \in I$ and put
\begin{align*}
  \Gamma^\bullet_{\dot{u},h}(x) & := \dot{u} (\beta^\bullet[i] (\alpha^\bullet(x)(\dot{u} h^{-1}))(h)) = \dot{u} (\beta^\bullet[i] (\alpha^\bullet(x)(a \dot{v}))(c \dot{w} d)) \\
   & = \dot{u}h^{-1} ({^{c_*}\Sh'_{\dot{w}}(\beta^\bullet[i])} (\alpha^\bullet(x)(a \dot{v})))  \\
   & = \dot{u}h^{-1} ({^{c_*}\Sh'_{\dot{w}}(\beta^\bullet[i])} (\dot{v}^{-1}a^{-1} ({^{a_*}\Sh'_{\dot{v}}(\alpha^\bullet)}(x))))  \\
   & = \dot{u}h^{-1}\dot{v}^{-1} a^{-1} ({^{(a\dot{v})_* c_*}\Sh'_{\dot{w}}(\beta^\bullet[i])}  ({^{a_*}\Sh'_{\dot{v}}(\alpha^\bullet)}(x)))   \\
   & = {^{(a\dot{v}c)_*}\Sh'_{\dot{w}}(\beta^\bullet[i])}  ({^{a_*}\Sh'_{\dot{v}}(\alpha^\bullet)}(x))  \ .
\end{align*}
In the second and third line we have used the reconstruction formula \eqref{f:reconstruct} for $\beta^\bullet$ and $\alpha^\bullet$, respectively, as well as \eqref{f:conj1}.

\begin{lemma}\phantomsection\label{lem:bij}
\begin{itemize}
  \item[i.] In the commutative diagram of surjective projection maps
\begin{equation*}
  \xymatrix@R=0.5cm{
  I_{v^{-1}} \backslash (v^{-1} Iu \cap IwI) \ar[dd] \ar[dr]            \\
                & I_{v^{-1}} \backslash (v^{-1} Iu \cap IwI)/I_{u^{-1}} \ar[dl]         \\
  I_{v^{-1}} \backslash (v^{-1} IuI \cap IwI)/I                 }
\end{equation*}
  the lower oblique arrow is bijective.
  \item[ii.] For $h \in v^{-1} Iu \cap IwI$ the map $b \mapsto I_{v^{-1}} hb$ from the set $(I_{u^{-1}} \cap h^{-1} I h) \backslash I_{u^{-1}}$ to the fiber of the projection map $I_{v^{-1}} \backslash (v^{-1} Iu \cap IwI) \twoheadrightarrow I_{v^{-1}} \backslash (v^{-1} Iu \cap IwI)/I_{u^{-1}}$ in the point $I_{v^{-1}} h I_{u^{-1}}$ is a bijection.

\end{itemize}
\end{lemma}
\begin{proof}
i. Let $I_{v^{-1}} h I = I_{v^{-1}} h' I$ with $h = \dot{v}^{-1} a^{-1} \dot{u}$, $h' = \dot{v}^{-1} a'^{-1} \dot{u}$, and $a,a' \in I$. Then $I_v a^{-1} \dot{u} I = I_v a'^{-1} \dot{u} I$ and hence $a'^{-1} = A a^{-1} \dot{u} B \dot{u}^{-1}$ for some $A \in I_v$ and $B \in I$. It follows that $B \in I_{u^{-1}}$ and $h' = \dot{v}^{-1} A a^{-1} \dot{u} B \dot{u}^{-1} \dot{u} = (\dot{v}^{-1} A \dot{v}) hB \in I_{v^{-1}} h I_{u^{-1}}$.

ii. The equality $I_{v^{-1}} h b = I_{v^{-1}} h$ for some $b \in I_{u^{-1}}$ is equivalent to
\begin{equation*}
  b \in h^{-1}I_{v^{-1}} h = h^{-1} I h \cap h^{-1} v^{-1} I vh = h^{-1} I h \cap u^{-1} I u \ .
\end{equation*}
But the latter is equivalent to $b \in I \cap u^{-1} I u \cap h^{-1} I h = I_{u^{-1}} \cap h^{-1} I h$.
\end{proof}

Coming back to $\Gamma^\bullet_{\dot{u},h}$ we note that, for $b \in I_{u^{-1}}$, we have $hb = cw(db) = \dot{v}^{-1} (\dot{u} b^{-1} \dot{u}^{-1} a)^{-1} \dot{u}$, where $c, db, \dot{u} b^{-1} \dot{u}^{-1} a \in I$. It follows that
\begin{equation}\label{f:equiv}
  \Gamma^\bullet_{\dot{u},hb}(x) = \dot{u}b^{-1} \dot{u}^{-1} \Gamma^\bullet_{\dot{u},h}(\dot{u} b\dot{u}^{-1}x) \ .
\end{equation}
By inserting \eqref{f:equiv} into \eqref{f:tech1} and by using Lemma \ref{lem:bij} we obtain
\begin{align}\label{f:tech2}
  \Sh'_{\dot{u}}(\gamma^\bullet_u)(x) & = \sum_{h \in I_{v^{-1}} \backslash (v^{-1} Iu \cap IwI)} \Gamma^\bullet_{\dot{u},h}(x)    \\   \nonumber
   & = \sum_{h \in I_{v^{-1}} \backslash (v^{-1} Iu \cap IwI)/I_{u^{-1}}}  \sum_{b \in (I_{u^{-1}} \cap h^{-1} I h) \backslash I_{u^{-1}}} \dot{u} b^{-1} \dot{u}^{-1} \Gamma^\bullet_{\dot{u},h}(\dot{u} b\dot{u}^{-1}x)  \\ \nonumber
   & = \sum_{h \in I_{v^{-1}} \backslash (v^{-1} Iu \cap IwI)/I_{u^{-1}}}  \sum_{b \in (I_u \cap \dot{u}h^{-1} I h\dot{u}^{-1}) \backslash I_u}  b^{-1} \Gamma^\bullet_{\dot{u},h}(bx)  \ .
\end{align}
Above and in the following every summation over $h$ is understood to be over a chosen set of representatives in $G$ of the respective double cosets. It also follows from \eqref{f:equiv} that
\begin{equation*}
 b \Gamma^\bullet_{\dot{u},h}(-) = \Gamma^\bullet_{\dot{u},h}(b -)  \qquad\text{for any $b \in I_u \cap \dot{u}h^{-1} I h\dot{u}^{-1}$}.
\end{equation*}
This says that $\Gamma^\bullet_{\dot{u},h} : \mathcal{J}^\bullet \longrightarrow \mathcal{J}^\bullet[* + i + j]$ is a homomorphism of injective complexes in $\Mod(I_u \cap \dot{u}h^{-1} I h\dot{u}^{-1})$ and therefore defines a cohomology class $[\Gamma^\bullet_{\dot{u},h}] \in H^{* + i + j}(I_u \cap \dot{u}h^{-1} I h\dot{u}^{-1},k)$. By \eqref{f:cores} the equality \eqref{f:tech2} then gives rise on cohomology to the equality
\begin{equation}\label{f:tech3}
  \Sh_u([\gamma^\bullet_u]) = [\Sh'_{\dot{u}}(\gamma^\bullet_u)] = \sum_{h \in I_{v^{-1}} \backslash (v^{-1} Iu \cap IwI)/I_{u^{-1}}} \mathrm{cores}^{I_u \cap \dot{u}h^{-1} I h\dot{u}^{-1}}_{I_u} ([\Gamma^\bullet_{\dot{u},h}])
\end{equation}
in $H^{*+i+j}(I_u,k)$.

We recall that $h = c \dot{w} d = \dot{v}^{-1} a^{-1} \dot{u}$ with $a,c,d \in I$ and
\begin{align*}
  \Gamma^\bullet_{\dot{u},h}(x) & = {^{(a\dot{v}c)_*}\Sh'_{\dot{w}}(\beta^\bullet[i])}  ({^{a_*}\Sh'_{\dot{v}}(\alpha^\bullet)}(x)) \ .
\end{align*}
Note that both groups $a I_v a^{-1} = I \cap \dot{u}h^{-1} I h\dot{u}^{-1}$ and $(a\dot{v}c) I_w (a\dot{v}c)^{-1} = u I u^{-1} \cap \dot{u}h^{-1} I h\dot{u}^{-1}$ contain $I_u \cap \dot{u}h^{-1} I h\dot{u}^{-1}$; in fact the latter is the intersection of the former two. Therefore the above identity should, more precisely, be written as
\begin{equation*}
  \Gamma^\bullet_{\dot{u},h} = \res^{u I u^{-1} \cap \dot{u}h^{-1} I h\dot{u}^{-1}}_{I_u \cap \dot{u}h^{-1} I h\dot{u}^{-1}} \big({^{(a\dot{v}c)_*}\Sh'_{\dot{w}}(\beta^\bullet[i])}\big) \circ \res^{I \cap \dot{u}h^{-1} I h\dot{u}^{-1}}_{I_u \cap \dot{u}h^{-1} I h\dot{u}^{-1}} \big( {^{a_*}\Sh'_{\dot{v}}(\alpha^\bullet)} \big) \ .
\end{equation*}
Using Subsection \ref{subsec:cup-pr2} as well as \eqref{f:Shapiro-compl} we deduce that on cohomology classes we have the equality
\begin{equation}\label{f:tech4}
  [\Gamma^\bullet_{\dot{u},h}] = \res^{u I u^{-1} \cap \dot{u}h^{-1} I h\dot{u}^{-1}}_{I_u \cap \dot{u}h^{-1} I h\dot{u}^{-1}} \big((a\dot{v}c)_*\Sh_{w}([\beta^\bullet])\big) \cup \res^{I \cap \dot{u}h^{-1} I h\dot{u}^{-1}}_{I_u \cap \dot{u}h^{-1} I h\dot{u}^{-1}} \big( a_*\Sh_v([\alpha^\bullet]) \big)
\end{equation}
in $H^{* + i + j}(I_u \cap \dot{u}h^{-1} I h\dot{u}^{-1},k)$.

\begin{proposition}\label{prop:techn-formula}
For any cohomology classes $\alpha \in H^i(I,\mathbf{X}(v))$ and $\beta \in H^j(I,\mathbf{X}(w))$ we have $\alpha \cdot \beta = \sum_{u \in \widetilde{W}, IuI \subseteq IvI \cdot IwI} \gamma_u$ with $\gamma_u\in H^{i+j}(I,\mathbf{X}(u))$ and
\begin{equation*}
  \Sh_u(\gamma_u) = \sum_{h \in I_{v^{-1}} \backslash (v^{-1} Iu \cap IwI)/I_{u^{-1}}} \mathrm{cores}^{I_u \cap \dot{u}h^{-1} I h\dot{u}^{-1}}_{I_u}
   \big( \tilde{\Gamma}_{u,h}\big)\\
\end{equation*}
with

\begin{equation*}
    \tilde{\Gamma}_{u,h}:=  \res^{I \cap \dot{u}h^{-1} I h\dot{u}^{-1}}_{I_u \cap \dot{u}h^{-1} I h\dot{u}^{-1}} \big( a_*\Sh_v(\alpha)
   \big) \cup \res^{u I u^{-1} \cap \dot{u}h^{-1} I h\dot{u}^{-1}}_{I_u \cap \dot{u}h^{-1} I h\dot{u}^{-1}} \big((a\dot{v}c)_*\Sh_w(\beta)\big)  \ ,
\end{equation*}

where $h = c \dot{w} d = \dot{v}^{-1} a^{-1} \dot{u}$ with $a,c,d \in I$.
\end{proposition}
\begin{proof}
Insert \eqref{f:tech4} into \eqref{f:tech3}  and use the anticommutativity of the cup product together with \eqref{f:careful-ij}.
\end{proof}

\begin{remark}  Let $u,v,w\in \widetilde W$ such that  $IuI \subseteq IvI \cdot IwI$.  Suppose that
$I_{v^{-1}} \backslash (v^{-1} Iu \cap IwI)/I_{u^{-1}}$ contains a single element
$I_{v^{-1}} h I_{u^{-1}}$ and that
$I_u \subset \dot{u}h^{-1} I h\dot{u}^{-1}$.
Then
for any cohomology classes $\alpha \in H^i(I,\mathbf{X}(v))$ and $\beta \in H^j(I,\mathbf{X}(w))$ the component $\gamma_u$ of
$\alpha\cdot \beta$ in $H^{i+j}(I, \X(u))$ is such that
\begin{equation*}\Sh_u(\gamma_u) =  \res^{I \cap \dot{u}h^{-1} I h\dot{u}^{-1}}_{I_u } \big( a_*\Sh_v(\alpha)
   \big) \cup \res^{u I u^{-1} \cap \dot{u}h^{-1} I h\dot{u}^{-1}}_{I_u} \big((a\dot{v}c)_*\Sh_w(\beta)\big)
\end{equation*}with $a$ and $c$ as in the proposition.
If $i=0$ and $\alpha=\tau_v$ (resp. $j=0$ and $\beta=\tau_w$)
it is simply equal to   $\res^{u I u^{-1} \cap \dot{u}h^{-1} I h\dot{u}^{-1}}_{I_u} \big((a\dot{v}c)_*\Sh_w(\beta)\big)$
 (resp. $ \res^{I \cap \dot{u}h^{-1} I h\dot{u}^{-1}}_{I_u} \big(a_*\Sh_v(\alpha)\big)$) because $\tau_v$ (resp. $\tau_w$) corresponds to the constant function equal to $1$ in $ H^0 (I_v, k)$ (resp. $ H^0 (I_w, k)$). Therefore for general $\alpha$ and $\beta$  in the context of this remark,  the components of
$\alpha\cdot \beta$ and of $\alpha\cdot \tau_w\cup \tau_v\cdot \beta$ in $H^{i+j}(I, \X(u))$   coincide.
\label{rema:uniquecoset}
\end{remark}

\begin{corollary}\label{coro:prod-goodlength}
Let $v,w \in \widetilde W$ such that  $\ell(vw)=\ell(v)+\ell(w)$. For any  cohomology classes $\alpha \in H^i(I,\mathbf{X}(v))$ and $\beta \in H^j(I,\mathbf{X}(w))$ we have $\alpha\cdot \beta \in H^{i+j}(I, \X(vw))$, and

\begin{equation}\label{f:prod-goodlength}
 \alpha\cdot \beta = (\alpha\cdot \tau_w) \cup  (\tau_v\cdot \beta) \ ,
\end{equation}
where we use the cup product in the sense of subsection \ref{subsec:cup-prod}; moreover
\begin{equation}\label{f:prod-goodlength2}
  \Sh_{vw}(\alpha\cdot \tau_w) =
 \res^{I _v}_{I_{vw}} \big(\Sh_v(\alpha)
   \big)  \quad\textrm{ and }\quad  \Sh_{vw}( \tau_v\cdot \beta)   =
   \res^{ v I_w v^{-1}}_{I_{vw}} \big(v_*\Sh_w(\beta)\big)\ .
\end{equation}
\end{corollary}
\begin{proof} Note before starting the proof that, for $x\in \widetilde W$, we have $\Sh_x(\tau_x)= 1\in H^0(I_x,k)=k$. If $\ell(vw)=\ell(v)+\ell(w)$, then $IvI \cdot IwI= I vw I$  (cf.\ Cor.\ \ref{coro:known}.ii) and there is only $u = vw$ to consider in Prop.\ \ref{prop:techn-formula}. We have  $\alpha \cdot \tau_w\in H^{i}(I, \X(vw))$ and $\tau_v \cdot \beta\in H^{j}(I, \X(vw))$ and $\alpha\cdot \beta\in H^{i+j}(I,\X(vw))$ by Prop.\   \ref{prop:techn-formula}.

From Lemma \ref{lem:bij}.i we know that the projection map
\begin{equation*}
  I_{v^{-1}} \backslash (v^{-1}Ivw \cap IwI) / I_{(vw)^{-1}} \xrightarrow{\; \sim \;} I_{v^{-1}} \backslash (v^{-1}IvwI \cap IwI) / I
\end{equation*}
is a bijection. On the other hand we have the inclusions
\begin{equation*}
  IwI \subseteq I_{v^{-1}} wI \subseteq v^{-1}IvwI \cap IwI \ ,
\end{equation*}
the left one coming from \eqref{f:length1} and the right one being trivial. It follows that, in fact,
\begin{equation*}
  I_{v^{-1}} wI = v^{-1}IvwI \cap IwI \ .
\end{equation*}
Furthermore, it is straightforward to check
\begin{equation*}
  I_{v^{-1}} w I_{(vw)^{-1}} \subseteq v^{-1}Ivw \cap IwI \subseteq I_{v^{-1}} wI \ .
\end{equation*}
We claim that the left inclusion actually is an equality. Let $g \in v^{-1}Ivw \cap IwI$ be an arbitrary element. Using the second inclusion above we may find a $g_0 \in I_{v^{-1}} w \subseteq v^{-1}Ivw \cap IwI$ and a $g_1 \in I$ such that $g = g_0 g_1$. The above bijectivity then implies that necessarily $g_1 \in I_{(vw)^{-1}}$. We conclude that $g \in I_{v^{-1}} w I_{(vw)^{-1}}$. This proves that indeed
\begin{equation*}
  v^{-1}Ivw \cap IwI = I_{v^{-1}} w I_{(vw)^{-1}} \ .
\end{equation*}

Hence it suffices to consider $h = \dot{w}$. We have $I_u \cap \dot{u}h^{-1} I h\dot{u}^{-1} = I_u \cap I_v = I_{vw}$  by \eqref{f:length2}. Using Remark \ref{rema:uniquecoset}, we have proved \eqref{f:prod-goodlength}. Moreover, we have  $I \cap \dot{u}h^{-1} I h\dot{u}^{-1}= I_v$  and $u I u^{-1} \cap \dot{u} h^{-1} I h\dot{u}^{-1}=  v I_{w} v^{-1}$ for obvious reasons. Furthermore, we may take $c = 1$ and $a\in T^1$ such that $\dot u=  a \dot v\dot w $. Note that $a$ is contained in  both $ v I_w v^{-1}$ and  $I_v$ so its acts trivially on the cohomology spaces $H^i(I_v,k)$ and $ H^j(v I_w v^{-1},k)$.
Therefore in this case the formula of Proposition \ref{prop:techn-formula} gives:
\begin{equation*}
  \Sh_{vw}(\alpha\cdot \tau_w) =
 \res^{I _v}_{I_{vw}} \big(\Sh_v(\alpha)
   \big)  \quad \textrm{ and } \quad  \Sh_{vw}( \tau_v\cdot \beta)   =
   \res^{ v I_w v^{-1}}_{I_{vw}} \big(\dot{v}_*\Sh_w(\beta)\big)
\end{equation*}
and
\begin{equation*}
 \Sh_{vw}(\alpha\cdot \beta) =
    \Sh_{vw}( \alpha\cdot \tau_w) \cup \Sh_{vw}( \tau_v\cdot \beta) \ .
\end{equation*}
\end{proof}

\subsection{Explicit left action of $H$ on the  Ext-algebra\label{subsec:Hactions}}

Here we draw from Prop. \ref{prop:techn-formula} the formula for the explicit left  action of $H$ on $E^*$.   The proposition and its proof use notation introduced in \S\ref{subsubsec:Chevalley}. See in particular Remark \ref{not:bart}.

\begin{proposition}\label{prop:explicitleftaction}
Let $\beta\in H^j(I, \X(w))$ with $w\in \widetilde W$ and $j\geq 0$. For $\omega \in \widetilde \Omega$, we have $\tau_\omega	\cdot \beta\in H^j(I, \X(\omega w))$ and
\begin{equation}\label{f:leftomega}
  \Sh_{\omega w}(\tau_\omega	\cdot \beta)=\omega_*\Sh_w(\beta) \ .
\end{equation}
For $ s= s_{(\alpha, \mathfrak h)}\in S_{aff}$  we have either $\ell(\tilde s w)=\ell(w)+1$
and $\tau_{\tilde s} \cdot \beta\in H^j(I, \X(\tilde s w))$  with
\begin{equation}\label{f:leftsgood}
  \Sh_{\tilde s w}(\tau_{\tilde s}	\cdot \beta)=\res^{\tilde s I_w \tilde s^{-1}}_{I_{sw}} \big(\tilde s_* \Sh_w(\beta)\big)\ ,
\end{equation}
or  $\ell(\tilde s w)=\ell(w)-1$ and
\begin{equation}\label{f:leftsbad-support}
  \tau_{\tilde s}	\cdot \beta=\gamma_{\tilde s w}+\sum_{t\in\ima}\gamma_{\bar{t} w}\in H^j(I, \X(\tilde s w))\oplus \bigoplus_{t\in\ima} H^j(I, \X(\bar{t} w))
\end{equation}
with
\begin{equation}\label{f:leftsbad1}
  \Sh_{\tilde s w}(\gamma_{\tilde s w})=\cores_{I_{\tilde sw}}^{\tilde s I_w \tilde s^{-1}}\big(\tilde s_* \Sh_w(\beta)\big) \textrm{ and }
\end{equation}
\begin{equation}\label{f:leftsbad2}
\Sh_{\bar{t}  w}(\gamma_{\bar{t} w})=  \sum_{z\in \mathbb F_q^\times,\: \check\alpha([z])=t}(n_s t^{-1} x_{\alpha}(\pi^{\mathfrak h}[z]) n_s^{-1})_*\Sh_w(\beta)\ .
\end{equation}
(Note that  the statements above in the case $\ell(\tilde s w)=\ell(w)+1$ are true for any lift $\tilde s$ for $s$ in $\widetilde W$ whereas \eqref{f:leftsbad-support} and the subsequent formulas are valid only for the specific choice of $\tilde s$ made in \eqref{f:ns}.)
\end{proposition}
\begin{proof}
Except for the case of $\ell(\tilde s w)=\ell(w)-1$, the statements follow easily from
Corollary \ref{coro:prod-goodlength}, see in particular formula \eqref{f:prod-goodlength2}.
So we consider the remaining case $\ell(\tilde s w)=\ell(w)-1$ and recall that   there is $(\alpha,\mathfrak h)\in\Pi_{aff}$ such that $s=s_{(\alpha,\mathfrak h)}$ and that   $\tilde s= n_sT^1$ was defined in \eqref{f:ns}. Using the notation from \S\ref{subsubsec:Chevalley} (see in particular \eqref{f:sIsI}),  we have
\begin{equation*}
  n_s I n_s^{-1} I=I \;\dot\cup \; \dot\bigcup_{z\in\mathbb F_q^\times} x_{\alpha}(\pi^{\mathfrak h}[z]) \check\alpha([z]) n_s^{-1} I \subset  I  \;\dot\cup \;\dot\bigcup_{t\in\ima}I tn_s^{-1} I
\end{equation*}
and hence, using Cor.\ \ref{coro:known}.ii,
\begin{align}\label{f:quadraticw}
  n_s I  w I & =   n_s I  n_s^{-1}  I n_s \dot w I=
  I  n_s \dot w I\;\dot\cup \; \dot\bigcup_{z\in\mathbb F_q^\times} x_{\alpha}(\pi^{\mathfrak h}[z]) \check\alpha([z]) n_s^{-1} I n_s \dot w I \\
  & \subset  In_s \dot w I  \;\dot\cup \;\dot\bigcup_{t\in\ima}I  t  \dot w I \ .  \nonumber
\end{align}
This proves \eqref{f:leftsbad-support} using Proposition \ref{prop:techn-formula}. It remains to compute  $\gamma_{\tilde s w}$ and $\gamma_{\bar{t}w}$ for $t\in\ima$.

Let $u:= \tilde s w $. From \eqref{f:length1} we deduce  that $\tilde s ^{-1}I \tilde s  w\subset I_{s ^{-1}} w I$, therefore $\tilde s ^{-1}I\tilde s  w I\cap I w I= I_{ s^{-1}} w I$, and using  Lemma \ref{lem:bij}.i we see that $I_{s^{-1}} \backslash (\tilde s^{-1} I u \cap IwI)/I_{u^{-1}}$ is made of the single double coset $I_{s^{-1}} w  I_{u^{-1}}$. We have $I_u= I_{\tilde s w}$ and
\begin{equation*}
  I_u \cap u w^{-1} I w u^{-1} = I_{\tilde s w}\cap \tilde s I \tilde s^{-1}= \tilde s ( wIw^{-1}\cap \tilde s ^{-1} I \tilde s \cap I )\tilde s^{-1}=
 \tilde s ( I_w\cap \tilde s ^{-1} I \tilde s)\tilde s^{-1}= \tilde s I_w \tilde s^{-1} \ ,
\end{equation*}
where the last equality is justified by \eqref{f:length2}. Furthermore
\begin{equation*}
  u I u^{-1} \cap uw^{-1} I w u^{-1}=\tilde s ( w I w^{-1}\cap I)\tilde s^{-1}=\tilde s I_w \tilde s^{-1} \ .
\end{equation*}
So Proposition \ref{prop:techn-formula} says that the component $\gamma_{\tilde s w}$ in  $H^{i+j}(I, \X(\tilde s w))$ of $\tau_{\tilde s}\cdot \beta$  is given by
\begin{equation*}
  \Sh_{\tilde s w}(\gamma_{\tilde s w}) =  \mathrm{cores}^{\tilde s I_w \tilde s^{-1}}_{I_{\tilde{s} w}}
   \big( \res^{\tilde s I_w \tilde s^{-1}}_{\tilde s I_w \tilde s^{-1}} ({\tilde s}_*\Sh_w(\beta)\big) = \mathrm{cores}^{\tilde s I_w \tilde s^{-1}}_{I_{\tilde{s} w}}
   \big( {\tilde s}_*\Sh_w(\beta)\big) \ .
\end{equation*}

Let $t \in \ima$ and $u_t:= \bar{t} w$.  We pick $\dot u_t:=   t \dot w \in N(T)$. We have $n_s^{-1} Iu_t I\cap IwI=n_ s^{-1}(I t w  I\cap n_ s I w I).$ From \eqref{f:quadraticw}
we obtain that
\begin{equation*}
  I \bar{t} w  I\cap n_ s I w I=
  \dot\bigcup_{z\in\mathbb F_q^\times, \:\check\alpha([z])=t} x_{\alpha}(\pi^{\mathfrak h}[z]) t n_s^{-1} I n_s \dot w I= \dot\bigcup_{z\in\mathbb F_q^\times, \:\check\alpha([z])=t}  I_s x_{\alpha}(\pi^{\mathfrak h}[z]) t \dot w I \ .
\end{equation*}
The second equality comes from the fact that $t$ and $x_{\alpha}(\pi^{\mathfrak h}[z])$ normalize $I_s$ (see Cor.\ \ref{coro:known}.iii and Lemma \ref{lemma:UC} for the latter) and from  \eqref{f:length1}. Therefore
\begin{equation*}
  n_s^{-1} Iu_t I\cap IwI=
  \dot\bigcup_{z\in\mathbb F_q^\times, \:\check\alpha([z])=t}  I_{s^{-1}} n_s^{-1} x_{\alpha}(\pi^{\mathfrak h}[z]) t \dot w I \ .
\end{equation*}
Let $z\in\mathbb F_q^\times$ such that  $\check\alpha([z])=t$ and $h_{t, z}:= n_s^{-1} x_{\alpha}(\pi^{\mathfrak h}[z]) \dot u_ t $.  It lies in $n_ s^{-1} I \dot u _t\cap IwI$.
Using  Lemma \ref{lem:bij} and the above equalities, we obtain that $I_{s^{-1}} \backslash (n_ s^{-1} I \dot u _t\cap IwI)/I_{u_t^{-1}}$ is made of the (distinct) double cosets $I_{s^{-1}} h_{t,z}  I_{u_t^{-1}}$ where  $z\in\mathbb F_q^\times$ is such that  $\check\alpha([z])=t$. (By Remark \ref{rema:kernelalpha}, there is one or two such double cosets.) Furthermore, we have $I_{u_t}= I_w$ and $\dot u_t  h_{t,z} ^{-1} I h_{t,z} \dot u_t^{-1}=x_{\alpha}(\pi^{\mathfrak h}[z])^{-1}n_ s I n_ s^{-1}x_{\alpha}(\pi^{\mathfrak h}[z]) $.
Therefore
\begin{align*}
   I_{u_t}\cap \dot u_t  h_{t,z} ^{-1} I h_{t,z} \dot u^{-1}&= I\cap w I w^{-1}\cap x_{\alpha}(\pi^{\mathfrak h}[z]) ^{-1}n_ s I  n_s^{-1}x_{\alpha}(\pi^{\mathfrak h}[z]) \cr&=x_{\alpha}(\pi^{\mathfrak h}[z])  ^{-1} I_s x_{\alpha}(\pi^{\mathfrak h}[z])  \cap w I w^{-1}= I_s\cap w I w^{-1}=I_s\cap I_w= I_w= I_{u_t} \ ,
\end{align*}
where the third equality uses Cor.\ \ref{coro:known}.iii  and the fifth equality uses \eqref{f:length2}. Now to apply the formula of Prop. \ref{prop:techn-formula}, we  need to find $a_{t,z}$ and $c_{t,z}$ in $I$ such that $h_{t,z}= n_s^{-1} a_{t,z}^{-1} \dot u_t \in c_{t,z}\dot w I$ where $z\in\mathbb F_q^\times$ is such that $\check\alpha([z])=t$. Before giving them explicitly,  first notice that
$a_{t,z}n_s c_{t,z}$ lies in $\bar{t} wI w^{-1}$ thus it normalizes $wI w^{-1}$ and it also lies in $I n_s I$ thus normalizes $I_s$ by  Corollary \ref{coro:known}.iii. By \eqref{f:length2} we have $I_w= w I w^{-1}\cap I_s$ hence $a_{t,z}n_s c_{t,z}$ normalizes $I_w$ and it follows that
$(a_{t,z}n_sc_{t,z}) I_w (a_{t,z}n_sc_{t,z})^{-1} = u_t I u_t^{-1} \cap u_t h_{t,z}^{-1} I h_{t,z}{u_t}^{-1}$ coincides with $I_w=I_{u_t}$. Therefore, we have
\begin{equation*}
  \Sh_{u_t}(\gamma_{u_t})=\sum_{z\in \mathbb F_q^\times,\: \check\alpha([z])=t}(a_{t,z} n_sc_{t,z})_*\Sh_w(\beta) \ .
\end{equation*}
By the above definitions we  have  $a_{t,z} = x_{\alpha}(\pi^{\mathfrak h}[z])^{-1}$. To find a suitable element $c_{t,z}$, notice that
\begin{equation*}
  n_s x_{\alpha}(\pi^{\mathfrak h}[z^{-1}]) h_{t,z} \dot w^{-1} n_s^{-1}=n_s x_{\alpha}(\pi^{\mathfrak h}[z^{-1}])n_s^{-1} x_{\alpha}(\pi^{\mathfrak h}[z])  t n_s^{-1}= x_{\alpha}(\pi^{\mathfrak h}[-z])\in \EuScript U_{(\alpha, \mathfrak h)}
\end{equation*}
by \eqref{f:usefulequality}. By \eqref{f:conjU}, we have $\dot w ^{-1}  n_s^{-1} \EuScript U_{(\alpha, \mathfrak h)} n_s \dot w= \EuScript U_{(s \hat w)^{-1}(\alpha, \mathfrak h)}$  where $\hat w$ denotes the image of $w$ in $W$. By \eqref{add}, $(s \hat w)^{-1}(\alpha, \mathfrak h)$ lies in $\Phi_{aff}^+$ and from Lemma \ref{lemma:UC}.ii we deduce that $ \EuScript U_{(s \hat w)^{-1}(\alpha, \mathfrak h)}$ is contained in $I$. Therefore $\dot w^{-1} x_{\alpha}(\pi^{\mathfrak h}[z^{-1}]) h_{t,z}\in I$ and we may pick $c_{t,z}:=x_{\alpha}(\pi^{\mathfrak h}[z^{-1}]) ^{-1}$. Lastly using \eqref{f:usefulequality}, we see that  with this choice we have $a_{t,z} n_s c_{t,z}=  n_s t^{-1} x_{\alpha}(\pi^{\mathfrak h}[z]) n_s^{-1}$,
which concludes the proof of \eqref{f:leftsbad2}.
\end{proof}

\subsection{Appendix}\label{subsec:Ron}

In \cite{Ron} \S4.1 a groupoid cohomology class is a $G$-equivariant function
\begin{equation*}
  f : (G \times G)/(I \times I) \longrightarrow \oplus_{(g_1,g_2) \in G^2/I^2} H^*(g_1 Ig_1^{-1} \cap g_2 Ig_2^{-1}, k)
\end{equation*}
such that
\begin{itemize}
  \item[--] $f$ is supported on finitely many $G$-orbits, and
  \item[--] $f(g_1I,g_2I) \in H^*(g_1 Ig_1^{-1} \cap g_2 Ig_2^{-1}, k)$ for any $(g_1,g_2) \in G^2$.
\end{itemize}
The product of two such functions $f$ and $\tilde{f}$ is defined as follows. We use the maps
\begin{align*}
  \iota_{\mu,\nu} : G \times G \times G & \longrightarrow G \times G \\
                     (g_1,g_2,g_3) & \longmapsto (g_\mu,g_\nu) \ ,
\end{align*}
for $1 \leq \mu < \nu \leq 3$, in order to first introduce the pulled back functions
\begin{equation*}
  (\iota_{1,2}^* f)(g_1,g_2,g_3) := \res^{g_1 Ig_1^{-1} \cap g_2 Ig_2^{-1}}_{g_1 Ig_1^{-1} \cap g_2 Ig_2^{-1} \cap g_3 Ig_3^{-1}} f(g_1,g_2)
\end{equation*}
and
\begin{equation*}
  (\iota_{2,3}^* f)(g_1,g_2,g_3) := \res^{g_2 Ig_2^{-1} \cap g_3 Ig_3^{-1}}_{g_1 Ig_1^{-1} \cap g_2 Ig_2^{-1} \cap g_3 Ig_3^{-1}} f(g_2,g_3)
\end{equation*}
(with value in $H^*(g_1 Ig_1^{-1} \cap g_2 Ig_2^{-1} \cap g_3 Ig_3^{-1},k)$) on $G^3/I^3$. These functions are no longer supported on finitely many $G$-orbits. Nevertheless we consider their cup product
\begin{equation*}
  F(g_1,g_2,g_3) := (\iota_{1,2}^* f)(g_1,g_2,g_3) \cup (\iota_{2,3}^* f)(g_1,g_2,g_3) \ .
\end{equation*}
One can check (in fact, it follows from the subsequent computations) that its push forward
\begin{equation*}
  (\iota_{1,3*} F)(g_1,g_3) := \sum_{g \in (g_1 Ig_1^{-1} \cap g_3 Ig_3^{-1})\backslash G/I} \cores^{g_1 Ig_1^{-1} \cap g Ig^{-1} \cap g_3 Ig_3^{-1}}_{g_1 Ig_1^{-1} \cap g_3 Ig_3^{-1}} F(g_1,g,g_3)
\end{equation*}
is well defined and again is a groupoid cohomology class, which is defined in \cite{Ron} Prop.\ 7 to be the product $f \cdot \tilde{f}$.

We now fix two ``ordinary'' cohomology classes $\alpha \in H^i(I,\mathbf{X}(v))$ and $\beta \in H^j(I,\mathbf{X}(w))$ and introduce the corresponding groupoid cohomology classes $f_\alpha$ and $f_\beta$  supported on $G(1,v)I^2$ and $G(1,w)I^2$ by
\begin{equation*}
  f_\alpha(1,v) := \Sh_v(\alpha) \in H^i(I_v,k) \quad \text{and} \quad f_\beta(1,w) := \Sh_w(\beta) \in H^j(I_w,k) \ ,
\end{equation*}
respectively. In the following we compute the product $f_\alpha \cdot f_\beta$. By the $G$-equivariance it suffices to compute the classes
\begin{equation*}
  (f_\alpha \cdot f_\beta)(1,u) \in H^{i+j}(I_u,k)  \qquad\text{for}\ u \in \widetilde{W}.
\end{equation*}
We have the three injective maps between double coset spaces
\begin{align*}
  D_v : I_v \backslash G/I & \hookrightarrow G \backslash G^3/I^3 \\
                        h & \mapsto (1,v,h) ,  \\
  D_w : I_w \backslash G/I & \hookrightarrow G \backslash G^3/I^3 \\
                        h & \mapsto (h,1,w) , \ \text{and}  \\
  D_u : I_u \backslash G/I & \hookrightarrow G \backslash G^3/I^3 \\
                        h & \mapsto (1,h,u) .  \\
\end{align*}
The functions $\iota^*_{1,2} f_\alpha$ and $\iota^*_{2,3} f_\beta$ are supported on the $G$-orbits in $\im (D_v)$ and $\im (D_w)$, respectively. Hence $F := \iota^*_{1,2} f_\alpha \cup \iota^*_{2,3} f_\beta$ is supported on the $G$-orbits in $\im (D_v) \cap \im (D_w)$. Moreover, we have
\begin{equation}\label{f:Ron1}
  (f_\alpha \cdot f_\beta)(1,u) = (\iota_{1,3*} F)(1,u) = \sum_{h \in I_u \backslash G/I} \cores^{I_u \cap hIh^{-1}}_{I_u} F(1,h,u) \ .
\end{equation}
Of course, on the right hand side only those $h$ can occur for which we have $D_u(h) \in \im (D_v) \cap \im (D_w)$.

\begin{lemma}\label{intersecD}
$D_u^{-1}(\im (D_v) \cap \im (D_w)) = I_u \backslash (uIw^{-1}I \cap IvI)/I$.
\end{lemma}
\begin{proof}
Let $h = \dot{u} x_0 \dot{w}^{-1} x_1 = x_2 \dot{v} x_3$ with $x_j \in I$. Then
\begin{equation*}
  D_u(h) = G(1,x_2 \dot{v} x_3,\dot{u}) I^3 = G x_2 (1,\dot{v},x_2^{-1}\dot{u})(x_2^{-1},x_3,1)I^3 = G(1,\dot{v},x_2^{-1}\dot{u})I^3 \in \im(D_v)
\end{equation*}
and
\begin{align*}
  D_u(h) & = G(1,\dot{u} x_0 \dot{w}^{-1} x_1,\dot{u})I^3 = G \dot{u} x_0 \dot{w}^{-1}(\dot{w} x_0^{-1} \dot{u}^{-1},1,\dot{w})(1,x_1,x_0^{-1})I^3 \\
  & = G(\dot{w} x_0^{-1} \dot{u}^{-1},1,\dot{w})I^3 \in \im(D_w) \ .
\end{align*}
On the other hand, let $h \in G$ be such that $D_u(h) = G(1,h,u)I^3 \in \im (D_v) \cap \im (D_w)$. Then $(1,h,\dot{u}) \in G(1,v,\tilde{h})I^3 \cap G(\bar{h},1,w)I^3$ for some $\tilde{h}, \bar{h} \in G$. Write
\begin{equation*}
  (1,h,\dot{u}) = (gx_1,g\dot{v}x_2,g\tilde{h} x_3) = (g' \bar{h} y_1,g' y_2,g'\dot{w} y_3) \qquad\text{with $g, g' \in G$ and $x_j, y_j \in I$}.
\end{equation*}
We see that $1 = gx_1$ and $h = g\dot{v}x_2 = x_1^{-1} \dot{v} x_2 \in IvI$ as well as $\dot{u} = g' \dot{w}y_3$ and $h = g' y_2 = \dot{u} y_3^{-1} \dot{w}^{-1} y_2 \in uIw^{-1}I$.
\end{proof}

We deduce that \eqref{f:Ron1} simplifies to
\begin{equation*}
  (f_\alpha \cdot f_\beta)(1,u) = \sum_{h \in I_u \backslash(uIw^{-1}I \cap IvI)/I} \cores^{I_u \cap hIh^{-1}}_{I_u} F(1,h,u) \ .
\end{equation*}
Recall that
\begin{align*}
  F(1,h,u) & = (\iota_{1,2}^* f_\alpha)(1,h,u) \cup (\iota_{2,3}^* f_\beta)(1,h,u)  \\
    & = \res^{I \cap hIh^{-1}}_{I_u \cap hIh^{-1}} f_\alpha(1,h) \cup \res^{uIu^{-1} \cap hIh^{-1}}_{I_u \cap hIh^{-1}} f_\beta(h,u) \ .
\end{align*}
We write
\begin{equation*}
  h = \dot{u}A\dot{w}^{-1}B = C_h \dot{v}D \qquad\text{with $A, B, C_h, D \in I$}
\end{equation*}
and put $x_h := \dot{u}A\dot{w}^{-1}$. Then
\begin{equation*}
  (1,h) = C_h (1,\dot{v})(C_h^{-1},D) \quad\text{and}\quad (h,\dot{u}) = x_h (1,\dot{w})(B,A^{-1})
\end{equation*}
and hence, by $G$-equivariance,
\begin{equation*}
  f_\alpha(1,h) = C_{h*} \Sh_v(\alpha)  \quad\text{and}\quad  f_\beta(h,u) = x_{h*} \Sh_w(\beta) \ .
\end{equation*}
Inserting this into the above formulas we arrive at
\begin{multline}\label{f:Ron2}
  (f_\alpha \cdot f_\beta)(1,u) =  \\
  \sum_{h \in I_u \backslash(uIw^{-1}I \cap IvI)/I} \cores^{I_u \cap hIh^{-1}}_{I_u} \big( \res^{I \cap hIh^{-1}}_{I_u \cap hIh^{-1}} C_{h*} \Sh_v(\alpha)  \cup \res^{uIu^{-1} \cap hIh^{-1}}_{I_u \cap hIh^{-1}} x_{h*} \Sh_w(\beta) \big) \ .
\end{multline}

\begin{remark}\label{complements}
\begin{enumerate}
  \item[1.] $h I h^{-1}\cap uI u^{-1} =C_h v  I  v^{-1} C_h^{-1}\cap uI u^{-1} $.
  \item[2.] $C_h I_v C_h^{-1} = I \cap hIh^{-1} = I \cap x_h I x_h^{-1}$.
  \item[3.] $uIu^{-1} \cap hIh^{-1} = x_h I_w x_h^{-1} = x_h I x_h^{-1} \cap x_h wIw^{-1} x_h^{-1} = x_h I x_h^{-1} \cap uIu^{-1}$.
  \item[4.] The coset $C_h I_v$ only depends on the coset $hI$ (for $c,c'\in I$, we have $cvI = c' vI \Longleftrightarrow c^{-1} c' \in I_v$).
  \item[5.] The coset $x_h I_w$ only depends on the coset $hI$.
  \item[6.] $x_h I x_h^{-1} uI = x_h IwI$.
\end{enumerate}
\end{remark}

\begin{lemma}\phantomsection\label{lem:bij2}
\begin{itemize}
  \item[i.] The projection map
\begin{equation*}
  I_u \backslash (uIw^{-1} \cap IvI)/I_w \xrightarrow{\; \simeq\;} I_u \backslash (uIw^{-1}I \cap IvI)/I
\end{equation*}
is bijective.
  \item[ii.] The map
\begin{align*}
  I_{v^{-1}} \backslash (v^{-1} Iu \cap IwI) / I_{u^{-1}} & \xrightarrow{\simeq\;} I_u \backslash (uIw^{-1} \cap IvI)/I_w \\
  h = cwd & \longmapsto \dot{u} h^{-1} c
\end{align*}
  is bijective.
\end{itemize}
\end{lemma}
\begin{proof}
i. Replace in Lemma \ref{lem:bij}.i the elements $v^{-1}, u, w$ with $u, w^{-1}, v$.

ii. First of all we note that the coset $cI_w$ only depends on $w$ (and $h$). We therefore wrote $cwd = h$ by a slight abuse of notation. The map is well defined since $I_u u h^{-1} c = I_u u d^{-1} \dot{w}^{-1} \subseteq uIw^{-1}$ and, if $h = \dot{v}^{-1} a \dot{u}$, then $\dot{u} h^{-1} c = a^{-1} \dot{v} c \in IvI$. The bijectivity follows by checking that the map $h' = CvD \mapsto D h'^{-1} u$ is a well defined inverse.
\end{proof}

We note that if $h = c\dot{w}d = \dot{v}^{-1} a^{-1} \dot{u}$ then $h' := \dot{u}h^{-1}c = \dot{u} d^{-1} \dot{w}^{-1} = a\dot{v}c$ and $x_{h'} = a\dot{v}c$ and $C_{h'} = a$.
Hence, rewriting the right hand sum in \eqref{f:Ron2} by using the composite bijection in Lemma \ref{lem:bij2}, we obtain
\begin{multline*}
  (f_\alpha \cdot f_\beta)(1,u) =
  \sum_{h \in I_{v^{-1}} \backslash (v^{-1} Iu \cap IwI) / I_{u^{-1}}} \cores^{I_u \cap \dot{u}h^{-1}Ih\dot{u}^{-1}}_{I_u} \big( \res^{I \cap \dot{u}h^{-1}Ih\dot{u}^{-1}}_{I_u \cap \dot{u}h^{-1}Ih\dot{u}^{-1}} a_* \Sh_v(\alpha) \\
   \cup \res^{uIu^{-1} \cap \dot{u}h^{-1}Ih\dot{u}^{-1}}_{I_u \cap \dot{u}h^{-1}Ih\dot{u}^{-1}} (a\dot{v}c)_* \Sh_w(\beta) \big) \ ,
\end{multline*}
where $h = c \dot{w}d = \dot{v}^{-1} a^{-1} \dot{u}$ with $a, c, d \in I$. By comparing this equality with the equality in Prop.\ \ref{prop:techn-formula} we deduce the following result.

\begin{proposition}\label{prop:Ron}
Let $\alpha \in H^i(I,\mathbf{X}(v))$, $\beta \in H^j(I,\mathbf{X}(w))$, and $\alpha \cdot \beta = \sum_u \gamma_u$ with $\gamma_u \in H^{i+j}(I,\mathbf{X}(u))$ as in \ref{prop:techn-formula}; then $f_\alpha \cdot f_\beta = \sum_u f_{\gamma_u}$.
\end{proposition}

\section{An involutive anti-automorphism of the algebra $E^*$}\label{sec:Gamma}

For $w\in \widetilde{W}$, we have $I_{w^{-1}}=w^{-1} I_w w$ and a linear isomorphism
\begin{equation*}
  (w^{-1})_*:  H^i(I_w, k)\overset{\cong}\rightarrow  H^i(I_{w^{-1}}, k) \ ,
\end{equation*}
for all $i\geq 0$. Recall that conjugation by an element in $T^1\subset I_{w}$ is a trivial operator on $H^i(I_{w}, k)$ and therefore the conjugation above is well defined and does not depend on the chosen lift for $w^{-1}$ in $N(T)$.
 Via the Shapiro isomorphism \eqref{f:Shapiro1}, this induces the linear isomorphism  $\anti_w$:
\begin{equation}
  \xymatrix{
H^i(I, \X(w))\ar[d]_{\Sh_w} \ar[rrr]^{\anti_w}_{\cong} && & H^i(I, \X(w^{-1})) \ar[d]^{\Sh_{w^{-1}}}_\cong\\
H^i(I_w, k)\ar[rrr]^{(w^{-1})_*} _\cong& && H^i(I_{w^{-1}}, k)  }
\end{equation}
Summing over all $w\in \widetilde W$, the maps $(\anti_w)_{w\in \widetilde{W}}$ induce a linear isomorphism
\begin{equation*}
  \anti:  H^i(I,\X) \overset{\cong}\longrightarrow H^i(I,\X) \ .
\end{equation*}

\begin{proposition} \label{prop:anti+product}The map $\anti$  defines an involutive anti-automorphism of the
graded $\Ext$-algebra $E^*$, namely
$$\anti(\alpha\cdot \beta)=(-1)^{ij}\anti(\beta)\cdot \anti(\alpha)$$ where $\alpha\in H^i(I,\X)$ and $\beta\in H^j(I,\X)$ for all $i,j\geq 0$.
Restricted to $H^0(I, \X)$ it yields the anti-involution   $$\tau_g\mapsto \tau_{g^{-1}}\textrm{ for any $g\in G$}$$ of the algebra $H$. \end{proposition}

\begin{proof}  First note that, for $w\in\widetilde W$,  the element $\tau_w\in H^0(I, \X(w))=\X(w)^I$ corresponds to $1\in H^0(I_w,k)=k$. Therefore, $\anti(\tau_w)=\tau_{w^{-1}}$.
Now we turn to the proof of the first statement of the proposition.
Let $\alpha \in H^i(I,\mathbf{X}(v))$ and $\beta \in H^j(I,\mathbf{X}(w))$. On the one hand, recall that we have $$\alpha \cdot \beta= \sum_{u \in \widetilde{W}, IuI \subseteq IvI \cdot IwI}\gamma_u$$ with $\gamma_u\in H^{i+j}(I,\mathbf{X}(u))$  as in
 Proposition \ref{prop:techn-formula} given by
\begin{equation*}
  \Sh_u(\gamma_u) = { (-1)^{ij}} \sum_{h \in I_{v^{-1}} \backslash (v^{-1} Iu \cap IwI)/I_{u^{-1}}} \mathrm{cores}^{I_u \cap \dot{u}h^{-1} I h\dot{u}^{-1}}_{I_u}
   \big( \Gamma_{u,h}\big)\\
\end{equation*}
where
\begin{equation*}
\Gamma_{u,h}:=  \res^{u I u^{-1} \cap \dot{u}h^{-1} I h\dot{u}^{-1}}_{I_u \cap \dot{u}h^{-1} I h\dot{u}^{-1}} \big((a\dot{v}c)_*(\Sh_w(\beta))\big) \cup \res^{I \cap \dot{u}h^{-1} I h\dot{u}^{-1}}_{I_u \cap \dot{u}h^{-1} I h\dot{u}^{-1}} \big( a_*(\Sh_v(\alpha))
   \big).
\end{equation*}
where $h = c \dot{w} d = \dot{v}^{-1} a^{-1} \dot{u}$ with $a,c,d \in I$.

We compute $\anti(\beta) \cdot \anti(\alpha)$.  Recall that $\Sh_{v^{-1}}(\anti (\alpha))=(v^{-1})_* \Sh_v(\alpha)$ and $\Sh_{w^{-1}}(\anti (\beta))=(w^{-1})_* \Sh_w(\beta)$.
Note that the map $u\mapsto u^{-1}$ yields a bijection between the set of $u\in\widetilde W$ such that
 $I u I\subseteq I v I\cdot I w I$   and the set of $u'\in\widetilde W$
 such that $Iu'I \subseteq Iw^{-1}I \cdot Iv^{-1}I$. Therefore,
 by Proposition \ref{prop:techn-formula}, we have
  $\anti(\beta) \cdot \anti(\alpha)= \sum_{u \in \widetilde{W}, IuI \subseteq IvI \cdot IwI}\delta_{u^{-1}}$ with $\delta_{u^{-1}}\in H^{i+j}(I,\mathbf{X}(u^{-1}))$
given by
\begin{equation*}
  \Sh_{u^{-1}}(\delta_{u^{-1}}) = { (-1)^{ij}} \sum_{h' \in I_{w} \backslash (w Iu^{-1} \cap Iv^{-1}I)/I_{{u}}} \mathrm{cores}^{I_{u^{-1}} \cap \dot{u}^{-1}h'^{-1} I h'\dot{u}}_{I_{u^{-1}}}
   \big(\Delta_{u^{-1},h'}\big)
\end{equation*}
where
\begin{align*}
& \Delta_{u^{-1},h'}  \\
& =\res^{u^{-1} I {u} \cap \dot{u}^{-1}{h'}^{-1} I h'\dot{u}}_{I_{u^{-1}} \cap \dot{u}^{-1}{h'}^{-1} I h'\dot{u}} \big((a'\dot{w^{-1}}c')_*(\Sh_{v^{-1}}(\anti(\alpha)))\big) \cup \res^{I \cap \dot{u}^{-1}{h'}^{-1} I h'\dot{u}}_{I_{u^{-1}} \cap \dot{u}^{-1}{h'}^{-1} I h'\dot{u}} \big( {a'}_*(\Sh_{w^{-1}}(\anti(\beta)))
   \big)\cr&=\res^{u^{-1} I {u} \cap \dot{u}^{-1}{h'}^{-1} I h'\dot{u}}_{I_{u^{-1}} \cap \dot{u}^{-1}{h'}^{-1} I h'\dot{u}} \big((a'\dot{w^{-1}}c'\dot{v^{-1}})_*( \Sh_v(\alpha))\big) \cup \res^{I \cap \dot{u}^{-1}{h'}^{-1} I h'\dot{u}}_{I_{u^{-1}} \cap \dot{u}^{-1}{h'}^{-1} I h'\dot{u}} \big( {(a'\dot{w^{-1}}})_*( \Sh_w(\beta))
   \big)
\end{align*}
with $h' = c' \dot{v^{-1}} d' = (\dot{w^{-1}})^{-1} {a'}^{-1} \dot{u^{-1}}$  with $a',c',d' \in I$.

Now we compute $\anti(\anti(\beta)\cdot \anti(\alpha))$. Recall that corestriction commutes with conjugation (cf. \S\ref{sec:basicprop}). We have: \begin{equation*}
\Sh_{u}(\anti(\delta_{u^{-1}}))= u_* \Sh_{u^{-1}}(\delta_{u^{-1}})= { (-1)^{ij}}\sum_{h' \in I_{w} \backslash (w Iu^{-1} \cap Iv^{-1}I)/I_{{u}}} \mathrm{cores}^{I_{u} \cap h'^{-1} I h'}_{I_{u}} u_*
   \big(\Delta_{u^{-1},h'}\big)
\end{equation*}
To an element $h'\in w Iu^{-1} \cap Iv^{-1}I$  written in the form $h'={c'}\dot {v^{-1}} d'=(\dot{w^{-1}})^{-1} {a'}^{-1} \dot{u^{-1}}$  as above, we attach
the double coset $ I_{v^{-1}}  h  I_{u^{-1} }$ where
$h:= {c'}^{-1} h' \dot u= \dot {v^{-1}} d'\dot u \in v^{-1} I u \cap I w I$.
This is  well defined because $d$ is defined up to multiplication on the left by an element in $I_v$.
It is easy to see that this yields a map  $I_{w} \backslash (w Iu^{-1} \cap Iv^{-1}I)/I_{{u}}\rightarrow I_{v^{-1}} \backslash (v^{-1} Iu \cap IwI)/I_{u^{-1}}$.
One can check that the map in the opposite direction induced by  attaching  to $h= c\dot  w d\in
v^{-1} Iu \cap IwI$ the  double coset  $I_w  h' I_u$
where $h'=c^{-1} h \dot u^{-1}=\dot  w d\dot u^{-1} $ is well defined. Therefore,  these maps are bijective.
Note that for $h'$ and $h$ corresponding to each other as above, we have
  ${I_{u} \cap h'^{-1} I h'}= I_u\cap  \dot u h^{-1} I h {\dot u}^{-1}$
 and $h = c \dot{w} d = \dot{v}^{-1} a^{-1} \dot{u}$ with $a^{-1}\in T^1 d'$, $c^{-1}\in T^1 c'$ and $d^{-1}\in a' T^1$, therefore we compute
\begin{align*}
& u_*(\Delta_{u^{-1},h'})  \\
&= \dot u_* \res^{u^{-1} I {u} \cap  h^{-1} I h}_{I_{u^{-1}} \cap h^{-1} I h} \big((a'\dot{w^{-1}}c'\dot{v^{-1}})_*( \Sh_v(\alpha))\big) \cup \res^{I \cap {h^{-1}}I h}_{I_{u^{-1}} \cap {h}^{-1}I h} \big( {(a'\dot{w^{-1}}})_*( \Sh_w(\beta))
   \big)\cr&= \res^{ I \cap   \dot u h^{-1} I h \dot{u}^{-1}}_{I_{u^{-1}} \cap h^{-1} I h} \big(( \dot u a'\dot{w^{-1}}c'\dot{v^{-1}})_*( \Sh_v(\alpha))\big) \cup \res^{\dot u I \dot{u}^{-1}\cap {\dot u h^{-1}}I h\dot{u}^{-1}}_{I_{u} \cap \dot u {h}^{-1}I h\dot u^{-1}} \big( {( \dot u a'\dot{w^{-1}}})_*( \Sh_w(\beta))
   \big)\cr&= \res^{ I \cap   \dot u h^{-1} I h \dot{u}^{-1}}_{I_{u^{-1}} \cap h^{-1} I h} \big((d'^{-1})_*( \Sh_v(\alpha))\big) \cup \res^{\dot u I \dot{u}^{-1}\cap {\dot u h^{-1}}I h\dot{u}^{-1}}_{I_{u} \cap \dot u {h}^{-1}I h\dot u^{-1}} \big(  {( \dot u a'\dot{w^{-1}}})_*( \Sh_w(\beta))
   \big)\cr &= \res^{ I \cap   \dot u h^{-1} I h \dot{u}^{-1}}_{I_{u^{-1}} \cap h^{-1} I h} \big(a_*( \Sh_v(\alpha))\big) \cup \res^{\dot u I \dot{u}^{-1}\cap {\dot u h^{-1}}I h\dot{u}^{-1}}_{I_{u} \cap \dot u {h}^{-1}I h\dot u^{-1}} \big(  (a \dot v c)_*( \Sh_w(\beta))
   \big)\cr&= (-1)^{ij} \Gamma_{u,h}
\end{align*}
and
$\Sh_{u}(\anti(\delta_{u^{-1}}))=(-1)^{ij}\Sh_{u}(\gamma_u)$. We proved $\anti(\anti(\beta)\cdot \anti(\alpha))=(-1)^{ij} \alpha\cdot \beta.$
\end{proof}

\begin{remark}\label{rema:Jcup}
By \eqref{f:cup+Sh} our cup product commutes with the Shapiro isomorphism. It also commutes with conjugation of the group (\cite{NSW} Prop.\ 1.5.3(i)). Therefore the anti-automorphism $\anti$ respects the cup product.
\end{remark}

\begin{remark}\label{rema:anti+ss}
We want to show that the anti-involution of $H$ induced by $\anti$ preserves the ideal $\mathfrak J$ of \S\ref{subsec:supersing}. Consider the space
$ {\mathbb Z}[G/I]$ of  finitely supported functions $G/I\rightarrow \mathbb Z$. The  ring of its $G$-equivariant ${\mathbb Z}$-endomorphisms is isomorphic to the convolution ring
$ {\mathbb Z}[I\backslash  G/I]$ with product given by $f\star f'({}_-)=\sum_{x\in G/I} f(x^{-1}{}_-)f'(x)$.
The opposite ring is denoted by $H_{\mathbb Z}$.
One easily checks that the map $f\mapsto [g\mapsto f(g^{-1})]$ defines an anti-involution $j_{\mathbb Z}$ of the convolution ring  $ {\mathbb Z}[I\backslash  G/I]$. It induces an anti-involution of the $k$-algebra   $H= H_{\mathbb Z}\otimes  k$ which coincides with $\anti$.
A basis $(z_{\{\lambda\}})_{\{\lambda\}\in \widetilde \Lambda/ W_0}$ of the center of the ring $H_{\mathbb Z}$ is described in  \cite{Vigprop}. It is indexed by the set of  $W_0$-orbits in the preimage $\widetilde \Lambda= T/T^1$ of $\Lambda= T/T^0$ in $\widetilde W$ (see \S\ref{pro-p Weyl}).
From  \cite{Oll} Lemma 3.4 we deduce that $j_{\mathbb Z}(z_{\{\lambda\}})= z_{\{\lambda^{-1}\}}$  for any $\lambda\in \widetilde \Lambda$ where $\{\lambda^{-1}\}$ denotes the $W_0$-orbit of $\lambda^{-1}$.

Recall that the map $\nu$ defined in \S\ref{subsubsec:apartment} induces an isomorphism $\Lambda\cong X_*(T)$.
As in \cite{Oll} 1.2.6, notice that the map  $X_*(T) \rightarrow T/T^1$, $\xi\mapsto \xi(\pi^{-1})\bmod T^1$ composed with $\nu$ splits the exact sequence
\begin{equation*}
  0\longrightarrow T^0/T^1 \longrightarrow  \widetilde \Lambda \longrightarrow \Lambda \longrightarrow 1
\end{equation*}
and we may see $\Lambda$ as a subgroup of $\widetilde\Lambda$ which  is  preserved by the action of $W_0$.
The set  $\Lambda/W_0$  of all $W_0$-orbits of in $\Lambda$ contains the set $(\Lambda/W_0)'$ of orbits of elements with nonzero length (when seen in $W$). Note that via the map $\nu$, it is indexed by the set $X_*^{dom}(T)\setminus (-X_*^{dom}(T))$. By the above remarks, the $\mathbb Z$-linear subspace   of $H_{\mathbb Z}$ with basis $(z_{\{\lambda\}})_{\{\lambda\}\in  (\Lambda/ W_0)'}$  is preserved by $j_{\mathbb Z}$.
Therefore, its image $\mathfrak J$ in $H$ (as defined in \cite{Oll} 5.2) is preserved by $\anti$. Note in passing that the algebra  $\mathcal Z^0(H)$ (as introduced in \S\ref{subsec:supersing})  has basis $(z_{\{\lambda\}})_{\{\lambda\}\in  \Lambda/ W_0}$.

We deduce from this that if $M$ is a left (resp. right) supersingular module in the sense of \S\ref{subsec:supersing}, namely if any element of $M$ is annihilated by a power of $\mathfrak J$, then $M^\anti$ (resp. ${}^\anti M$) is a right (resp. left) supersingular module.

Alternatively, if $\mathbf{G}$ is semisimple, one can argue that $\anti$ preserves supersingular modules as follows. First of all it suffices to show this for cyclic supersingular modules  and such modules have finite length. (This is because  $H$ is finitely generated over $\mathcal Z^0(H)$  by \cite{Oll} Prop.\ 2.5ii and since  $\mathbf{G}$ is semisimple, $\mathfrak J$ has finite codimension in  $\mathcal Z^0(H)$ and therefore in $H$.) It then further suffices to do this after a suitable extension of the coefficient field. By the equivalence in the proof of Lemma \ref{lemma:critsupersing} this finally reduces us to quotients of  $H\otimes_{H_{aff}}\chi$ (or $\chi\otimes_{H_{aff}}H$) where $\chi$ is a supersingular character of $H_{aff}$. It is easy to see that the composite $\chi\circ\anti:H_{aff}\rightarrow k$  is also a supersingular character. The right (resp. left) $H$-module $(H\otimes_{H_{aff}}\chi)^\anti$
(resp. ${}^\anti(\chi\otimes_{H_{aff}} H)$) is generated as an $H$-module by $1\otimes 1$ which supports the character $\chi\circ\anti$ of $H_{aff}$.  Therefore by Lemma \ref{lemma:critsupersing} it is annihilated by $\mathfrak J$ and supersingular as an $H$-module.
\end{remark}

\section{Dualities}\label{sec:duality}

\subsection{Finite and twisted duals\label{subsec:duals}}
Given a vector space $Y$, we denote by $Y^\vee$ the dual  space $Y^\vee:=\Hom_k(Y, k)$ of $Y$.
If $Y$ is a left, resp. right, module over $H$, then $Y^\vee$ is naturally a
right, resp. left, module over $H$. Recall that $H$ is endowed with an anti-involution
respecting the product and given by the map $\anti$  (see  Proposition \ref{prop:anti+product}).
We may twist the action of $H$ on a left, resp. right, module $Y$ by $\anti$ and thus obtain the right, resp. left module $Y^\anti$, resp. ${}^\anti Y$, with the twisted action of $H$ given by  $(y, h)\mapsto  \anti(h)y$, resp. $(h,y)\mapsto y \anti(h)$. If $Y$ is an $H$-bimodule, then we may define the twisted $H$-bimodule
${}^\anti Y^\anti$ the obvious way.

\begin{remark} \label{etavee}For a left, resp right, resp. bi-, $H$-module,
the identity map yields an isomorphism of right, resp. left,  resp. bi-,  $H$-modules  $$({}^\anti Y)^\vee=(Y^\vee)^{\anti},\quad \textrm{resp.  }(Y^\anti)^\vee={}^\anti(Y^\vee), \quad \textrm{resp.  }({}^\anti Y^\anti)^\vee={}^\anti(Y^\vee)^\anti.$$
\end{remark}

Since $\X$ is a right $H$-module, the space $\X^\vee$ is naturally a left $H$-module via $(h,\varphi)\mapsto \varphi({}_-h)$.  It is also endowed with a left action of $G$ which commutes with the action of $H$ via $(g,\varphi)\mapsto \varphi(g^{-1}{}_-)$. It is however not a smooth representation of $G$.
Since $\X$ decomposes into $\oplus_{w\in \widetilde W } \X(w)$ as a vector space, $\X^\vee$  identifies with $\prod_{w\in \widetilde W } \X(\omega)^\vee$ which contains $\oplus_{w\in \widetilde W} \X(w)^\vee$. We denote by $\X^{\vee,f}$ the image of the latter in $\X^\vee$. It is  stable under the action of $G$ on $\X^\vee$, and $\X^{\vee,f}$  is a smooth representation of $G$. Moreover it follows from Cor.\ \ref{coro:known}.ii and \eqref{f:quadratic2} that $\X^{\vee,f}$ is an $H$-submodule of $\X^\vee$.

More generally, for $Y$ a vector space which decomposes into a direct sum $Y = \oplus_{w\in\widetilde W} Y_w$, we denote by $Y^{\vee, f}$ the so-called finite dual of $Y$ which is defined to be the image in $Y^\vee = \prod_{w\in \widetilde W } Y_w^\vee$ of
$\oplus_{w\in \widetilde W } Y_w^\vee$.

For $g\in G$ denote by   $\ev_{g}$ the evaluation map $ \mathbf X\rightarrow k, f\mapsto f(g)$.  This is an element  in $\mathbf X^{\vee, f}$. For  $g_0,g\in G$ and $f\in \mathbf X$ we have
$({^{g_0}\ev}_{g})(f) = \ev_{g}({^{g_0^{-1}}f}) = ({^{g_0^{-1}}f})(g)= f(g_0g) = \ev_{g_0 g}(f)$. In particular, $\ev_1\in \mathbf X^{\vee, f}$ is fixed under the action of $I$ and there is a well defined morphism of smooth representations of $G$:
\begin{align}\label{f:contra}
    \mathbf{ev} : \X & \longrightarrow \X^{\vee, f} \\
      \mathrm{ char}_{gI} = {^g \chara}_I & \longmapsto  \ev_g = {^g \ev}_1 \quad{\textrm{ for any $g\in G$}}.  \nonumber
\end{align}
It is clearly a bijection.  The basis $(\ev_g)_{g\in G/I}$ of $\X^{\vee, f}$ is dual to the basis $(\chara_{gI})_{g\in G/I}$ of $\X$.
For $w\in \widetilde W$, the space $\X(w)$ corresponds to $\X(w)^\vee$ under the  isomorphism $\mathbf{ev}$.

\begin{lemma}\label{lemma:selfcontra}
  The map $\mathbf{ev}$ induces an isomorphism of right $H$-modules $\mathbf X \xrightarrow{\cong} (\X^{\vee,f})^\anti$.
\end{lemma}

\begin{proof}
We only need to show that the composite map $\mathbf{X} \xrightarrow{\mathbf{ev}} (\X^{\vee,f})^\anti \xrightarrow{\subseteq} (\X^{\vee})^\anti$ is right $H$-equivariant.
Since  $\X$ and $(\X^\vee)^\anti$ are $(G, H)$-bimodules and  since $\X$ is generated by $\mathrm{char}_I$ under the action of $G$, it is enough to prove that  $\mathbf{ev}((\mathrm{char}_I) \tau)= \anti(\tau)\mathbf{ev}(\mathrm{char}_I)$ for any
$\tau \in H$, or equivalently, that
\begin{equation*}
  \mathbf{ev}(\mathrm{char}_{Iw I})= \tau_{w^{-1}}\mathbf{ev}(\mathrm{char}_I)
\end{equation*}
for any $w\in \widetilde W$. Decompose $IwI$ into simple cosets $I w I=\sqcup_{x}  x w I$ for
$x\in  I\cap w I w^{-1}\backslash I$. Then on the one hand, $\mathbf{ev}(\mathrm{char}_{IwI})=\sum_ x \ev_{xw}$.  For $g\in G$, it sends the function $\mathrm{char}_{gI}$ onto $1$ if and only if $g\in I w I$ and to $0$ otherwise.
On the other hand, we have $(\tau_{w^{-1}}\ev_1)(\mathrm{char}_{gI})=\ev_1(\mathrm{char}_{gI}\mathrm{char}_{Iw^{-1}I})=\ev_1(\mathrm{char}_{gIw^{-1}I})$. It is equal to $1$ if and only if $g^{-1}\in I w^{-1} I$ and to $0$ otherwise. This  proves the lemma.
\end{proof}

\subsection{\label{subsec:d-i} Duality between $E^i$ and $E^{d-i}$ when $I$ is a Poincar\'e group of dimension $d$}

In this  section we always \textbf{assume} that the pro-$p$ Iwahori group $I$ is torsion free. This forces the field $\mathfrak{F}$ to be a finite extension of $\mathbb{Q}_p$ with $p \geq 5$. Then $I$ is a Poincar\'e group of dimension $d$ where $d$ is the dimension of $G$ as a $p$-adic Lie group: According to \cite{Laz} Thm.\ V.2.2.8 and \cite{Ser} the group $I$ has finite cohomological dimension; then \cite{Laz} Thm.\ V.2.5.8 implies that $I$ is a Poincar\'e group of dimension $d$. Any open subgroup of a Poincar\'e group is a Poincar\'e group of the same dimension (cf.\ \cite{S-CG} Cor.\ I.4.5). This applies to our groups $I_w$ for any $w \in \widetilde{W}$. It follows that
\begin{equation*}
  E^* = H^*(I,\mathbf{X}) = 0 \qquad\text{for $* > d$}
\end{equation*}
and that
\begin{equation} \label{f:1dim}
  H^d(I,\mathbf{X}(w)) \cong H^d(I_w,k) \ \text{is one dimensional for any $w \in \widetilde{W}$}.
\end{equation}

\begin{remark}\label{rem:cores}
Let $L$ be a proper open subgroup of $I$. By \cite{S-CG} Chap.\ 1 Prop.\ 30(4) and Exercise 5) respectively, $\cores_I^L: H^d(L, k) \rightarrow  H^d(I,k)$ is a linear isomorphism while $\res^I_L : H^d(I,k)\rightarrow  H^d(L, k)$ is  the zero map.
\end{remark}

Let $\trace\in \mathbf X^\vee$ be the linear map given by
\begin{equation}\label{f:trace}\trace:= \sum_{g\in G/I} \ev_g\ .\end{equation}
It is easy to check that $\trace:\X\rightarrow k$ is $G$-equivariant  when $k$ is endowed with the trivial action of $G$. We denote by $\trace^i := H^i(I, \trace)$ the maps induced on cohomology.

\begin{remark}\label{rema:tracew}
We may decompose $\mathbf \trace=\sum_{w\in \widetilde W} \trace_w$ where $\trace_w=\sum_{g\in IwI/I} \ev_g$. Each summand $\trace_w: \X\rightarrow k$ is $I$-equivariant and  $\trace_w\vert_{\X(v)}=0$ if $v\neq w \in \widetilde W$ and the following diagram is commutative:
\begin{equation*}
  \xymatrix@R=0.5cm{
                &         H^i(I,\X(w)) \ar[dd]^{H^i(I,\trace_w)} \ar[dl] _{\Sh_w}    \\
  H^i(I_w, k)     \ar[dr] _{\cores_{I}^{I_w}}              \\
                &         H^i(I,k)                  }
\end{equation*}
\end{remark}
\begin{proof}
 We contemplate the larger diagram
\begin{equation*}
  \xymatrix{
    H^i(I,\mathbf{X}(w)) \ar[dd]_{\Sh_w} \ar[rr]^{H^i(I,\trace_w)} && H^i(I,k) &  \\
 & H^i(I_w,\mathbf{X}(w)) \ar[lu]_-{\cores^{I_w}_I} \ar[rr]^{H^i(I_w,\trace_w)} & & H^i(I_w,k) \ar[lu]_{\cores^{I_w}_I} \\
    H^i(I_w,k). \ar[ru]^{?} \ar[rrru]_-{=} &  &
     }
\end{equation*}
Here the map $?$ is induced by the map between coefficients which sends $a \in k$ to $a \chara_{wI}$.
The parallelogram is commutative since the corestriction is functorial in the coefficients. The right lower triangle is commutative since $\trace_w (a \chara_{wI}) = a$. The left triangle is commutative since the composite of the upwards pointing arrows is the inverse of the Shapiro isomorphism by \eqref{f:Shapiro-inverse}.
\end{proof}

\begin{lemma}\label{lemma:nondegenerate}
For $0\leq i\leq d$, the  bilinear map defined by the composite
\begin{equation*}
  H^i(I, \mathbf{X}) \otimes_k H^{d-i}(I,\mathbf{X})  \xrightarrow{\cup} H^{d}(I,\mathbf{X}) \xrightarrow{\trace^d
} H^{d}(I,k)\cong k
\end{equation*}
is nondegenerate.
\end{lemma}

\begin{proof}
Let $w\in \widetilde{W}$. We consider the diagram
\begin{equation*}
  \xymatrix{
     H^i(I, \mathbf{X}(w)) \otimes_k H^{d-i}(I,\mathbf{X}(w)) \ar[d]_\cong^{\Sh_w \otimes \Sh_w} \ar[r]^-{\cup} & H^{d}(I,\mathbf{X}(w)) \ar[d]_\cong^{\Sh_w} \ar[rr]^{H^d(I,\trace_w)} && H^{d}(I,k)  \ar@{=}[d]\\     H^i(I_w,k) \otimes_k H^{d-i}(I_w,k) \ar[r]^-{\cup} & H^d(I_w,k)  \ar[rr]_\cong^-{\cores^{I_w}_I} && H^{d}(I,k),
     }
\end{equation*} where the lower right corestriction map is an isomorphism by Remark \ref{rem:cores}.
The left square is commutative by \eqref{f:cup+Sh} and the right one by Remark \ref{rema:tracew}. The lower pairing is nondegenerate since $I_w$ is a Poincar\'e group of dimension $d$. Therefore,  the top horizontal composite induces a perfect pairing. Using \eqref{f:orth}, this proves the lemma.
\end{proof}

\subsubsection{Congruence subgroups}\label{subsubsec:congruence}

We consider a smooth affine group scheme $\mathcal{G} = \Spec(A)$ over $\mathfrak{O}$ of dimension $\delta$. In particular, $A$ is an $\mathfrak{O}$-algebra via a homomorphism $\alpha : \mathfrak{O} \rightarrow A$. A point $s \in \mathcal{G}(\mathfrak{O})$ is an $\mathfrak{O}$-algebra homomorphism $s : A \rightarrow \mathfrak{O}$; it necessarily  satisfies $s \circ \alpha = \id$. The reduction map is
\begin{align*}
  \mathcal{G}(\mathfrak{O}) & \longrightarrow \mathcal{G}(\mathfrak{O}/\pi \mathfrak{O}) \\
  s & \longmapsto \bar{s} := [A \xrightarrow{s} \mathfrak{O} \xrightarrow{\pr} \mathfrak{O}/\pi \mathfrak{O}] \ .
\end{align*}
Let $\epsilon : A \rightarrow \mathfrak{O}$ denote the unit element in $\mathcal{G}(\mathfrak{O})$; then $\bar{\epsilon}$ is the unit element in $\mathcal{G}(\mathfrak{O}/\pi \mathfrak{O}$.

Let $\mathfrak{p} := \ker(\epsilon)$. The formal completion $\widehat{\mathcal{G}}$ of $\mathcal{G}$ in the unit section is the formal group scheme $\widehat{\mathcal{G}} := \mathrm{Spf}(\widehat{A}^\mathfrak{p})$ where $\widehat{A}^\mathfrak{p}$ is the $\mathfrak{p}$-adic completion of $A$. By our smoothness assumption $\widehat{A}^\mathfrak{p} = \mathfrak{O}[[X_1,\ldots,X_\delta]]$ is a formal power series ring in $\delta$ many variables $X_1,\ldots,X_\delta$. A point in $\widehat{\mathcal{G}}(\mathfrak{O})$ is a point $s : A \rightarrow \mathfrak{O}$ in $\mathcal{G}(\mathfrak{O})$ which extends to a continuous homomorphism $s : \widehat{A}^\mathfrak{p} \rightarrow \mathfrak{O}$, i.e., which satisfies $s(\mathfrak{p}) \subseteq \mathfrak{M}$, or equivalently, $\bar{s}(\mathfrak{p}) = 0$. One checks that $\bar{s}(\mathfrak{p}) = 0$ if and only if $\bar{s} = \bar{\epsilon}$. This shows that
\begin{equation*}
  \widehat{\mathcal{G}}(\mathfrak{O}) = \ker \Big(\mathcal{G}(\mathfrak{O}) \xrightarrow{\mathrm{reduction}} \mathcal{G}(\mathfrak{O}/\pi \mathfrak{O}) \Big) \ .
\end{equation*}
On the other hand we have the bijection
\begin{align*}
  \xi : \widehat{\mathcal{G}}(\mathfrak{O}) & \xrightarrow{\;\simeq\;} \mathfrak{M}^\delta  \\
        s & \longmapsto (s(X_1),\ldots s(X_\delta)) \ .
\end{align*}
We see that $\widehat{\mathcal{G}}(\mathfrak{O})$ is a standard formal group in the sense of \cite{S-LL} II Chap.\ IV \S8. We then have in $\widehat{\mathcal{G}}(\mathfrak{O})$ the descending sequence of normal subgroups
\begin{equation*}
  \widehat{\mathcal{G}}_m(\mathfrak{O}) := \xi^{-1}((\pi^m\mathfrak{O})^\delta)  \qquad\text{for $m \geq 1$}
\end{equation*}
(loc.\ cit.\ II Chap.\ IV \S9). It is clear that
\begin{equation*}
  \widehat{\mathcal{G}}_m(\mathfrak{O})  = \ker \Big(\mathcal{G}(\mathfrak{O}) \xrightarrow{\mathrm{reduction}} \mathcal{G}(\mathfrak{O}/\pi^m \mathfrak{O}) \Big) \ .
\end{equation*}

\begin{proposition}\label{uniform}
  Suppose that $\mathfrak{O} = \mathbb{Z}_p$; then $\widehat{\mathcal{G}}_m(\mathbb{Z}_p)$, for any $m \geq 1$ if $p \neq 2$, resp.\ $m \geq 2$ if $p=2$, is a uniform pro-$p$ group.
\end{proposition}
\begin{proof}
By \cite{DDMS} \S13.2 (the discussion before Lemma 13.21) and Exercise 5 the group $\widehat{\mathcal{G}}_m(\mathfrak{O})$ is standard in the sense of loc.\ cit.\ Def.\ 8.22. Hence it is uniform by loc.\ cit.\ Thm.\ 8.31.
\end{proof}

To treat the general case we observe that the Weil restriction $\mathcal{G}_0 := \Res_{\mathfrak{O}/\mathbb{Z}_p}(\mathcal{G})$ is a smooth affine group scheme over $\mathbb{Z}_p$ (cf.\ \cite{BLR} \S7.6 Thm.\ 4 and Prop.\ 5). Let $e(\mathfrak{F}/\mathbb{Q}_p)$ denote the ramification index of the extension $\mathfrak{F}/\mathbb{Q}_p$. By the definition of the Weil restriction we have
\begin{equation*}
  \mathcal{G}_0(\mathbb{Z}_p) = \mathcal{G}(\mathfrak{O})  \quad\text{and}\quad  \mathcal{G}_0(\mathbb{Z}_p/p^m \mathbb{Z}_p) = \mathcal{G}(\mathfrak{O}/p^m \mathfrak{O}) = \mathcal{G}(\mathfrak{O}/\pi^{me(\mathfrak{F}/\mathbb{Q}_p)} \mathfrak{O})  \quad\text{for $m \geq 1$}.
\end{equation*}
We therefore obtain the following consequence of the above proposition.

\begin{corollary}\label{uniform2}
  Let $m = j e(\mathfrak{F}/\mathbb{Q}_p)$ with $j \geq 1$ if $p \neq 2$, resp.\ $j \geq 2$ if $p = 2$. Then $\widehat{\mathcal{G}}_m(\mathfrak{O})$ is a uniform pro-$p$ group.
\end{corollary}

\subsubsection{Bruhat-Tits group schemes}\label{subsubsec:BT}

We fix a facet $F$ in the standard apartment $\mathscr{A}$. Let $\mathbf{G}_F$ denote the Bruhat-Tits group scheme over $\mathfrak{O}$ corresponding to $F$ (cf. \cite{Tit}). It is affine smooth with general fiber $\mathbf{G}$, and $\mathbf{G}_F(\mathfrak{O})$ is the pointwise stabilizer in $G$ of of the preimage of $F$ in the extended building (denoted by $\mathscr{G}_{\pr^{-1}(F)}$ in \cite{Tit} 3.4.1). Its neutral component is denoted by $\mathbf{G}_F^\circ$. The group of points $K_F := \mathbf{G}_F^\circ(\mathfrak{O)}$ is the parahoric subgroup associated with the facet $F$. We introduce the descending sequence of normal congruence subgroups
\begin{equation*}
  K_{F,m} := \ker \Big( \mathbf{G}_F^\circ(\mathfrak{O}) \xrightarrow{\mathrm{reduction}} \mathbf{G}_F^\circ(\mathfrak{O}/\pi^m \mathfrak{O}) \Big) \quad\text{for $m \geq 1$}
\end{equation*}
in $K_F$. Let $\EuScript{P}_F^\dagger$ denote the stabilizer of $F$ in $G$. It follows from \cite{Tit} 3.4.3 (or \cite{BT2} 4.6.17) that each $K_{F,m}$, in fact, is a normal subgroup of $\EuScript{P}_F^\dagger$. Note that, given $F$, any open subgroup of $G$ contains some $K_{F,m}$.

\begin{corollary}\label{uniform3}
  For any $m = j e(\mathfrak{F}/\mathbb{Q}_p)$ with $j \geq 2$ the group $K_{F,m}$ is a uniform pro-$p$ group.
\end{corollary}
\begin{proof}
Apply Cor.\ \ref{uniform2} with $\mathcal{G} := \mathbf{G}_F^\circ$.
\end{proof}

In the following we will determine the groups $K_{F,m}$ in terms of the root subgroups $\EuScript{U}_\alpha$ and the torus $\mathbf{T}$. Let $\mathscr{T}$ over $\mathfrak{O}$ denote the neutral component of the Neron model of $\mathbf{T}$. We have $\mathscr{T}(\mathfrak{O}) = T^0$, and we put
\begin{equation*}
  T^m :=  \ker \Big( \mathscr{T}(\mathfrak{O}) \xrightarrow{\mathrm{reduction}} \mathscr{T}(\mathfrak{O}/\pi^m \mathfrak{O}) \Big) \quad\text{for $m \geq 1$.}
\end{equation*}
By \cite{BT2} 5.2.2-4 the group scheme $\mathbf{G}_F^\circ$ possesses, for each root $\alpha \in \Phi$, a smooth closed $\mathfrak{O}$-subgroup scheme $\mathscr{U}_{\alpha,F}$ such that
\begin{equation}\label{f:unipotent-points}
  \mathscr{U}_{\alpha,F}(\mathfrak{O}) = \EuScript{U}_{\alpha,f_F(\alpha)} \ .
\end{equation}

Moreover the product map induces an open immersion of $\mathfrak{O}$-schemes
\begin{equation}\label{f:open-imm}
  \prod_{\alpha \in \Phi^-} \mathscr{U}_{\alpha,F} \times \mathscr{T} \times \prod_{\alpha \in \Phi^+} \mathscr{U}_{\alpha,F} \hookrightarrow \mathbf{G}_F^\circ \ .
\end{equation}

\begin{proposition}\label{product}
  For any $m \geq 1$ the map \eqref{f:open-imm} induces the equality
\begin{equation*}
  \prod_{\alpha \in \Phi^-} \EuScript{U}_{\alpha,f_F(\alpha) + m} \times T^m \times \prod_{\alpha \in \Phi^+} \EuScript{U}_{\alpha,f_F(\alpha) + m} = K_{F,m} \ .
\end{equation*}
\end{proposition}
\begin{proof}
Let $\mathscr{Y}$ denote the left hand side of the open immersion \eqref{f:open-imm}. Because of \eqref{f:unipotent-points} the left hand side of our assertion is equal to the subset of all points in $\mathscr{Y}(\mathfrak{O})$ which reduce to the unit element modulo $\pi^m$ and hence is contained in $K_{F,m}$. On the other hand it follows from \cite{SchSt} Prop.\ I.2.2 and \eqref{f:unipotent-points} that any point in $\mathbf{G}_F^\circ(\mathfrak{O})$, which reduces to a point of the unipotent radical of its special fiber, already lies in $\mathscr{Y}(\mathfrak{O})$. It follows that $K_{F,m}$ corresponds to points in $\mathscr{Y}(\mathfrak{O})$ which reduce to the unit element modulo $\pi^m$ and hence is contained in the left hand side of the assertion.
\end{proof}

\begin{remark}
  In Chap.\ I of \cite{SchSt} certain pro-$p$ subgroups $U_F^{(e)} \subseteq G$ for $e \geq 0$ were introduced and studied. If $F = x$ is a hyperspecial vertex then $K_{x,m} = U_x^{(m+1)}$. On the other hand, if $F = D$ is a chamber then $K_{D,m} = U_D^{(m)} T^m$.
\end{remark}

\begin{corollary}\label{Frattini}
  Suppose that $m$ is large enough so that $K_{F,m}$ is uniform. Then the Frattini quotient $(K_{F,m})_\Phi$ of $K_{F,m}$ satisfies
\begin{equation*}
     \prod_{\alpha\in \Phi^-} \frac{{\EuScript U}_{\alpha, f_F(\alpha)+m}}{{\EuScript U}_{\alpha,f_F(\alpha)+m}^p} \times  \frac{T^m}{(T^m)^p} \times \prod_{\alpha\in \Phi^+} \frac{{\EuScript U}_{\alpha, f_F(\alpha)+m}}{{\EuScript U}_{\alpha, f_F(\alpha)+m}^p} \overset{\sim}\longrightarrow (K_{F,m})_\Phi \ .
\end{equation*}
\end{corollary}
\begin{proof}
As a consequence of Prop.\ \ref{product} the map in the assertion, which is given by multiplication, exists and is a surjection of $\mathbb{F}_p$-vector spaces. But both sides have the same dimension $d$. Hence the map is an isomorphism.
\end{proof}

In the case where the facet is a vertex $x$ in the closure of our fixed chamber $C$ we also introduce the notation
\begin{align*}
  \Phi_x := &\ \{(\alpha,\mathfrak{h}) \in \Phi_{aff} : (\alpha,\mathfrak{h})(x) = 0\},\ \Phi_x^\pm := \Phi_x \cap \Phi_{aff}^\pm, \\
  \Pi_x := &\ \Phi_x \cap \Pi_{aff},\ S_x := \{s \in S_{aff} : s(x) = x\}, \\
  W_x := &\ \text{subgroup of $W_{aff}$ generated by all $s_{(\alpha;\mathfrak{h})}$ such that $(\alpha,\mathfrak{h}) \in \Phi_x\}$}.
\end{align*}
The pair $(W_x,S_x)$ is a Coxeter system with finite group $W_x$ (cf. \cite{OS1} \S4.3 and the references therein).

For any such vertex we have the inclusions $K_{x,1} \subseteq I \subseteq J \subseteq K_x$.

\begin{lemma}\label{lemma:JwJ}
The parahoric subgroup $K_x$ is the disjoint union of the double cosets $JwJ$ for all $w \in W_x$.
\end{lemma}
\begin{proof}
See \cite{OS1} Lemma 4.9.
\end{proof}

\subsubsection{Triviality of actions on the top cohomology}\label{subsubsec:triviality}

We recall from section \ref{sec:posiroots} that $I$ and $J$ are normal subgroups of $\EuScript{P}_C^\dagger$ and that $\EuScript{P}_C^\dagger = \bigcup_{\omega \in \Omega} \omega J$. In the case where the root system is irreducible the following result was shown in \cite{Koziol} Thm.\ 7.1. The first part of our proof is essentially a repetition of his arguments.

\begin{lemma}\label{lemma:trivOmega}
For $g \in \EuScript{P}_C^\dagger$, the endomorphism $g_*$ on  the one dimensional $k$-vector space $H^d(I,k)$ is the identity.
\end{lemma}
\begin{proof}
As noted above, $g$ normalizes each subgroup $K_{C,m}$. Moreover, by Lemma \ref{lemma:UC}.ii and Prop.\ \ref{product}, $K_{C,m}$ is contained in $I$. Hence the same argument as at the beginning of the proof of Lemma \ref{lemma:trivK} reduces us to showing that the endomorphism $g_*$ on  the one dimensional $k$-vector space $H^d(K_{C,m},k)$ is the identity. Using Cor.\ \ref{uniform3} we may, by choosing $m$ large enough, assume that $K_{C,m}$ is a uniform pro-$p$ group.

Then, by \cite{Laz} V.2.2.6.3 and V.2.2.7.2, the one dimensional $k$-vector space $H^d(K_{C,m},k)$ is the maximal exterior power (via the cup product) of the $d$-dimensional $k$-vector space $H^1(K_{C,m},k)$. Conjugation commuting with the cup product, we see that the endomorphism $g_*$ on $H^d(K_{C,m},k)$ is the determinant of $g_*$  on $H^1(K_{C,m},k)$. We have
\begin{equation*}
  H^1(K_{C,m},k) = \Hom_{\mathbb{F}_p}((K_{C,m})_\Phi,k)
\end{equation*}
where $(K_{C,m})_\Phi$ is the Frattini quotient of the group $K_{C,m}$. This further reduces us to showing that the conjugation by $g$ on $(K_{C,m})_\Phi$ has trivial determinant. In Cor.\ \ref{Frattini} we computed this Frattini quotient to be
\begin{equation*}
     (K_{C,m})_\Phi = \prod_{\alpha\in \Phi^-} \frac{{\EuScript U}_{\alpha, f_C(\alpha)+m}}{{\EuScript U}_{\alpha,f_C(\alpha)+m}^p} \times  \frac{T^m}{(T^m)^p} \times \prod_{\alpha\in \Phi^+} \frac{{\EuScript U}_{\alpha, f_C(\alpha)+m}}{{\EuScript U}_{\alpha, f_C(\alpha)+m}^p}  \ .
\end{equation*}
using Lemma \ref{lemma:UC}.i this simplifies to
\begin{equation}\label{f:Frattini}
     (K_{C,m})_\Phi = \prod_{\alpha\in \Phi^-} \frac{{\EuScript U}_{\alpha, m+1}}{{\EuScript U}_{\alpha,m+1}^p} \times  \frac{T^m}{(T^m)^p} \times \prod_{\alpha\in \Phi^+} \frac{{\EuScript U}_{\alpha, m}}{{\EuScript U}_{\alpha,m}^p} \ .
\end{equation}
Recall that, for any $\alpha \in \Phi$, we have the additive isomorphism $x_\alpha : \mathfrak{F} \xrightarrow{\cong} \EuScript{U}_\alpha$ defined  in \eqref{f:xalpha} by
$  x_\alpha (u) := \varphi_\alpha( \left(
  \begin{smallmatrix}
    1 & u \\ 0 & 1
  \end{smallmatrix}
  \right) )$.
Put $s_0 :=
\left(
  \begin{smallmatrix}
    0 & 1 \\ -1 & 0
  \end{smallmatrix}
  \right) \in \mathrm{SL}_2(\mathfrak{F})
$ and $n_\alpha := \varphi_\alpha(s_0)$. We observe that
\begin{itemize}
  \item[--] $n_\alpha = n_{s_{(\alpha,0)}}$, and
  \item[--] $x_{-\alpha}(u) = n_\alpha x_\alpha(u) n_\alpha^{-1} = \varphi_\alpha( \left(
  \begin{smallmatrix}
    1 & 0 \\ -u & 1
  \end{smallmatrix}
  \right) )$ for any $u \in \mathfrak{F}$.
\end{itemize}
By \cite{Tit} 1.1 and 1.4 the map $x_\alpha$ restricts, for any $r \in \mathbb{Z}$, to an isomorphism $\pi^r \mathfrak{O} \xrightarrow{\cong} {\EuScript U}_{\alpha,r}$. This implies that all the $\mathbb{F}_p$-vector spaces $\frac{{\EuScript U}_{\alpha, m}}{{\EuScript U}_{\alpha,m}^p}$ and $\frac{{\EuScript U}_{-\alpha, m+1}}{{\EuScript U}_{-\alpha,m+1}^p}$, for $\alpha \in \Phi^+$, have the same dimension equal to $[\mathfrak{F}:\mathbb{Q}_p]$.

First let $g \in T^0$. Obviously $g$ centralizes $T^m$. It acts on $\EuScript{U}_\alpha$, resp.\ $\EuScript{U}_{-\alpha}$, via $\alpha$, resp.\ $-\alpha$. Therefore on the right hand side of \eqref{f:Frattini} the conjugation by $g$ visibly has trivial determinant. Since the conjugation action of $I$ on $H^d(I,k)$ is trivial we obtain our assertion for any $g \in J$.

For the rest of the proof we fix an $\omega \in \Omega$. It remains to establish our assertion for the elements $g \in \omega J$. In fact, by the above observation, it suffices to do this for one specific $\ddot{\omega} \in \omega J$, which we choose as follows. We write the image of $\omega$ in $W$ as a reduced product $s_{\alpha_1} \cdots s_{\alpha_\ell}$ of simple reflections and put
\begin{equation*}
  w_\omega := n_{\alpha_1} \cdots n_{\alpha_\ell} \in K_{x_0} \ .
\end{equation*}
Then $t := \omega w_\omega^{-1} \in T$, and we now define
\begin{equation*}
  \ddot{\omega} := t w_\omega \ .
\end{equation*}
Let $\Phi = \Phi_1 \dot{\cup} \ldots \dot{\cup} \Phi_r$ be the decomposition into orbits of (the image in $W$ of) $\omega$ and put
\begin{equation*}
      \Theta_i := \Big(\prod_{\alpha\in \Phi^- \cap \Phi_i} \frac{{\EuScript U}_{\alpha, m+1}}{{\EuScript U}_{\alpha,m+1}^p} \Big) \times \Big(\prod_{\alpha\in \Phi^+ \cap \Phi_i} \frac{{\EuScript U}_{\alpha, m}}{{\EuScript U}_{\alpha,m}^p} \Big)  \ .
\end{equation*}
The Chevalley basis $(x_\alpha)_{\alpha \in \Phi}$ has the following property (cf.\ \cite{BT2} 3.2):
\begin{align}\label{defiepsi}
   & \textrm{For any $\alpha \in \Phi$ there exists $\epsilon_{\alpha,\beta} \in \{\pm 1\}$ such that}\\
    & x_{s_\beta(\alpha)}(u) = n_\beta x_\alpha(\epsilon_{\alpha,\beta} u) n_\beta^{-1}  \quad\text{for any $u \in \mathfrak{F}$}. \nonumber
\end{align}
This implies that the conjugation by $w_\omega$ preserves each $\Theta_i$. The conjugation $t_*$ preserves each root subgroup. Since $\ddot{\omega}_*$ preserves $(K_{C,m})_\Phi$ it follows that $\ddot{\omega}_*$ preserves each $\Theta_i$. Setting $\Theta_0 := \frac{T^m}{(T^m)^p}$ we conclude that
\begin{equation*}
  (K_{C,m})_\Phi = \Theta_0 \times \Theta_1 \times \cdots \times \Theta_r
\end{equation*}
is an $\ddot{\omega}_*$-invariant decomposition. We will determine the determinant of $\ddot{\omega}_*$ on each factor.

The $\ddot{\omega}_*$-action on
\begin{equation*}
  \Theta_0 = \frac{X_*(T)}{p X_*(T)} \otimes_{\mathbb{F}_p} \frac{1+ \pi^m \mathfrak{O}}{(1+ \pi^m \mathfrak{O})^p}
\end{equation*}
is through the product $s_{\alpha_1} \cdots s_{\alpha_\ell} \in W$ acting on the left factor. Note that the dimension of the $\mathbb{F}_p$-vector space $\tfrac{1+ \pi^m \mathfrak{O}}{(1+ \pi^m \mathfrak{O})^p}$ is equal to $[\mathfrak{F}:\mathbb{Q}_p]$. Let $Q^\vee \subseteq X_*(T)$ denote the coroot lattice. By \cite{B-LL} VI.1.9 Prop.\ 27 the action of a simple reflection $s_\beta$ on the quotient $X_*(T)/Q^\vee$ is trivial. Therefore in the exact sequence
\begin{equation*}
  0 \longrightarrow \Tor_1^{\mathbb{Z}}(X_*(T)/Q^\vee,\mathbb{F}_p) \longrightarrow Q^\vee \otimes_{\mathbb{Z}} \mathbb{F}_p \longrightarrow X_*(T) \otimes_{\mathbb{Z}} \mathbb{F}_p \longrightarrow (X_*(T)/Q^\vee) \otimes_{\mathbb{Z}} \mathbb{F}_p \longrightarrow 0
\end{equation*}
the action of $s_\beta$ on the two outer terms is trivial. Hence the determinant of $s_\beta$ on $X_*(T)/pX_*(T)$ is equal to its determinant on $Q^\vee/pQ^\vee$. If $p \neq 2$ then $s_\beta$ is a reflection on the $\mathbb{F}_p$-vector space $Q^\vee/pQ^\vee$ and therefore has determinant $-1$. For $p=2$, as $s_\beta^2 = \id$, the determinant is $-1 = 1$ as well. We deduce that
\begin{equation}\label{f:det-0}
  \text{the determinant of $\ddot{\omega}_*$ on $\Theta_0$ is equal to $(-1)^{\ell[\mathfrak{F}:\mathbb{Q}_p]}$}.
\end{equation}

Next we consider $\Theta_i$ for some $1 \leq i \leq r$.
\def\cc{c}
Let $\cc_i$ denote its cardinality, and fix some root $\beta \in \Phi_i$. We distinguish two cases.

\textit{Case 1:} $-\Phi_i \cap \Phi_i = \emptyset$. Note that $\Phi_{i'} := -\Phi_i$ is the orbit of $-\beta$. By interchanging $i$ and $i'$ we may assume that $\beta \in \Phi^+$. We then have for completely formal reasons that
\begin{equation*}
  \text{determinant of $\ddot{\omega}_*$ on $\Theta_i$} = (-1)^{(\cc_i   - 1)[\mathfrak{F}:\mathbb{Q}_p]} \cdot \text{determinant of $\ddot{\omega}^{\cc_i  }_*$ on $\frac{{\EuScript U}_{\beta, m}}{{\EuScript U}_{\beta,m}^p}$ }
\end{equation*}
and correspondingly that
\begin{equation*}
  \text{determinant of $\ddot{\omega}_*$ on $\Theta_{i'}$} = (-1)^{(\cc_i   - 1)[\mathfrak{F}:\mathbb{Q}_p]} \cdot \text{determinant of $\ddot{\omega}^{\cc_i  }_*$ on $\frac{{\EuScript U}_{-\beta, m+1}}{{\EuScript U}_{-\beta,m+1}^p}$. }
\end{equation*}
We have $\ddot{\omega}^{\cc_i  } = w_\omega^{\cc_i  } t_i$ for some $t_i \in T$. The determinants of $t_{i*}$ on $\frac{{\EuScript U}_{\beta, m}}{{\EuScript U}_{\beta,m}^p}$ and $\frac{{\EuScript U}_{-\beta, m+1}}{{\EuScript U}_{-\beta,m+1}^p}$ (we must have $\beta(t_i) \in \mathfrak{O}^\times$) are inverse to each other. On the other hand let $w_i$ denote the image of $w_\omega^{\cc_i  }$ in $W$, which fixes $\beta$. As a consequence of the property \eqref{defiepsi} of the Chevalley basis, which we have recalled above, there are signs $\epsilon_{\pm \beta,w_i} \in \{\pm 1\}$ such that
\begin{equation*}
  x_{\pm\beta}(u) = w_\omega^{\cc_i  } x_{\pm \beta}(\epsilon_{\pm\beta,w_i} u) w_\omega^{-\cc_i  }  \qquad\text{for any $u \in \mathfrak{F}$}.
\end{equation*}
Altogether we obtain that
\begin{equation*}
  \text{determinant of $\ddot{\omega}_*$ on $\Theta_i \times \Theta_{i'}$} = (\epsilon_{\beta,w_i} \epsilon_{-\beta,w_i})^{[\mathfrak{F}:\mathbb{Q}_p]} \ .
\end{equation*}
But we have $\epsilon_{\beta,w_i} = \epsilon_{-\beta,w_i}$. This follows by a straightforward induction from the fact that $\epsilon_{-\beta,\alpha} = \epsilon_{\beta,\alpha}$ (cf.\ \cite{Spr} Lemma 9.2.2(ii)). Hence in the present case we deduce that
\begin{equation}\label{f:det-i,i'}
  \text{the determinant of $\ddot{\omega}_*$ on $\Theta_i \times \Theta_{i'}$ is equal to $1$}.
\end{equation}

\textit{Case 2:} $-\Phi_i = \Phi_i$. Again we may assume that $\beta \in \Phi^+$. Then $\cc_i  $ is even, and $\ddot{\omega}_*^{\cc_i  /2}$ preserves the product $\frac{{\EuScript U}_{\beta, m}}{{\EuScript U}_{\beta,m}^p} \times \frac{{\EuScript U}_{-\beta, m+1}}{{\EuScript U}_{-\beta,m+1}^p}$. The formal argument now says that
\begin{equation*}
  \text{determinant of $\ddot{\omega}_*$ on $\Theta_i$} = \text{determinant of $\ddot{\omega}^{\cc_i  /2}_*$ on $\frac{{\EuScript U}_{\beta, m}}{{\EuScript U}_{\beta,m}^p} \times \frac{{\EuScript U}_{-\beta, m+1}}{{\EuScript U}_{-\beta,m+1}^p}$. }
\end{equation*}
This time we write $\ddot{\omega}^{\cc_i  /2} = w_\omega^{\cc_i  /2} t_i$ for some $t_i \in T$ and we let $w_i$ denote the image of $w_\omega^{\cc_i  /2}$ in $W$, which maps $\beta$ to $-\beta$. We have $\beta(t_i) \in \pi \mathfrak{O}^\times$. Using the Chevalley basis we compute the above right hand determinant as being equal to
\begin{equation*}
  (-1)^{[\mathfrak{F}:\mathbb{Q}_p]} (\beta(t_i)\epsilon_{\beta,w_i})^{[\mathfrak{F}:\mathbb{Q}_p]} (\beta(t_i)^{-1}\epsilon_{-\beta,w_i})^{[\mathfrak{F}:\mathbb{Q}_p]} = (-1)^{[\mathfrak{F}:\mathbb{Q}_p]} \ .
\end{equation*}
Hence in this case we deduce that
\begin{equation}\label{f:det-i}
  \text{the determinant of $\ddot{\omega}_*$ on $\Theta_i$ is equal to $(-1)^{[\mathfrak{F}:\mathbb{Q}_p]}$}.
\end{equation}

Combining \eqref{f:det-0}, \eqref{f:det-i,i'}, and \eqref{f:det-i} we have established at this point that
\begin{equation}
  \text{the determinant of $\ddot{\omega}_*$ on $(K_{C,m})_\Phi$ is equal to $(-1)^{(N + \ell)[\mathfrak{F}:\mathbb{Q}_p]}$}
\end{equation}
where $N$ is the number of $\omega$-orbits $\Phi_i = - \Phi_i$ and, we repeat, $\ell$ is the length of $w_\omega$. It remains to show that the sum $N + \ell$ always is even. This will be done in the subsequent lemma in a more general situation.
\end{proof}

\begin{lemma}\label{even}
  Let $w \in W$ be any element and let $N(w)$ be the number of $w^{\mathbb{Z}}$-orbits $\Psi \subseteq \Phi$ with the property that $-\Psi = \Psi$; then the number $N(w) + \ell(w)$ is even.
\end{lemma}
\begin{proof}
We have to show that $(-1)^{N(w)} = (-1)^{\ell(w)}$ holds true for any $w \in W$.

\textit{Step 1:} The map $w \mapsto (-1)^{\ell(w)}$ is a homomorphism. This is immediate from the fact that $(-1)^{\ell(w)}$ is equal to the determinant of $w$ in the reflection representation of $W$ (cf. \cite{Hum}).

\textit{Step 2:} The map $w \mapsto (-1)^{N(w)}$ is a homomorphism. For this let $\mathbb{Z}[S]$ denote the free abelian group on a set $S$. We consider the obvious action of $W$ on $\mathbb{Z}[\Phi]$. If $\Phi = \Phi_1 \dot{\cup} \ldots \dot{\cup} \Phi_r$ is the decomposition into $w^{\mathbb{Z}}$-orbits, then
\begin{equation*}
  \mathbb{Z}[\Phi] = \mathbb{Z}[\Phi_1] \oplus \ldots \oplus \mathbb{Z}[\Phi_r]
\end{equation*}
is a $w$-invariant decomposition. Obviously $\det(w|\mathbb{Z}[\Phi_i]) = (-1)^{|\Phi_i| - 1}$. If $-\Phi_i \neq \Phi_i$ then $-\Phi_i = \Phi_j$ for some $j \neq i$. If $-\Phi_i = \Phi_i$ then $|\Phi|$ is even. It follows that $\det(w|\mathbb{Z}[\Phi]) = (-1)^{N(w)}$.

\textit{Step 3:} If $w = s$ is a reflection at some $\alpha \in \Phi$ then $N(s) = 1$ and $(-1)^{\ell(s)} = -1$.
\end{proof}

\begin{lemma}\label{lemma:trivK}
Let $L$  be an open subgroup of $I$. For $g\in K_x$  normalizing $L$, the endomorphism $g_*$ on the one dimensional $k$-vector space $H^d(L,k)$ is the identity.
\end{lemma}
\begin{proof}
We choose $m\geq 1$ large enough so that $K_{x,m}$ is contained in $L$.  Recall that $\cores^{K_{x,m}}_L : H^d(K_{x,m},k) \xrightarrow{\cong} H^d(L,k)$ is an isomorphism by Remark \ref{rem:cores}. Since  corestriction commutes with conjugation (\S\ref{sec:basicprop}), the following diagram commutes:
\begin{equation*}
  \xymatrix{
    H^d(L,k) \ar[r]^{g_*} & H^d(L,k)  \\
    H^d(K_{x,m},k)  \ar[u]_{\cores^{K_{x,m}}_{L}}^{\cong} \ar[r]^{g_*} & H^d(K_{x,m},k) \ar[u]^{\cores^{K_{x,m}}_{L}}_{\cong}  \ . }
\end{equation*}
Therefore it is enough to prove the assertion for $L=K_{x,m}$. In this case the group $K_x$ acts by conjugation on the one dimensional space $H^d(K_{x,m},k)$. This action is given by a character $\xi : K_x \rightarrow k^\times$. Its kernel $\Xi := \ker(\xi)$ is a normal subgroup of $K_x$.

First of all we recall again that the corestriction map commutes with conjugation, that $\cores^{K_{x,m}}_I: H^d(K_{x,m},k) \xrightarrow{\cong} H^d(I,k)$ is an isomorphism, and that conjugation by $g_*$, for $g \in I$, induces the identity on the cohomology $H^*(I, k)$. Therefore we have the commutative diagram
 \begin{equation*}
  \xymatrix@R=0.5cm{
                &         H^d(K_{x,m},k)  \ar[dl]_{\cores_I^{K_{x,m}}}^{\cong}   \\
  H^d(I, k)                 \\
                &         H^d(K_{x,m},k)    \ar[uu]_{g_*} \ar[ul]^{\cores_I^{K_{x,m}}}_{\cong}\ .              }
\end{equation*}
This shows that $I \subseteq \Xi$. Since $J = I T^0$ we deduce from Lemma \ref{lemma:trivOmega} that even $J \subseteq \Xi$. Since $(W_x, S_x)$ is a Coxeter system with  finite group $W_x$, we may consider its (unique) longest element $w_x$. The  normal subgroup   $\Xi$ of $K_x$  then must contain $J w_x J w_x^{-1} J$. For any $s\in S_x$, we have $\ell(sw_x)=\ell(w)-1$. By \eqref{f:quadratic2} we have
\begin{equation*}
  JsJ \cdot Jw_xJ = Jsw_xJ \; \dot{\cup} \; Jw_xJ
\end{equation*}
and hence $\dot{s}J\dot{w}_x \cap Jw_x J \neq \emptyset$. It follows that
\begin{equation*}
  \dot{s} \in Jw_xJw_x^{-1}J  \subseteq \Xi
\end{equation*} and therefore $\Xi= K_x$ by Lemma \ref{lemma:JwJ}.
\end{proof}

\begin{proposition}\label{prop:triviality}
For $g\in G$, and   $v,w\in \widetilde W$ such that $\ell(vw)=\ell(v)+\ell(w)$, the following diagrams of one dimensional $k$-vector spaces are commutative:
\begin{equation*}
\begin{array}{cc}
 \xymatrix@R=0.5cm{
                &         H^d(I_g,k)  \ar[dl]_{\cores_I^{I_g}}^{\cong}   \\
  H^d(I, k)                 \\
                &         H^d(I_{g^{-1}},k)    \ar[uu]_{g_*} \ar[ul]^{\cores_I^{I_{g^{-1}}}}_{\cong}\,             }& \quad\textrm{ \textrm{ and } }\quad
  \xymatrix@R=0.5cm{
                &         H^d(I_{vw},k) \ .  \ar[dl]_{\cores_I^{I_{vw}}}^{\cong}   \\
  H^d(I, k)                 \\
                &         H^d(v^{-1}I_{vw}v,k).    \ar[uu]_{v_*} \ar[ul]^{\cores_I^{v^{-1}I_{vw} v} }_{\cong}              }\cr
                \end{array}
\end{equation*}
\end{proposition}

\begin{proof}
We prove the commutativity of the left diagram. We will see along the way that the commutativity of the right one follows.

\textit{Step 1:} We claim that it suffices to establish the commutativity of the left diagram for elements $w \in \widetilde{W}$.

Let $g \in G$ and $h_1, h_2 \in I$. We have the commutative diagram
\begin{equation*}
  \xymatrix{
    H^d(I_{(h_1 gh_2)^{-1}},k) \ar[d]_{\cores_I^{I_{(h_1 gh_2)^{-1}}}}^{\cong} \ar[r]^-{h_{2*}} & H^d(I_{g^{-1}},k) \ar[d]_{\cores_I^{I_{g^{-1}}}}^{\cong} \ar[r]^-{g_*} & H^d(I_g,k) \ar[d]_{\cores_I^{I_g}}^{\cong} \ar[r]^-{h_{1*}} & H^d(I_{h_1 gh_2},k) \ar[d]_{\cores_I^{I_{h_1 gh_2}}}^{\cong} \\
    H^d(I,k) \ar[r]^{=} & H^d(I,k) \ar[r]^{?} & H^d(I,k) \ar[r]^{=} & H^d(I,k) ,  }
\end{equation*}
where the equality signs in the lower row use the facts that corestriction commutes with conjugation (cf.\ \S\ref{sec:basicprop}) and that the conjugation
by an element in $I$ is trivial on $H^d(I,k)$. This shows that the commutativity of the left diagram in the assertion only depends on the double coset $IgI$.

\textit{Step 2:} Let $v , w \in \widetilde{W}$ be two elements for which the left diagram commutes and such that $\ell(vw) = \ell(v) + \ell(w)$; we claim that then the left diagram for $vw$ as well as the right diagram commute.

This is straightforward from the following commutative diagram (cf.\ Lemma \ref{lemma:vw}):
\begin{equation*}
  \xymatrix{
     & H^d(I_v,k) \ar[ddl]_{\cores}^{\cong}  & H^d(I_{vw},k) \ar[l]_{\cores}^{\cong} \\
     & H^d(I_{v^{-1}},k) \ar[dl]^{\cores}_{\cong} \ar[u]_{v_*}^{\cong} &  \\
    H^d(I,k) &   & H^d(v^{-1}Iv \cap wIw^{-1},k) \ar[uu]_{v_*}^{\cong} \ar[ul]^{\cores}_{\cong}  \ar[dl]_{\cores}^{\cong} \\
     & H^d(I_w,k) \ar[ul]_{\cores}^{\cong}  & \\
     & H^d(I_{w^{-1}},k) \ar[uul]^{\cores}_{\cong} \ar[u]_{w_*}^{\cong} & H^d(I_{(vw)^{-1}},k) \ar[uu]_{w_*}^{\cong} \ar[l]^{\cores}_{\cong}  \ar@/_15ex/[uuuu]_{(vw)_*}^{\cong} }
\end{equation*}
At this point we are reduced to establishing the commutativity of the left diagram in our assertion in the following two cases:
\begin{itemize}
  \item[(A)] $g=\dot \omega$ for $\omega \in \widetilde \Omega$. In that case, the claim is given by Lemma \ref{lemma:trivOmega}.
  \item[(B)] $g = n_s$ for $s = s_{(\beta,\mathfrak{h})}$ where $(\beta,\mathfrak{h}) \in \Pi_{aff}$ is any simple affine root.
\end{itemize}

\textit{Step 3:} It remains to treat case (B). We pick a vertex $x$ in the closure of the chamber $C$ such that $s(x) = x$. Then $g$ lies in $K_x$ and therefore normalizes $K_{x,1}$. Hence $K_{x,1} \subseteq I_g \subseteq I$, and we have the commutative diagram
\begin{equation*}
  \xymatrix{
    H^d(I,k)  \ar[r]^{?} & H^d(I,k)  \\
    H^d(I_{g^{-1}},k) \ar[u]_{\cores_I^{I_{g^{-1}}}}^{\cong} \ar[r]^{g_*} & H^d(I_g,k) \ar[u]^{\cores_I^{I_g}}_{\cong} \\
    H^d(K_{x,1},k) \ar@/^13ex/[uu]^{\cores^{K_{x,1}}_I}_{\cong} \ar[u]_{\cores^{K_{x,1}}_{I_{g^{-1}}}}^{\cong} \ar[r]^{g_*} & H^d(K_{x,1},k) \ar[u]^{\cores^{K_{x,1}}_{I_g}}_{\cong} \ar@/_13ex/[uu]_{\cores^{K_{x,1}}_I}^{\cong} , }
\end{equation*}
using again that corestriction commutes with conjugation. This reduces us to showing that the endomorphism $g_*$ on $H^d(K_{x,1},k)$ is the identity. This claim is given by Lemma \ref{lemma:trivK}.
\end{proof}

\begin{corollary}\label{coro:SJ=S}
We have $\trace^d \circ \anti = \trace^d$ on  $E^d = H^d(I, \X)$.
\end{corollary}
\begin{proof}
Let $w\in \widetilde W$ and $\alpha\in H^d(I, \X(w))$. Recalling that $\anti(\alpha)\in
 H^d(I, \X(w^{-1}))$ satisfies $\Sh_{w^{-1}}(\anti(\alpha))=(w^{-1})_* \Sh_{w}(\alpha)$, we have
\begin{align*}
   \trace^d \circ \anti (\alpha) & = H^d(I,\trace_{w^{-1}}) \circ \anti (\alpha)=
      H^d(I,\trace_{w^{-1}}) \circ \Sh_{w^{-1}}^{-1}\circ (w^{-1})_* (\Sh_w (\alpha))  \\
      & = \cores_I^{I_{w^{-1}}}\circ (w^{-1})_* (\Sh_w (\alpha))  \quad \text{ by Remark \ref{rema:tracew}}  \\
      & = \cores_I^{I_{w}} (\Sh_w (\alpha)) \qquad\qquad\ \ \ \ \;\text{ by Prop.\ \ref{prop:triviality}}  \\
      & = H^d(I,\trace_w)(\alpha) = \trace^d(\alpha) \qquad\ \ \ \,\text{ by Remark \ref{rema:tracew}}.
\end{align*}
\end{proof}

\subsubsection{The duality theorem}\label{subsubsec:duality}

We fix an isomorphism $\eta:H^d(I,k)\longrightarrow k$. By Lemma \ref{lemma:nondegenerate} the map
\begin{align}\label{f:dual}
\Delta^i : E^i = H^i(I,\mathbf X) &\longrightarrow H^{d-i}(I, \X)^\vee = (E^{d-i})^\vee \\
\alpha & \longmapsto  l_\alpha(\beta):=\eta\circ \trace^d (\alpha\cup \beta)   \nonumber
\end{align}
is a linear injectionwith image $(E^{d-i})^{\vee,f}$. The space  $E^{d-i}$ is naturally a bimodule under $H$.
As in \S\ref{subsec:duals}, we consider the  twisted $H$-bimodule ${}^\anti (E^{d-i})^\anti$  namely the space $E^{d-i}$
with the action of $H$ on $\beta \in E^{d-i}$ given by
\begin{equation*}
  (\tau, \beta, \tau')\mapsto \anti(\tau')\cdot \beta  \cdot \anti(\tau) \quad\textrm{ for $\tau, \tau'\in H$.}
\end{equation*}

\begin{proposition}\label{prop:bimodule}
The map \eqref{f:dual} induces an injective morphism of $H$-bimodules
\begin{equation*}
  E^i \longrightarrow ({}^\anti (E^{d-i})^\anti)^\vee\
\end{equation*}
with image $({}^\anti (E^{d-i})^\anti)^{\vee,f}$.
\end{proposition}
\begin{proof}
The fact that the map is injective with image  $(E^{d-i})^{\vee,f}$ comes directly from Lemma \ref{lemma:nondegenerate} and its proof.

We prove that for all $\alpha\in H^i (I,\mathbf X)$ and $\beta\in H^{d-i}(I, \mathbf X)$ and all $\tau, \tau'\in H$, we have $l_{\tau\cdot \alpha \cdot \tau'}(\beta)= l_\alpha(\anti(\tau)\cdot \beta\cdot \anti(\tau'))$ namely
\begin{equation}\label{f:toprove-anti}
\trace^d(\tau\cdot \alpha \cdot \tau'\cup \beta)=\trace^d(\alpha\cup \anti(\tau)\cdot \beta\cdot \anti(\tau')) \ .
\end{equation}
We first show that
\begin{equation}\label{f:rightaction}
  \trace^d( \alpha \cdot \tau'\cup \beta) = \trace^d(\alpha\cup \beta\cdot \anti(\tau')) \ .
\end{equation}
The right action of $H$  on $H^*(I, \X)$ being through the coefficients it is enough to prove that for any $a,b\in \X$ and $\tau'\in H$ we have
\begin{equation*}
  \trace((a \cdot \tau')b - a (b \cdot \anti(\tau'))) = 0 \ .
\end{equation*}
We check this equality for $a=\chara_{xI}$ and $b=\chara_{yI}$ with $x,y\in G$ and  $\tau'=\tau_g$ with $g\in G$. We then have
\begin{equation*}
   (a \cdot \tau')b - a (b \cdot \anti(\tau')) = \chara_{xI g I\cap y I}- \chara_{xI\cap yIg^{-1}I}.
\end{equation*}
It lies in the kernel of $\trace$ if and only if $xI g I\cap y I/I$ and  $xI\cap yIg^{-1}I/I$ have the same cardinality. But observe that $xI g I\cap y I$ is equal to
$yI$ if $x^{-1}y\in IgI$ and is empty otherwise, while $xI\cap yIg^{-1}I$ is equal to  $x I$ if $y^{-1}x\in I g^{-1} I$ and is empty otherwise. This proves the equality and \eqref{f:rightaction} follows.

That
\begin{equation}\label{f:leftaction}
  \trace^d(\tau\cdot \alpha \cup \beta)=\trace^d(\alpha\cup \anti(\tau)\cdot \beta)\ .
\end{equation}
holds true as well follows now from the following computation:
We may assume that $\alpha \in E^i$, $\beta \in E^j$, and $\gamma \in E^m$. We then compute
\begin{align*}
  \trace^d(\tau\cdot \alpha \cup \beta) & = \trace^d(\anti(\tau\cdot \alpha \cup \beta)) & \text{by Cor.\ \ref{coro:SJ=S}}  \\
   & = \trace^d(\anti(\tau\cdot \alpha) \cup \anti(\beta)) & \text{by Remark \ref{rema:Jcup}}  \\
   & = \trace^d(\anti(\alpha) \cdot \anti(\tau) \cup \anti(\beta)) & \text{by Prop.\ \ref{prop:anti+product}}  \\
   & = \trace^d(\anti(\alpha) \cup \anti(\beta)\cdot \tau) & \text{by \eqref{f:rightaction}}  \\
   & = \trace^d(\anti(\alpha) \cup \anti(\anti(\tau) \cdot \beta)) & \text{by Prop.\ \ref{prop:anti+product}} \\
   & = \trace^d(\anti(\alpha \cup \anti(\tau) \cdot \beta)) & \text{by Remark \ref{rema:Jcup}} \\
   & = \trace^d(\alpha \cup \anti(\tau) \cdot \beta) & \text{by Cor.\ \ref{coro:SJ=S}}.
\end{align*}
\end{proof}

\begin{corollary}
   For any $i\in\{0, \cdots , d\}$ the space $(E^i)^{\vee,f}$ is a sub-$H$-bimodule  of $(E^i)^\vee$.
\end{corollary}

\section{The structure of $E^d$}\label{sec:top}

In this section we still \textbf{assume}, as in \S\ref{subsec:d-i}, that the pro-$p$ Iwahori group $I$ is torsion free. In this section we describe the top cohomology space  $E^d = H^d (I,\X)$ as an $H$-bimodule.

\begin{remark}
The  space $E^d$ as a right $H$-module had already been computed in \cite{SDGA} \S5.1.
\end{remark}

Recall  that the pairing $E^d\times E^0\rightarrow k, \: (\alpha, \beta)\mapsto \eta\circ \trace^d(\alpha\cup \beta)$ induces the isomorphism of $H$-bimodules
\begin{equation}\label{f:Deltad}
\Delta^d:\quad E^d\xrightarrow{\cong}\: ({}^\anti E^0\,^\anti)^{\vee, f}
\end{equation}
of Proposition \ref{prop:bimodule} when $i=d$. We denote by $(\phi_w)_{w\in\widetilde W}$ the basis of $E^d$ obtained by dualizing the basis of $(\tau_w)_{w\in \widetilde W}$ of $E^0$. For each $w\in \widetilde W$, the element $\phi_w$  is  the only one  in $H^d(I, \X(w))$ satisfying $\eta(\trace^d(\phi_w))=1$ (see \eqref{f:orth}). Now $\anti(\phi_w)\in  H^d(I, \X(w^{-1}))$ and
$\eta(\trace^d(\anti(\phi_w)))$ and $\eta(\trace^d(\phi_{w^{-1}}))$ are both equal to $1$ by Corollary \ref{coro:SJ=S}. Therefore,
\begin{equation}\label{f:Jphi}
  \anti(\phi_w)=\phi_{w^{-1}} \ .
\end{equation}

We now describe the explicit  action of $H$ on the elements $(\phi_w)_{w\in\widetilde W}$ in $E^d$. For any $s\in S_{aff}$ recall that we introduced  the idempotent element $\theta_s$ in \eqref{f:thetas}:
\begin{equation*}
    \theta_s := -\vert\kera\vert \sum_{t \in \ima} \tau_t\ \in H\ .
\end{equation*}

\begin{proposition}\label{prop:Hdformulas}
Let  $w\in {\widetilde W}$, $\omega\in \widetilde \Omega$ and $s\in S_{aff}$. We have the formulas:
\begin{equation}\label{f:Hdomega}
 \phi_w\cdot \tau_\omega= \phi_{w\omega} \ ,\ \tau_\omega \cdot \phi_w= \phi_{\omega w},
\end{equation}
\begin{equation}\label{f:rightHd}
 \phi_w\cdot \tau_{n_s}= \begin{cases} \phi_{w\tilde s}+\vert\kera\vert\sum_{t\in \ima} \phi_{w\bar{t}} = \phi_{w\tilde s}-\phi_w\cdot \theta_s & \text{ if $\ell(w\tilde s)=\ell(w)-1$,}\cr 0& \text{ if $\ell(w\tilde s)=\ell(w)+1$,}\end{cases}
\end{equation}
\begin{equation}\label{f:leftHd}
 \tau_{n_s}\cdot \phi_w= \begin{cases} \phi_{\tilde s w}+\vert\kera\vert\sum_{t\in \ima} \phi_{\bar{t}w} = \phi_{\tilde s w}- \theta_s\cdot \phi_w& \text{ if $\ell(\tilde sw)=\ell(w)-1$,}\cr 0& \text{ if $\ell(\tilde sw)=\ell(w)+1$.}\end{cases}
\end{equation}
\end{proposition}
\begin{proof}
Before proving the proposition, we recall that the quadratic relations in $H$ are
given by $\tau_{n_s}^2= -\theta_s \tau_{n_s}=\tau_{n_s} \theta_s$ for any $s\in S_{aff}$. Let $w\in {\widetilde W}$, $s\in  S_{aff}$ and $\omega \in \Omega$. We study the right action of $\tau_\omega $ and of $\tau_{n_s}$ on $\phi_w\in E^d $. Recall that ${\phi_w}\cdot\tau_\omega ({}_-)={\phi_w}({}_-\tau_{\omega^{-1}})$ and
${\phi_w}\cdot \tau_{n_s} ({}_-)={\phi_w}({}_-\tau_{{n_s}^{-1}})$ (see \eqref{f:toprove-anti}).
Let $w'\in {\widetilde W}$.  Below we use the braid relation \eqref{braid} repeatedly.
\begin{itemize}
\item We have ${\phi_w}\cdot \tau_\omega \, (\tau_{w'})={\phi_w}(\tau_{w'\omega^{-1}})$. It is nonzero if and only if $w'=w\omega$ in which case ${\phi_w}\cdot \tau_\omega (\tau_{w'})=1$.

\item If $\ell(w'{\tilde s})=\ell(w')+1$ then ${\phi_w}\cdot \tau_{n_ s}\, (\tau_{w'})={\phi_w}(\tau_{w'}\tau_{{n_s}^{-1}})={\phi_w}(\tau_{w'{\tilde s}^{-1}})$ and it is nonzero if and only if $I w I$ is contained in $ I w'{\tilde s}^{-1}I$ which is equivalent  to $w{\tilde s}=w'$  and in which case ${\phi_w}.\tau_{n_ s}\,(\tau_{w'})=1$.

\item If $\ell(w'{\tilde s})=\ell(w')-1$ then $\tau_{w'}=\tau_{w'\tilde s}\tau_{\tilde s^{-1}}$  and  ${\phi_w}\cdot \tau_{n_s }(\tau_{w'})={\phi_w}(\tau_{w'{\tilde s}}\tau_{n_s^{-1}}\tau_{n_s^{-1}})=
{\phi_w}(\tau_{w'{\tilde s}}\tau_{n_s}^2\tau_{n_s^2})=-{\phi_w}(\tau_{w'}\theta_{s})$  which is nonzero if and only if there is $t\in \ima$ such that $w' \bar{t}=w$ and in which case ${\phi_w}\cdot \tau_{n_s}(\tau_{w'})=\vert\kera\vert$.
\end{itemize}

From the discussion above, we immediately deduce that ${\phi_w}\cdot \tau_\omega=\phi_{w\omega}$. Now we compute ${\phi_w}\cdot \tau_{n_s}$. Suppose that $\ell(w{\tilde s})=\ell(w)+1$, then for any $w'\in {\widetilde W}$ we have ${\phi_w}\cdot \tau_{n_s}(\tau_{w'})=0$,  so ${\phi_w}\cdot \tau_{n_ s}=0$.
Suppose that $\ell(w{\tilde s})=\ell(w)-1$, then for any $w'\in {\widetilde W}$ we have
${\phi_w}\cdot \tau_{n_s}(\tau_{w'})=0$ except if $w'=w{\tilde s}$ in which cases
${\phi_w}\cdot \tau_{n_s}(\tau_{w'})=1$, or if there is $t\in \ima$ such that
$w'= w \bar{t}^{-1}$ in which case ${\phi_w}\cdot \tau_{\tilde s}(\tau_{w'})=\vert\kera\vert$, so
${\phi_w}\cdot \tau_{n_s}=\phi_{w{\tilde s}}+\vert\kera\vert\sum_{t\in\ima}\phi_{w \bar{t}}$.\\

Since \eqref{f:Deltad} is an isomorphism of right $H$-modules, the calculation above gives the right action of $H$ on the $\phi$'s. The formulas for the left action follow using \eqref{f:Jphi}, Proposition \ref{prop:anti+product} and
\begin{align*}
  \tau_x\cdot\phi_w&=\anti(\anti(\tau_x\cdot\phi_w))=\anti(\anti(\phi_w)\cdot  \anti(\tau_x)) \\
  & = \anti(\phi_{w^{-1}}\cdot  \tau_{x^{-1}}).
\end{align*}
\end{proof}

\begin{remark}
\begin{itemize}
\item[i.] The formula for the right action coincides with the one given in \cite{SDGA} \S 5.1 p.9.
\item[ii.]  Recall that $\upiota(\tau_{n_s})= -\tau_{n_s}-\theta_s$  where $\upiota$ is the involutive automorphism of $H$ defined in \eqref{f:upiota}. Formulas \eqref{f:rightHd} and \eqref{f:leftHd} can be given in the form:
\begin{equation*}
 \phi_w\cdot \upiota(\tau_{n_s})= \begin{cases}- \phi_{w\tilde s}& \text{ if $\ell(w\tilde s)=\ell(w)-1$}\cr - \phi_w\cdot \theta_s & \text{ if $\ell(w\tilde s)=\ell(w)+1$}\end{cases}, \
\upiota( \tau_{n_s})\cdot \phi_w= \begin{cases} -\phi_{\tilde s w}& \text{ if $\ell(\tilde sw)=\ell(w)-1$}\cr - \theta_s\cdot\phi_w& \text{ if $\ell(\tilde sw)=\ell(w)+1$.}\end{cases}
\end{equation*}
\end{itemize}
\end{remark}

Recall that we defined in \eqref{f:trace} the $G$-equivariant  map $\trace= \sum_{g\in G/I} \ev_g\ : \X\longrightarrow k$, where $k$ is endowed with the trivial action of $G$.
Note that the restriction of $\trace$ to $H=\X^I$ coincides with the  trivial character $\chi_{triv}: H\rightarrow k$ as defined in \S\ref{sec:charH}. This is because for $w\in\widetilde W$ we have
 $\trace(\tau_w)=\trace(\chara_{IwI})=\vert IwI/I\vert=\vert I/I_w\vert= q^{\ell(w)}.1_k$ by Corollary \ref{coro:known}.i.
Since $\X$  is generated by $\chara_I$ as a representation of $G$, it follows that $\trace$ is a morphism of
$(G,H)$ modules
$\X\rightarrow k_{triv}$   where  $k_{triv}$ denotes  the  $(G,H)$-bimodule $k$  with the trivial action of $G$ and the action  of $H$ via $\chi_{triv}$.

The cohomology group $H^d(I,k_{triv}) = \Ext^d_{\Mod(G)}(\mathbf{X},k_{triv})$ is naturally an $H$-bimodule where the left action comes from the right action on $\mathbf{X}$ and the right action comes from the trivial action on $k_{triv}$.

\begin{proposition}\phantomsection\label{Hdtriv}
\begin{itemize}
  \item[i.] The left as well as the right $H$-action on $H^d(I,k_{triv})$ are trivial, i.e., are through $\chi_{triv}$.
  \item[ii.] The map $\trace ^d$ induced on cohomology by $\trace$ yields an exact sequence of $H$-bimodules
\begin{equation}\label{ses:d}
   0\longrightarrow  \ker(\trace^d )\longrightarrow E^d \longrightarrow  H^d(I, k_{triv})\longrightarrow 0 \ .
\end{equation}
\end{itemize}
\end{proposition}
\begin{proof}
i. The triviality of the right action is obvious. For the left action we first consider the following diagram
\begin{equation*}
  \xymatrix{
     \Ext^*_{\Mod(G)}(\mathbf{X},V) \ar[d]_{=} \ar[rrr]^{\tau_{g}} &  &  & \Ext^*_{\Mod(G
     )}(\mathbf{X},V) \ar[d]^{=} \\
     H^*(I,V) \ar[dr]^{g_*} \ar[r]^-{\res} & H^*(I \cap g^{-1}Ig,V) \ar[dr]^{g_*} &  & H^*(I,V)  \\
     & H^*(gIg^{-1},V) \ar[r]^-{\res} & H^*(I \cap gIg^{-1},V) , \ar[ur]^{\cores} &    }
\end{equation*}
where $V$ is an arbitrary object in $\Mod(G)$ and where the upper horizontal arrow is the action of $\tau_g \in H$ induced by its right action on $\mathbf{X}$. For its commutativity it suffices, by using an injective resolution of $V$, to consider the case $* = 0$, i.e., the diagram
\begin{equation*}
  \xymatrix{
     \Hom_{k[G]}(\mathbf{X},V) \ar[d]_{=}^{f \mapsto f(\chara_I)} \ar[rrr]^{f \mapsto f(_-\tau_{g})} &  &  & \Hom_{k[G
     ]}(\mathbf{X},V) \ar[d]^{=}_{f \mapsto f(\chara_I)} \\
     V^I \ar[dr]^{g}  &  &  & V^I \\
     & V^{gIg^{-1}} \ar[r]^-{\subseteq} & V^{I \cap gIg^{-1}} \ar[ur]^{\sum_{h \in I/I \cap gIg^{-1}} h}. &    }
\end{equation*}
Its commutativity is verified easily by direct inspection. So our assertion reduces to the claim that the composed map
\begin{equation*}
  H^d(I,k) \xrightarrow{\res} H^d(I_{g^{-1}},k) \xrightarrow{g_*} H^d(I_g,k) \xrightarrow{\cores} H^d(I,k)
\end{equation*}
coincides with the multiplication by $\chi_{triv}(\tau_g)$. Suppose that $g \in IwI$ with $w \in \widetilde{W}$. By Cor.\ \ref{coro:known}.i we have $I_{g^{-1}} \subsetneqq I$ if and only if $\ell(w) > 0$. In this case the left restriction map above is the zero map by Remark \ref{rem:cores} and $\chi_{triv}(\tau_w) = 0$. If $\ell(w) = 0$ then $\chi_{triv}(\tau_w) = 1$  and the above composed map simply is the map $H^d(I,k) \xrightarrow{w_*} H^d(I,k)$, which is the identity by Lemma \ref{lemma:trivOmega}.

ii. The map $\trace ^d: \Ext^d_{\Mod(G)}(\mathbf{X},\mathbf{X}) \longrightarrow \Ext^d_{\Mod(G)}(\mathbf{X}, k_{triv})$ is surjective, for example, by Remark \ref{rema:tracew}. It is right $H$-equivariant because $\trace$ is right $H$-equivariant, and because  the right action of $H$ on the cohomology spaces is through the coefficients. Finally, it is left $H$-equivariant since the left actions are functorially induced by the right $H$-action on $\mathbf{X}$.
\end{proof}

\begin{remark}
Using directly the Hecke operators $\cores \circ g_* \circ \res$ in order to define the left $H$-action on $H^d(I,k_{triv})$ the Prop.\ \ref{Hdtriv}.i is proved independently in \cite{Koziol} Thm.\ 7.1 in case the root system is irreducible.
\end{remark}

The kernel  $\ker({\chi_{triv}})$ of $\chi_{triv}$ is a sub-$H$-bimodule of $E^0=H$.
Passing to duals, we have the restriction map
\begin{equation}\label{f:res1}
   (E^0)^\vee\longrightarrow \ker({\chi_{triv}})^\vee \ .
\end{equation}
Its  kernel is   a one dimensional vector space isomorphic to $\chi_{triv}$ as an $H$-bimodule.  It is generated by the linear map
\begin{equation}\label{f:upphi}
  \upphi: \tau_\omega\mapsto 1, \: \text{for $\omega\in \widetilde \Omega$ and }\tau_w\mapsto 0, \: \text{for $w\in \widetilde W$ with length $>0$}.
\end{equation}

\textbf{Assume that $\Omega$ is finite.} Then $\upphi$ lies in $(E^0)^{\vee, f}$  and we have a short exact sequence of $H$-bimodules
\begin{equation}\label{ses:0}
  0\longrightarrow k \upphi\longrightarrow (E^0)^{\vee, f}\longrightarrow \ker({\chi_{triv}})^{\vee, f} \longrightarrow  0
\end{equation}
where  $\ker(\chi_{triv})^{\vee,f}$ denotes the image of $(E^{0})^{\vee, f}\subset (E^0)^\vee$ by \eqref{f:res1}.

\begin{proposition}\label{prop:Omegafinite}
  If $\Omega$ is finite and $\vert\Omega\vert$ is invertible in $k$, then we have a decomposition of $(E^0)^{\vee,f}$ into a direct sum of $H$-bimodules
\begin{equation*}
  (E^0)^{\vee,f}\cong \chi_{triv}\oplus \ker(\chi_{triv})^{\vee,f}
\end{equation*}
where $\chi_{triv}$ is supported by the element $\upphi$ defined in \eqref{f:upphi} and  $\ker(\chi_{triv})^{\vee,f}$ is the image of $(E^0)^{\vee, f}$ by the restriction map \eqref{f:res1}.
Via the isomorphism \eqref{f:Deltad}, this corresponds to the decomposition
\begin{equation}
  E^d\cong H^d(I, k_{triv})\oplus \ker(\trace^d).
\end{equation}
\end{proposition}
\begin{proof}
Note that the image of $\upphi$ in $E^d$ is $\sum_{\omega\in\widetilde \Omega}\phi_\omega$.
Combining \eqref{ses:d}, \eqref{ses:0} and \eqref{f:Deltad}, we obtain
a diagram of morphisms of $H$-bimodules
\begin{equation*}
 \xymatrix{
0 \ar[r]&  {}^\anti (k\upphi)^\anti \ar[r] \ar@{.>}[d] ^{\partial} & ({}^\anti E^0\,^\anti)^{\vee, f} \ar[r] &({}^{\anti}\ker({\chi_{triv}})^\anti)^{\vee, f}  \ar[r]  & 0  \\
0 &\ar[l]  H^d(I, k_{triv}) &E^d \ar[u]^{\Delta^d}_{\cong}  \ar[l] ^{}&\ker(\trace^d) \ar[l] \ar@{.>}[u] ^{\varepsilon} & 0 \ar[l]  }
\end{equation*}
where $\partial$ and $\varepsilon$ are such that the diagram commutes. The map $\partial$ is a map between one dimensional vector spaces, and it is not trivial since
\begin{equation*}
  \partial(\upphi)=\trace^d(\sum_{\omega\in \widetilde \Omega}\phi_\omega)=\vert \widetilde\Omega \vert\eta^{-1}(1)=-\vert \Omega \vert \eta^{-1}(1)\neq 0 \ .
\end{equation*}
This implies that the second short exact sequence splits. By a standard argument, $\varepsilon$ is also an isomorphism and the first exact sequence also splits.
\end{proof}

Recall that we defined in \eqref{f:defifil} a decreasing filtration of $H$ as an $H$-bimodule. For the sake of homogeneity of notations, we denote it  here by $(F^nE^0)_{n\geq 0}$ and recall that
\begin{equation}\label{f:filE0}
   F^n E^0= \oplus_{\ell(w)\geq n} k\tau_w \ .
\end{equation}
For $n\geq 1$, define $(F^{n} E^0)^{\vee,f}$ to be the $H$-bimodule image of $(E^0)^{\vee,f}$ by the  surjective restriction map $(E^0)^{\vee}\rightarrow  ( F^n E^0)^\vee$. Since $\Omega$ is assumed to be finite, the kernel  of $(E^0)^{\vee, f}\twoheadrightarrow (F^{n} E^0)^{\vee,f}$ coincides with the dual space
$(E^0/F^nE^0)^{\vee}$. Furthermore $\ker(\chi_{triv})^{\vee,f}$ is the increasing union of the subspaces $( \ker({\chi_{triv}})/ F^nE^0)^\vee$ of  all linear maps $\varphi$  in $\ker(\chi_{triv})^{\vee}$  which are trivial on $F^n E^0$ for some  $n\geq 1$. The supersingular $H$-modules were defined in \S\ref{subsec:supersing}.

\begin{corollary}\label{coro:supersingEd}
   If $\mathbf G$ is semisimple simply connected with irreducible root system, then the $H$-module
   $\ker(\trace^d)$ is a union of $H$-bimodules which are (finite length) supersingular on the right and on the
   left.
\end{corollary}
\begin{proof}
It suffices to prove, for $n\geq 1$ that the finite dimensional space ${}^\anti(( \ker({\chi_{triv}})/ F^nE^0)^\vee)^\anti$ is a supersingular  $H$-module on the left and on the right. Since $\Omega=\{1\}$, we have
\begin{equation*}
\ker({\chi_{triv}})= (1-e_1) F^0E^0+ F^1E^0 \ ,
\end{equation*}
and $\ker({\chi_{triv}})/ F^nE^0$, by Lemma \ref{lemma:F1/Fm}.ii, is annihilated by the action of $\mathfrak J^n$ on the left (resp. right). Therefore,  $( \ker({\chi_{triv}})/ F^nE^0)^\vee$ is annihilated by the action of
$\mathfrak J^n$ on the right (resp. left) and therefore supersingular. Due to Remark \ref{rema:anti+ss} this remains the case after twisting by $\anti$.
\end{proof}

\phantomsection

\printendnotes


\begin{thebibliography}{DDMS}

\bibitem[Ber]{Ber} Bernstein J. (r\'edig\'e par P.\ Deligne): \emph{Le \enquote{centre} de Bernstein}. In Repr\'esentations des groupes r\'eductifs sur un corps local (eds.\ Bernstein, Deligne, Kazhdan, Vign\'eras),  1--32. Hermann 1984

\bibitem[BLR]{BLR}
Bosch S., L\"utkebohmert W., Raynaud M.: \emph{N\'eron Models}. Springer 1990

\bibitem[B-LL]{B-LL}
Bourbaki N.: \emph{Lie Groups and Lie Algebras. Chap.\ 4-6}. Springer 2002

\bibitem[Bro]{Bro}
Brown K.S.: \emph{Cohomology of Groups}. Springer 1982

\bibitem[BT1]{BT1}
Bruhat F., Tits J.:  \emph{Groupes r\'eductifs sur un corps local. I. Donn\'ees radicielles valu\'ees}.  Publ.\ Math.\ IHES {41},   5--251  (1972)

\bibitem[BT2]{BT2} Bruhat F., Tits J.:
\emph{Groupes r\'eductifs sur un corps local. II. Sch\'emas en groupes. Existence d'une donn\'ee radicielle valu\'ee}. Publ.\ Math.\ IHES  {60}, \ 5--184  (1984)


\bibitem[DDMS]{DDMS}
Dixon J.D., du Sautoy M.P.F., Mann A., Segal D.: \emph{Analytic Pro-$p$ Groups}. Cambridge Univ.\ Press 1999

\bibitem[Har]{Har}
Hartshorne R.: \emph{Residues and Duality}. Springer Lect.\ Notes Math.\ 20 (1966)

\bibitem[Hum]{Hum}
Humphreys J.E.: \emph{Reflection groups and Coxeter groups}. Cambrudge Univ.\ Press 1990
\bibitem[IM]{IM}
Iwahori N., Matsumoto H.: \emph{On some Bruhat decomposition and the structure of the Hecke rings of $p$-adic Chevalley groups}.
Publ.\ Math.\ IHES 25, 5--48 (1965)

\bibitem[Jan]{Jan}
Jantzen J.C.: \emph{Representations of Algebraic Groups}. 2nd ed. AMS 2003

\bibitem[Koz1]{Koz1}
Koziol K.: \emph{Pro-$p$-Iwahori invariants for ${\rm SL}_2$ and $L$-packets of Hecke modules}. Int. Math. Res. Not. no. 4,  1090--1125  (2016)

\bibitem[Koz2]{Koziol} Koziol K.: \emph{Hecke module structure on first and top pro-$p$-Iwahori Cohomology}. {arXiv:1708.03013}  (2017)

\bibitem[Lan]{Lan}
Lang S.: \emph{Topics in Cohomology of Groups}. Springer Lect.\ Notes Math.\ 1625 (1996)

\bibitem[Laz]{Laz}
Lazard M.: \emph{Groupes analytiques $p$-adiques}.
Publ.\ Math.\ IHES 26, 389--603 (1965)

\bibitem[Lu]{Lu}
Lusztig G.: \emph{Affine Hecke algebras and their graded version}. J.\ AMS  {2}(3), 599--635 (1989)

\bibitem[NSW]{NSW}
Neukirch J., Schmidt A., Wingberg K.: \emph{Cohomology of Number Fields}. Springer Grundlehren der math. Wissenschaften \ 323 (2000)

\bibitem[Oll1]{Ollequiv}
Ollivier R.: \emph{Le foncteur des invariants sous l'action du pro-$p$-Iwahori de $GL_2(F)$}. J.\ reine angew.\ Math.\ 635, 149--185 (2009)

\bibitem[Oll2]{Oll}
Ollivier R.: \emph{Compatibility between Satake and Bernstein  isomorphisms in characteristic $p$}. ANT 8-5, 1071--1111 (2014)

\bibitem[OS1]{OS1}
Ollivier R., Schneider P.: \emph{Pro-$p$ Iwahori Hecke algebras are Gorenstein}. J.\ Inst.\ Math.\ Jussieu 13, 753--809 (2014)

\bibitem[OS2]{OS2}
Ollivier R., Schneider P.: \emph{A canonical torsion theory for pro-$p$ Iwahori-Hecke modules}. Advances in Math.\ 327, 52--127 (2018)
\bibitem[Pas]{Pas}
Pa$\check{\textrm{s}}$k$\bar{\textrm{u}}$nas, V.:\emph{The image of Colmez's Montr\'eal functor}. Publ.\ Math.\ IHES 118, 1--191 (2013)

\bibitem[Ron]{Ron}
Ronchetti N.: \emph{Satake map for the mod p derived Hecke algebra}. Preprint 2016

\bibitem[SDGA]{SDGA}
Schneider P.: \emph{Smooth representations and Hecke modules in characteristic $p$}. Pacific J. Math. 279, 447--464 (2015)

\bibitem[SchSt]{SchSt}
Schneider P.,  Stuhler U.: \emph{Representation theory and sheaves on the Bruhat-Tits building}, Publ.\ Math.\ IHES {85}, 97--191 (1997)

\bibitem[S-CG]{S-CG} Serre J.-P.: \emph{Cohomologie Galoisienne}. Springer Lect.\ Notes
Math.\ 5, 5. \'ed. (1997)

\bibitem[S-LL]{S-LL}
Serre J.-P.: \emph{Lie Algebras and Lie Groups}. Lect.\ Notes Math.\ 1500. Springer 1992

\bibitem[Ser1]{Ser} Serre J.-P.: \emph{Sur la dimension cohomologique des
groupes profinis}. Topology 3, 413--420 (1965)

\bibitem[Se2]{Se2}
Serre J.-P.: \emph{Local Fields}. Springer 1979

\bibitem[Spr]{Spr}
Springer T.A.: \emph{Linear Algebraic Groups}. 2nd Edition. Birkh\"auser 1998

\bibitem[Tits]{Tit}
Tits, J.: \emph{Reductive groups over local fields}. In Automorphic Forms,
Representations, and $L$-Functions (eds. Borel, Casselmann). Proc.\ Symp.\ Pure Math.\  {33} (1),  29--69. American Math. Soc. 1979

\bibitem[Vig1]{Vigprop}
Vign\'eras M.-F.:
\emph{Pro-$p$-Iwahori Hecke ring and  supersingular $\overline{\mathbb F}_{p}$-representations}. Math.\ Annalen {331}, p.\ 523--556 (2005). Erratum vol.\
333 (3), 699--701.

\bibitem[Vig2]{Vig}
Vign\'eras M.-F.: \emph{Repr\'esentations $\ell$-modulaires
d'un groupe r\'eductif $p$-adique avec $\ell \neq p$}. Progress in
Math.\ 137, Birkh\"auser 1996

\bibitem[Vig3]{VigIII}
Vign\'eras M.-F.: \emph{The pro-$p$-Iwahori Hecke algebra of a reductive $p$-adic group III}. J.\ Inst.\  Math.\ Jussieu 16, 571--608 (2017)
\end{thebibliography}
\end{document}